\documentclass[a4paper]{article}
\usepackage[latin1]{inputenc} %
\usepackage[T1]{fontenc} %
\usepackage{appendix}
\usepackage{hyperref}
\usepackage{amssymb}
\usepackage{a4wide}
\usepackage{t1enc}
\usepackage{graphics}
\usepackage{graphicx}
\usepackage{epsfig}
\usepackage{verbatim}
\usepackage{color}
\usepackage{times}
\usepackage{amsthm}
\usepackage{bbm}
\usepackage{amsmath}
\usepackage{subfigure}
\usepackage{pdfsync}
\DeclareMathAlphabet{\mathpzc}{OT1}{pzc}{m}{it}
\def\R{\mathbb{R}}
\def\C{\mathbb{C}}
\def\N{\mathbb{N}}
\def\Z{\mathbb{Z}}

\def\dsp{\displaystyle}

\def\u{\mathpzc{u}}

\newcommand{\md}{{\mathrm{ d}}}

\newcommand{\me}{{\mathrm{ e}}}
\newcommand{\bbA}{\mathbb{A}}
\newcommand{\bbP}{\mathbb{P}}
\newcommand{\bbM}{\mathbb{M}}
\newcommand{\bbF}{\mathbb{F}}
\newcommand{\bbS}{\mathbb{S}}
\newcommand{\bbV}{\mathbb{V}}
\newcommand{\bbT}{\mathbb{T}}
\newtheorem{hypo}{Hypoth\`ese}[section]
\newtheorem{defi}[hypo]{Definition}

\newtheorem{theo}[hypo]{Theorem}
\newtheorem{lema}[hypo]{Lemma}
\newtheorem{coro}[hypo]{Corollary}
\newtheorem{prop}[hypo]{Proposition}
\newtheorem{rem}[hypo]{Remark}
\begin{document}

\title{Mathematical models for dispersive electromagnetic waves: an overview}
\author{Maxence Cassier$^{a}$, Patrick Joly$^{b}$ and Maryna Kachanovska$^{b}$ \\  \\
{\footnotesize $^a$ Department of Mathematics of the University of Utah, Salt Lake City, UT, 84112, United States}\\ 
{\footnotesize $^b$ ENSTA / POEMS$^1$, 32 Boulevard Victor, 75015 Paris, France}\\ 
{\footnotesize (cassier@math.utah.edu, patrick.joly@inria.fr, maryna.kachanovska@ensta.fr)}}
\footnotetext[1]{POEMS (Propagation d'Ondes: Etude Math\'ematique et Simulation) is a mixed research team (UMR 7231) between CNRS (Centre National de la Recherche Scientifique), ENSTA ParisTech (Ecole Nationale Sup\'erieure de Techniques Avanc\'ees) and INRIA (Institut National de Recherche en Informatique et en Automatique).}
\footnotetext[2]{The first two authors gratefully acknowledge the funding by a pubic grant as a part of the ANR Project `Metamath', reference ANR-11-MONU-0016. The third author gratefully acknowledges the support  by a public grant as part of the Investissement d'avenir project, reference ANR-11-LABX-0056-LMH,	LabEx LMH.}

\date{}

%%\ead{cassier@math.utah.edu)}
%%\author[poems]{Patrick Joly}
%%\ead{patrick.joly@inria.fr}
%\author[poems]{Maryna Kachanovska \footnote{The first two authors gratefully acknowledge the funding by a pubic grant as a part of the ANR Project `Metamath', reference ANR-11-MONU-0016. The third author gratefully acknowledges the support  by a public grant as part of the Investissement d'avenir project, reference ANR-11-LABX-0056-LMH, LabEx LMH. 
%}}
%\ead{maryna.kachanovska@ensta.fr}

	\maketitle
\begin{abstract}	
In this work, we investigate mathematical models for electromagnetic wave
propagation in dispersive isotropic media. We emphasize the link between
physical requirements and mathematical properties of the models. A
particular attention is devoted to the notion of non-dissipativity and
passivity. We consider successively the case of so-called local media and
general passive media. The models are studied through
energy techniques, spectral theory and dispersion analysis of plane waves.
For making the article self-contained, we provide in appendix some useful
mathematical background.
\end{abstract}
	

\section{Introduction, motivation} \label{model}
The theory of wave propagation in dispersive media, and more specifically negative index materials in electromagnetism, 
had known recently a regain of interest with the appearance of electromagnetic metamaterials. Their  theoretical behaviour had been, much before their experimental realization, predicted in the pioneering article of Veselago \cite{veselago}. Since 
the beginning of the century, several works \cite{smith2004metamaterials}, \cite{cui2010metamaterials}, \cite{brien2002photonic} have shown a practical realisability of metamaterials, with the help of a periodic assembly of small resonators whose effective macroscopic behaviour corresponds to
a negative index (acoustic metamaterials with similar effects can also be produced \cite{cummer_acoustic}). Their existence opened new perspectives of application for physicists, in particular in optics and photonic crystals,  related to new physical phenomena such as backward propagating waves, negative refraction \cite{veselago} or plasmonic surface waves \cite{Mai-07} which are used for creating perfect lenses \cite{pendry2000negative}, in superlensing \cite{Mil-05} or cloaking \cite{Mil-06}. On the other hand the study of the corresponding 
mathematical models raised new exciting  questions for mathematicians (see \cite{JLiReview} for a recent review), in particular numerical analysts \cite{JLiBook}, \cite{ziolkowski2001wave}, \cite{JLiLorentz}.\\[12pt]
Writing this paper has been decided at a EPSRC Workshop held in Durham on Mathematical and Computational Aspects of Maxwell's Equations in July 2016, where the first two authors gave oral presentations about the mathematics of metamaterials, one of 
the main topics of the Workshop. During the past three years, the authors have been working, in collaboration or independently, 
on wave propagation problems involving dispersive electromagnetic materials, and, more specifically, negative index materials. For instance, in \cite{cas-14-thes}, \cite{cas-haz-Jol-16}, \cite{cas-haz-Jol-vin-17}, we studied a transmission problem 
between a negative index and the vacuum and more especially the large time behaviour of the solution of the evolution problem with a time harmonic source. In \cite{ValEliPatr}, \cite{PMLCANUM}, \cite{PMLPLASMAS} we addressed the question of the construction and analysis of stable Perfectly Matched Layers (PML's) for dispersive Maxwell's equations, for the time domain numerical simulation purpose. Finally, in \cite{cas-mil-16}, we address the question of broadband passive cloaking, in other 
words, whether it possible to construct an electromagnetic passive cloak that cloaks an object 
over a whole frequency band. We answer negatively to this question in the so-called quasistatic
regime and provide quantitative limitations on the cloaking effect 
over a finite frequency range.\\[12pt]
When working on this subject we have encountered two main difficulties. The first one is the absence of 
a work that would provide a unified, rigorous presentation and analysis of the existing mathematical models,
despite the fact that many related publications can be found in a broad range of fields
including applied and theoretical physics \cite{Jack-75}, \cite{landau}, electric circuit theory and pure 
and applied mathematics  \cite{Akhiezer}, \cite{Kato}, \cite{Don-00}. The second difficulty lies in the fact that, because of the abundance of the specialized literature it is not clear which statements are proven and which are simply commonly admitted. Thus, 
in the present work, we would like to partially fill these gaps.
\\[12pt]
Properly speaking, this article is not a research paper. It has to be considered more as a review paper in which we try 
to gather the results from the literature that we found the most useful for applied mathematicians,  provide 
an original presentation of these results and propose some new ideas (which, to our knowledge, have not occurred in the existing literature). We tried to keep the presentation rigorous, even though sometimes, for the sake of readability, we sacrificed formalism. Most proofs are detailed and only use elementary tools (and those that do not are postponed to appendix). In this way, the article is self-contained and accessible to readers (physicists, engineers) who
are not mathematicians. We hope that it can be seen as a useful toolbox for any scientist starting to study the subject, 
especially for applied mathematicians and numerical analysts. We are happy to dedicate this work to Peter Monk, who has been a major contributor of the numerical analysis of Maxwell's equations \cite{Monk}, on the occasion of his 60$^{th}$ birthday.\\[12pt]
We conclude this introduction by a brief outline of the rest of the paper. In Section \ref{sec-Models} we formulate properties of the electric permittivity and magnetic permeability, studying them from mathematics and physics based points of view. In particular, we concentrate on the mathematical description of the so-called passivity property (Section \ref{sec-passive}) and discuss the relationship between its physical and mathematical interpretation in Section \ref{sec-equivalence}. 
In Section \ref{sec-local}, we address the case when the permittivity and permeability are rational fractions (or 'local' materials, the name will be explained later). In the time domain they give rise to the Maxwell's equations coupled with ODEs. The results of this section include:
the mathematical characterization of local passive materials (Section \ref{sec-defi}),
the equivalence of passivity and well-posedness for a class of models (Section \ref{sec-hom}),
a characterization of forward and backward propagating waves based on the  analysis of the dispersion relation (Section \ref{sec-dispersionanalysis}).
Finally, Section \ref{Generalpassive} is dedicated to the extension of the analysis of Section \ref{sec-local} to general passive media.
\section{Mathematical models for dispersive electromagnetic waves} \label{sec-Models}
\subsection{Maxwell's equations in dispersive media:  introduction} \label{sec-presentation}
Maxwell's equations relate the space variations of the electric and magnetic fields
${\bf E}({\bf x},t)$ and ${\bf H}({\bf x},t)$ (where ${\bf x} \in \R^3$ denotes the space variable and $t>0$ is the time) to the time variations of the corresponding electric and magnetic inductions ${\bf D}({\bf x},t)$ and ${\bf B}({\bf x},t)$:
\begin{equation} \label{Maxwell}
	\partial_t {\bf B} + {\bf rot} \, {\bf E} = 0, \quad \partial_t {\bf D} - {\bf rot} \, {\bf H} = 0, \quad {\bf x} \in \R^3, \quad t >0.
	\end{equation}
These equations need to be completed by so-called {\bf constitutive laws} that characterize the material in which electromagnetic waves propagate by relating the electric (or magnetic) field and the corresponding induction. In this paper, we shall restrict ourselves to materials which are {\bf local in space} (i.e. the induction at a given point only depends on the corresponding field) and {\bf linear} (this dependence is linear).\\[12pt]
\noindent In standard {\bf dielectric media}, it is common to assume that the relationship is also local in time (typically the electric induction ${\bf D}$ at a given point only depends on the magnetic field ${\bf E}$). If, moreover, one assumes that the medium is {\bf isotropic} (roughly speaking, the relationship between ${\bf D}$ and ${\bf E}$ does not see the orientation of the fields), it is natural to suggest that the fields are proportional
\begin{equation} \label{Dielectric}
{\bf D}({\bf x},t) = \varepsilon({\bf x}) \, {\bf E}({\bf x},t) , \quad  {\bf B}({\bf x},t) = \mu({\bf x}) \, {\bf H}({\bf x},t),
	\end{equation}
where at any point ${\bf x}$,  $\varepsilon({\bf x}) $ and $\mu({\bf x})$ are positive numbers called respectively the electric permittivity and the magnetic permeability of the material at a point ${\bf x}$. The fact that they may depend of ${\bf x}$ characterizes the possible {\bf heterogeneity} of the material. In the vacuum, these coefficients are of course independent of ${\bf x}$:
$
 \varepsilon({\bf x}) = \varepsilon_0 \approx (36\pi)^{-1} \, 10^{-9} Fm^{-1}, \mu({\bf x}) = \mu_0 = 4\pi \, 10^{-7} Hm^{-1}.
$
In the matter, the law (\ref{Dielectric}) cannot be true and can be seen only as an approximation (often accurate). It appears that simple proportionality laws can be valid only in the vacuum, otherwise this would violate some physical principles (\cite{landau}).
In order to be consistent with such physical principles, one needs to abandon the idea that the constitutive laws are local in time and to accept e.g. that ${\bf D}({\bf x},t)$ depends
on the history of the values of ${\bf E}$ between $0$ and $t$, i. e. \begin{equation} \label{causal_law}{\bf D}({\bf x},t) = F\big(x,t \, ; \big\{{\bf E}({\bf x},\tau), 0 \leq \tau \leq t \big\} \big).\end{equation}
The above obeys a fundamental physical principle: the causality principle. Adding the \textbf{time invariance principle}, i.e. that the material behaves the same way whatever the time one observes it, one infers that the function $F$ is also independent of time:
$F(x,t \, ;\cdot) = F(x \, ;\cdot)$.\\[12pt]
To translate the above in more mathematical terms, it is useful to go to the {\bf frequency domain}. Let us remind the definition of the {\bf Fourier-Laplace} transform and some of its properties. \\[12pt]
Let $u(t)$ be a (measurable) complex-valued, locally bounded and causal ($u(t) = 0 \mbox{ for } t < 0 \}$) function of time, which
we suppose to be {\bf exponentially bounded} for large $t$ (for simplicity). More precisely, given $\alpha \geq 0$ we introduce the class of functions which we shall denote in the following as $u \in PB_\alpha (\R^+)$ with
\begin{equation} \label{defPB}
PB_\alpha (\R^+) = \{ u(t) : \R^+ \rightarrow \C \; / \; \exists \; (C, p) \in \R^+ \times \N \mbox{ such that } \u(t) \leq C \, e^{\alpha t} \, (1 + t^p) \}.
	\end{equation} 
For $\alpha = 0$, one recovers the class $PB(\R^+) \equiv PB_0 (\R^+)$ of {\bf polynomially bounded} functions. The Fourier-Laplace transform $\widehat{u}(\omega)$ of $u$ is the function defined in the complex half space (see e.g. \cite{DautrayLions})
\begin{equation} \label{defC+}
	\C_\alpha^+ = \{ \omega \in \C \;  / \; {\cal I}m \, \omega > \alpha \}, \quad (\mbox{where } \C_0^+ \mbox{ will be denoted by } \C^+ \mbox{ when }  \alpha = 0)
	\end{equation}
by the following integral formula (we use here the convention which is usual for physicists)
\begin{equation} \label{defFLT}
	\forall \; \omega \in \C^+, \quad  \widehat{u}(\omega) = \frac{1}{\sqrt{2\pi}} \int_0^{+\infty}
	u(t) \; e^{i \omega t} \; dt.
\end{equation}
Note that, with this convention, as soon as $u$ and $\partial_t$ belong to $PB_\alpha (\R^+)$, we have
\begin{equation} \label{propFLT}
	\forall \; \omega \in \C^+, \quad  \widehat{\partial_t u}(\omega) = - i \omega \, \widehat{u}(\omega) + u(0), \quad \forall \; \omega \in \C_\alpha^+,
\end{equation}
which reduces to $\widehat{\partial_t u}(\omega) = - i \omega \, \widehat{u}(\omega)$ when $u(0)=0$.\\[12pt]
This transform is related to the usual Fourier transform $u(t) \rightarrow {\cal F}u(\omega)$ (where $t$ and $\omega$ are here real) by
\begin{equation} \label{Link_F-FLT}
\forall \; \eta > \alpha, \quad \forall \; \omega \in \R, \quad \widehat{u}(\omega + i\eta) = {\cal F}\big(u \, e^{-\eta t})(\omega)
\end{equation}
which proves in particular that (this is Plancherel's theorem)
\begin{equation} \label{FLT_L2}
	\forall \; \eta > \alpha, \quad \omega \in \R \mapsto \widehat{u}(\omega + i\eta) \in L^2(\R) \quad \mbox{and} \quad \int_{-\infty}^{+\infty} |u(\omega + i\eta)|^2 \; d\omega = \int_0^{+\infty} |u(t)|^2 \, e^{-2\eta t} \; dt.
\end{equation}
On the other hand, one easily sees that 
\begin{equation} \label{Analyticity}
\forall u \in PB_\alpha(\R^+),  \quad \omega \mapsto \widehat{u}(\omega) \mbox{ is analytic in } \C_\alpha^+.
\end{equation}
One can expect that $\widehat{u}(\omega)$ can be extended as an analytic function in a domain of the complex plane that contains the half-space $\C_\alpha^+$. We shall use the same notation $\widehat{u}(\omega)$ for the function defined by (\ref{defFLT}) and its analytic extension. In the following $\omega$ will be referred to as the (possibly complex) {\bf frequency}.\\[12pt]
The half-plane $\C^+$ in invariant under the transformation $\omega \rightarrow - \, \overline{\omega}$, which corresponds to the symmetry with respect to the imaginary axis. Laplace-Fourier transforms of {\bf real-valued} functions have a particular property with respect 
to this transformation:
\begin{equation} \label{FLTrealfields}
u(t) \in \R, \quad \forall \; t\geq 0, \quad \Longleftrightarrow \quad	\forall \; \omega \in \C^+, \quad  \widehat{u}(- \, \overline{\omega}) = \overline{\widehat{u}(\omega)}
\end{equation}
In the sequel, we shall assume that all the functions of time that are used in this article (for instance, one of the components of the electric and magnetic field at a given point), belong to some $PB_\alpha(\R+)$.\\[12pt]
Dispersive (isotropic) electromagnetic materials are most often defined as materials in which the proportionality laws of the form (\ref{Dielectric}) hold true \textbf{in the frequency domain}. Namely, they are satisfied by the Laplace-Fourier transforms of the fields, rather than by the fields themselves. In this case there is no reason to require that $\varepsilon$ and $\mu$ are real and independent of the frequency. That is why a dispersive isotropic medium will be defined as obeying constitutive laws of the form
\begin{equation} \label{Dispersive}
\widehat{\bf D}({\bf x}, \omega) = \varepsilon({\bf x}, \omega) \, \widehat{\bf E}({\bf x},\omega) , \quad  \widehat{\bf B}({\bf x},\omega) = \mu({\bf x}, \omega) \, \widehat{\bf H}({\bf x},\omega).
\end{equation}
where for each ${\bf x}$, $\omega \in \C^+ \mapsto \varepsilon({\bf x}, \omega)$ (the {\bf permittivity}) and $\omega \in \C^+ \mapsto \mu({\bf x}, \omega)$ (the {\bf permeability}) are non-trivial functions of the frequency that describe the {\bf dispersivity} of the medium. For {\bf non-dispersive} materials these functions are real positive and constant, i.e. (\ref{Dielectric}) holds. Of course, these functions satisfy some particular properties imposed by physical or mathematical reasons, as we show later.
\begin{rem} \label{rem_dielectric}
Non-dispersive constitutive laws like (\ref{Dielectric}) are commonly used in many applications, as presented in e.g. \cite{Monk}. Even though they cannot be rigorously true for physical reasons, they can be considered as a very good approximation as soon as $\varepsilon, \; \mu$ are real and constant over a broad range of frequencies and one excites the medium with a temporal source whose frequency content, or spectrum, is "mainly contained" in this range of frequencies. In such a case, the medium behaves as a non-dispersive one.
\end{rem}
\noindent {\bf Causality principle.} To ensure the causality of ${\bf D}(x,t)$ (or ${\bf B}(x,t)$) provided that ${\bf E}(x,t)$ (or ${\bf H}(x,t)$) is causal, it is natural to impose 
$$
{\bf (CP)} \qquad \omega \mapsto \varepsilon({\bf x}, \omega) \quad \mbox{and} \quad \omega \mapsto \mu({\bf x}, \omega) \quad \mbox{are analytic in } \C_\alpha^+, \; \mbox{for some } \alpha \geq 0.
$$
{\bf Reality principle.} A second requirement is that if ${\bf D}(x,t)$ (or ${\bf B}(x,t)$)
is real then ${\bf E}(x,t)$ (or ${\bf H}(x,t)$) is real too. According to (\ref{Dispersive}) and (\ref{FLTrealfields})
$$
{\bf (RP)} \qquad \forall \; \omega \in \C^+, \quad  {\varepsilon}({\bf x}, - \, \overline{\omega}) = \overline{{\varepsilon}({\bf x},\omega)}, \quad  {\mu}({\bf x}, - \, \overline{\omega}) = \overline{{\mu}({\bf x},\omega)}.
$$
{\bf High frequency principle.} A fundamental property from the physical point of view is that, at high frequency, any material "behaves as the vacuum". Mathematically, this amounts to requiring that
$$
{\bf (HF)} \qquad \forall \; \eta > 0, \mbox{ if } \; {\cal I}m \, \omega \geq \eta > 0, \quad \lim_{|\omega| \rightarrow + \infty} \varepsilon({\bf x}, \omega) =\varepsilon_0, \quad 
\lim_{|\omega| \rightarrow + \infty} \mu({\bf x}, \omega) =\mu_0.
$$
This means that the material is "{\bf less and less dispersive}" at high frequencies. In fact, the only non-dispersive medium is the vacuum (see however remark \ref{rem_dielectric}). This condition is not only a physical requirement: it also plays a role in the {\bf well-posedness} of  Maxwell's equations in {\bf local media} (see remark \ref{HFcondition}).\\[12pt]
From the mathematical point of view, assuming that the causality principle ${\bf(CP})$ is satisfied,  ${\bf (HF)}$ implies that fields related by one of the constitutive laws (\ref{Dispersive}) have the same {\bf time regularity}, more precisely,
$$
t \rightarrow {\bf E}({\bf x},t) \in H^s_{loc}(\R^+), s \geq 0 \quad \Longrightarrow \quad t \rightarrow {\bf D}({\bf x},t) \in H^s_{loc}(\R^+).
$$
Indeed, according to (\ref{Link_F-FLT}), for $\eta > 0$,  the operator $\big( t \mapsto e^{-\eta t} \, {\bf E}({\bf x},t)\big) \rightarrow \big( t \mapsto e^{-\eta t} \,{\bf D}({\bf x},t)\big)$ corresponds in the Fourier domain to the multiplication by $\varepsilon({\bf x}, \omega + i\eta), \;\omega \in \R$. From the analyticity property ${\bf (CP)}$, we infer
that $ \omega \in \R \rightarrow \varepsilon({\bf x}, \omega + i\eta)$ is a continuous function which has, because of ${\bf (HF)}$, a finite limit at infinity. Therefore this function is bounded and it is easy to conclude. 
\\[12pt]
A particular example of a material satisfying ${\bf (CP)}$ (with $\alpha = 0$), ${\bf (RP)}$ and ${\bf (HF)}$ is the case where there exists, for any ${\bf x} \in \R^3$, two causal real functions $t \mapsto \chi_e({\bf x}, t)$ and $t \mapsto \chi_m({\bf x},t)$ in $L^1(\R^+)$ such that :
\begin{equation} \label{exampleL1}
\varepsilon({\bf x}, \omega) = \varepsilon_0 \, \big(1 + \widehat{\chi}_e({\bf x}, \omega)\big), \quad \mu({\bf x}, \omega) = \mu_0 \, \big(1 + \widehat{\chi}_m({\bf x}, \omega)\big),
\end{equation}
where, by Riemann-Lebesgue's theorem, $\widehat{\chi}_e({\bf x}, \omega)$ and $\widehat{\chi}_m({\bf x}, \omega)$ extend to the closed half-space ${\overline{\C^+}}$ to a continuous function that tends to $0$ when $|\omega| \rightarrow + \infty$.
In this case, using the properties of the Fourier-Laplace transform with respect to convolution, the constitutive laws are expressed as follows:
\begin{equation} \label{Dispersive_time}
	\begin{array}{ll}
\dsp {\bf D}({\bf x}, t) = \varepsilon_0 \, \Big( {\bf E}({\bf x}, t) + \int_0^t \chi_e({\bf x}, \tau) \; {\bf E}({\bf x}, t-\tau) \; d \tau \Big),\\[12pt]
\dsp {\bf B}({\bf x}, t) = \mu_0 \, \Big( {\bf H}({\bf x}, t) + \int_0^t \chi_m({\bf x}, \tau) \; {\bf H}({\bf x}, t-\tau) \; d \tau \Big).
\end{array}
\end{equation}
\subsection{Passive materials} \label{sec-passive}
In this section, the case $\alpha = 0$ plays a particular role, since here we are interested in situations where the electric and magnetic fields and corresponding inductions are polynomially bounded in time. In such a medium, dispersive
 Maxwell's equations are {\bf stable} in the sense that there exists no mechanism of exponential blow-up (think for instance of the Cauchy problem). Thus,
  according to ${\bf (CP)}$, $\omega \mapsto \varepsilon({\bf x}, \omega)$ and $\omega \mapsto \mu({\bf x}, \omega)$ are analytic in $\C^+$. A particular subclass of materials
   satisfying this property are {\bf passive} materials. Their mathematical definition requires the introduction of the notion of {\bf Herglotz} function.
\begin{defi} {\bf (Herglotz function)} A Herglotz function is a complex-valued function  $f(\omega) : \C^+ \rightarrow \C$, analytic in $\C^+$ and whose image is included in the closure of $\C^+$, i.e. 
\begin{equation} \label{propHerglotz}
	{\cal I}m \, \omega > 0 \quad \Longrightarrow \quad {\cal I}m \, f(\omega) \geq 0 
	\end{equation}
	\end{defi}
\noindent Let us formulate and prove some of their elementary properties that 
will be of use later.
\begin{lema} \label{lemHerglotz} Let $f$ be a non-constant Herglotz function. Then the following properties hold:
	\begin{itemize}
\item[(i)] ${\cal I}m \, \omega > 0 \quad \Longrightarrow \quad {\cal I}m \, f(\omega) > 0 ,$	
\item[(ii)] $g(\omega) = - \, f(\omega)^{-1}$ is a Herglotz function, too.
\end{itemize}
Moreover, assuming that $f$ extends meromorphically to a neighborhood of $\omega_0\in \R$, 
\begin{itemize}
\item[(iii)] Any real zero $\omega_0$ of $f(\omega)$ is simple and $f'(\omega_0)$ is real and positive, 
\item[(iv)] Any real pole of $f(\omega)$ is simple and the corresponding residue ${\cal R}es (f, \omega_0)$ is negative.
				\end{itemize}
\end{lema}
\begin{proof} (i) Let $\omega_0 \in \C^+$ be such that $f(\omega_0) \in \R$. Since $f$ is analytic and non-constant,  there exists $n \in \N^*$ and $a  = r \, e^{i \phi} \neq 0$
such that $f(\omega) - f(\omega_0) \sim a \, (\omega - \omega_0)^n,$ when $\omega \rightarrow \omega_0$. Take $\omega = \omega_0 + 
 \rho \; e^{i \theta}$, with $\theta \in [0,2\pi]$ and $ 0 < \rho < {\cal I}m \; \omega_0$ so that $\omega \in \C^+$. Then $f(\omega) - f(\omega_0) 
 \sim  r \,  \rho^n \; e^{i (n\theta + \phi)}$ when  $\rho \rightarrow 0$ implies
	$$
{\cal I}m \, f(\omega) =  r\, \rho^n \; \sin (n\theta + \phi)+ O(\rho^{n+1}), \quad \rho \rightarrow 0.
	$$
Since $n\theta + \phi$ describes $[\phi, \phi + 2n\pi]$, ${\cal I}m \, f(\omega)$ would take negative values for $\rho$ small enough which contradicts the fact that $f$ is a Herglotz function. \\[12pt]
(ii) By (i), $g(\omega)=-f(\omega)^{-1}$ is well-defined and $\dsp g(\omega) = - {\overline{f(\omega)}} \, {|f(\omega)|^{-2}} $ shows that $g$ is Herglotz too.\\[12pt]
(iii) If $\omega_0$ is a real zero of $f$ of multiplicity $n \geq 1$, then $f(\omega) - f(\omega_0) \sim a \, (\omega - \omega_0)^n $ when $\omega \rightarrow \omega_0$ with $a  = r \; e^{i \theta} \neq 0$. Let $\omega = \omega_0 + \rho \; e^{i \theta}$ with $\rho>0$ and $\theta \in \, ]0,\pi[$, so that $\omega \in \C^+$ again. We have
$$
{\cal I}m \, f(\omega) =  r\,\rho^n \; \sin (n\theta + \phi)+ O(\rho^2).
$$
Since $(n\theta + \phi)$ describes $]\phi, \phi + n\pi[$, if $n \geq 2$, ${\cal I}m \, f(\omega)$ would take negative values which contradicts the Herglotz nature of $f$. For $n=1$, $\theta + \phi$ describes $]\phi, \phi + \pi[$. If $\phi$ belonged to $]0, \pi[$
the intersection $]\phi, \phi + \pi[ \, \cap \, ]\pi,  2\pi[$ would be non-empty and again, ${\cal I}m \, f(\omega)$ would take negative values for $\rho$ small enough.\\[12pt]
(iv) If $\omega_0$ is a real pole of $f$, it is a zero of $g = - \, f^{-1}$. To conclude, it suffices to combine (ii) and (iii) and the fact that
$g'(\omega_0) = - {\cal R}es (f, \omega_0)^{-1}$.
\end{proof}
\noindent Then, {\bf passive} materials are defined \textbf{mathematically} as follows.
\begin{defi} {\bf (Passive material)} \label{defpassive} A dispersive electromagnetic material as defined in section \ref{sec-presentation} is said to be passive if and only if, for each ${\bf x} \in \R^3$
\begin{equation} \label{defPassive}
\omega \mapsto \omega \; \varepsilon({\bf x}, \omega) \quad \mbox{and} \quad \omega \mapsto \omega \; \mu({\bf x}, \omega) \quad \mbox{are Herglotz functions}.
	\end{equation}
	\end{defi}
\noindent From lemma \ref{lemHerglotz} (i) and (ii), one sees that, for a passive material, the relationships $\big( t \mapsto  {\bf E}({\bf x},t)\big) \rightarrow \big( t \mapsto {\bf D}({\bf x},t)\big)$ and $\big( t \mapsto  {\bf H}({\bf x},t)\big) \rightarrow \big( t \mapsto {\bf B}({\bf x},t)\big)$ can be inverted. The mathematical definition of passivity is related to a \textbf{physical} notion of passivity, which is linked to energy.
\begin{defi} {\bf (Physical passivity)} 
\noindent  Defining the electromagnetic energy as in the vacuum, i.e.
\begin{equation} \label{defEnergy}
{\cal E}(t) :=  \frac{1}{2}\int_{\R^3} \big( \varepsilon_0 \, |{\bf E}|^2  + \mu_0 \, |{\bf H}|^2 \, \big)({\bf x}, t)\; d{\bf x}, \quad t >0,
\end{equation}
we shall say that a material is {\bf physically passive} if, when ${\bf E}$, ${\bf H}$, ${\bf D}$ and ${\bf B}$ are causal fields solving (\ref{Maxwell}) in the absence of source terms (however, with non-vanishing initial conditions) for $t\geq 0$ and are related by (\ref{Dispersive}), the corresponding electromagnetic energy does not increase between $0$ and $T$ for any $T \geq 0$, namely,
\begin{equation} \label{propEnergy}
{\cal E}(T) \leq {\cal E}(0).
\end{equation}
\end{defi}
\begin{rem} \label{decroissance} The property (\ref{propEnergy}) does not imply that ${\cal E}(t)$ is a decreasing function of time. Indeed, since (\ref{propEnergy}) is supposed to hold only for causal fields, the "initial time" $t=0$ cannot be replaced by any other "initial time" $t_0$. This will be made more precise in section \ref{sec_EnergyLorentz}.
\end{rem}
\noindent For further investigation of (\ref{propEnergy}), let us define the electric polarization ${\bf P}$ and the magnetization ${\bf M}$:
\begin{equation} \label{defPM}
{\bf D}({\bf x}, t) = \varepsilon_0 \, {\bf E}({\bf x}, t) + {\bf P}({\bf x}, t), \quad {\bf B}({\bf x}, t) = \mu_0 \,  {\bf H}({\bf x}, t)  + \,{\bf M}({\bf x}, t).	
		\end{equation}
Notice that, in physics, one defines the magnetization $\mathbf{M}$ by  
	${\bf B}= \mu_0 \, \left({\bf H} + \,{\bf M}\right)$. Then, Maxwell's equations can be rewritten as
\begin{equation} \label{Maxwell2}
	\varepsilon_0 \, \partial_t {\bf E} + {\bf rot} \, {\bf H} +  \partial_t {\bf P}= 0, \quad  \mu_0 \, \partial_t {\bf H} - {\bf rot} \, {\bf E}  + \partial_t {\bf M} = 0, \quad {\bf x} \in \R^3, t >0.
	\end{equation}
Defining, like in (\ref{exampleL1}),
\begin{equation} \label{epschi}
\varepsilon({\bf x}, \omega) =  \varepsilon_0 \, \big( 1 + \widehat{\chi}_e({\bf x}, \omega)\big),  \quad  \varepsilon({\bf x}, \omega)  =  \mu_0 \, \big( 1 + \widehat{\chi}_m({\bf x}, \omega)\big),
\end{equation}
the constitutive laws (\ref{Dispersive}) can  be rewritten as follows:
\begin{equation} \label{Dispersive2}
\widehat{\bf P}({\bf x}, \omega) =  \varepsilon_0 \, \widehat{\chi}_e({\bf x}, \omega) \, \widehat{\bf E}({\bf x},\omega) , \quad  \widehat{\bf M}({\bf x},\omega) =  \mu_0 \, \widehat{\chi}_m({\bf x}, \omega) \, \widehat{\bf H}({\bf x},\omega).
\end{equation}
One easily deduces from (\ref{Maxwell2}) that
 \begin{equation} \label{idEnergy}
\dsp \frac{d}{dt} \, {\cal E}(t) + \int_{\R^3} \big(  \partial_t {\bf P} \cdot {\bf E} +  \partial_t {\bf M} \cdot {\bf H} \big)({\bf x}, t) \,  d{\bf x}  = 0.
\end{equation}
Thus, for any $T > 0$,
\begin{equation} \label{varEnergy}
{\cal E}(T) - {\cal E}(0) + \int_{\R^3} \Big[ \int_0^T \big(  \partial_t {\bf P} \cdot {\bf E} +  \partial_t {\bf M} \cdot {\bf H} \big)({\bf x}, t) \, dt \Big] \; d{\bf x}  = 0.
\end{equation}
\begin{theo} \label{Thmpassivity}
A passive material in the sense of definition (\ref{defpassive}) is physically passive. 
\end{theo}
\begin{proof} 
Let ${\bf E}_T({\bf x}, t) := \mathbbm{1}_T(t) \; {\bf E}({\bf x}, t)$, where $\mathbbm{1}_T(t)$ is the indicator function of the interval $[0,T]$, and ${\bf P}_T ({\bf x}, t)$ the corresponding induction field via (\ref{Dispersive2}), i. e.
 \begin{equation}\label{relation} \widehat{{\bf P}_T}({\bf x}, \omega) =  \varepsilon_0 \, \
 \widehat{\chi}_e({\bf x}, \omega) \, \widehat{{\bf E}_T}({\bf x},\omega).
 \end{equation}
By causality, ${\bf P}_T ({\bf x}, t) = {\bf P} ({\bf x}, t)$ for any $t \leq T$. Let $\eta > 0$. Then
$$
\int_0^T  \partial_t {\bf P} \cdot {\bf E} \; e^{-2 \eta t} \; dt \equiv \int_0^{+\infty} \partial_t {\bf P}_T \cdot {\bf E}_T \, e^{-2 \eta t} \; dt = - \, \int_{-\infty+i\eta}^{+\infty+ i\eta} \; 
i \omega \, \widehat{{\bf P}_T}(\omega ) \cdot \overline{\widehat{{\bf E}_T}(\omega ) } \; d\omega
$$
where we used (\ref{propFLT}) and Plancherel's theorem. Thus, using (\ref{relation}) and $\varepsilon_0 \, \widehat{\chi}_e({\bf x}, \omega )= \varepsilon({\bf x}, \omega ) - \varepsilon_0$, see \ref{epschi},
$$
	\begin{array}{ll}
\dsp  \int_0^T  \partial_t {\bf P} \cdot {\bf E} \; e^{-2 \eta t} \; dt & = \dsp - \int_{-\infty+i\eta}^{+\infty+ i\eta} \; 
i\omega \, \varepsilon_0 \, \widehat{\chi}_e({\bf x}, \omega ) \, |\widehat{{\bf E}_T}(\omega)|^2 \; d\omega,\\[12pt]
&\dsp = - \int_{-\infty+i\eta}^{+\infty+ i\eta} \; 
i\omega \, \varepsilon({\bf x}, \omega ) \, |\widehat{\bf E}_T(\omega)|^2 \; d\omega
+ \int_{-\infty+i\eta}^{+\infty+ i\eta} \; 
i\omega \, \varepsilon_0 \, |\widehat{\bf E}_T(\omega)|^2 \; d\omega.
\end{array} 
$$
Since ${\bf P}$ and ${\bf E}$ are real,  taking the real part of the above and using $ - {\cal R}e \, (iz) = {\cal I}m \, z$, we get
\begin{equation}\label{int1}
 \int_0^T  \partial_t {\bf P} \cdot {\bf E} \, e^{-2 \eta t} \; dt =
\int_{-\infty+i\eta}^{+\infty+ i\eta} \; 
{\cal I}m \big(\omega \, \varepsilon({\bf x}, \omega )\big) \, |\widehat{\bf E}_T(\omega)|^2 \; d\omega	- \eta \, \varepsilon_0 \, \int_{-\infty+i\eta}^{+\infty+ i\eta} \;|\widehat{\bf E}_T(\omega)|^2 \; d\omega,
\end{equation}	
or, equivalently, 
\begin{equation}\label{int1bis}
\int_0^T  \partial_t {\bf P} \cdot {\bf E} \, e^{-2 \eta t} \; dt =
\int_{-\infty+i\eta}^{+\infty+ i\eta} \; 
{\cal I}m \big(\omega \, \varepsilon({\bf x}, \omega )\big) \, |\widehat{\bf E}_T(\omega)|^2 \; d\omega	- \eta \, \varepsilon_0 \, \int_0^T  |{\bf E}|^ 2 \, e^{-2 \eta t} \; dt.
\end{equation}	
Since by passivity  ${\cal I}m \big(\omega \, \varepsilon({\bf x}, \omega )\big) > 0$ for ${\cal I}m \,\omega = \eta >0$, we have 
\begin{equation}\label{int2}
\int_0^T  \partial_t {\bf P} \cdot {\bf E} \, e^{-2 \eta t} \; dt \geq	- \, \eta \, \varepsilon_0 \, \int_0^T  |{\bf E}|^ 2 \, e^{-2 \eta t} \; dt.
\end{equation}	
Taking the limit of the above inequality when $\eta$ tends to 0, we get
\begin{equation}\label{int3}
\int_0^T  \partial_t {\bf P} \cdot {\bf E} \; dt \; \geq \; 0.
\end{equation}
In the same way, we have $\dsp\int_0^T  \partial_t {\bf M} \cdot {\bf H}  \; dt \; \geq \; 0$, and conclude with the help of  (\ref{varEnergy}). 
\end{proof}
\subsection{On the equivalence between the different notions of passivity} \label{sec-equivalence}
It is natural to wonder whether the reciprocal of theorem \ref{Thmpassivity}, namely 
"any physically passive material is passive in the sense of definition (\ref{defpassive})", is true. Such a property seems to be commonly or implicitly admitted in the literature. However, it is far from obvious, as this is mentioned in \cite{Cessenat} for instance. 
Note that, as a consequence of (\ref{varEnergy}), the definition of physical passivity is equivalent to assuming that
\begin{equation} \label{Physicalpassivity}
 \int_0^T  \partial_t {\bf P} \cdot {\bf E} \; dt +  \int_0^T  \partial_t {\bf M} \cdot {\bf H}  \; dt \; \geq \; 0.
\end{equation}
for vector fields  ${\bf E}$, ${\bf H}$, ${\bf P}$ and ${\bf M}$ related by (\ref{Dispersive2}) and also by Maxwell's equations (\ref{Maxwell2}). \\[12pt]
Let us introduce a third notion of passivity, clearly stronger than physical passivity:
\begin{defi} {\bf (Strong physical passivity)}
A material is {\bf strongly} passive if and only if for any causal fields ${\bf E}$, ${\bf H}$, ${\bf P}$ and ${\bf M}$ related by (\ref{Dispersive2}) (but not necessarily solving (\ref{Maxwell2})), it holds
\begin{equation} \label{Strongpassivity}
\dsp \forall \; T > 0, \quad  \int_0^T  \partial_t {\bf P} \cdot {\bf E} \; dt  \geq 0 \quad \mbox{and} \quad  \int_0^T  \partial_t {\bf M} \cdot {\bf H}  \; dt \; \geq \; 0.
\end{equation}
\end{defi}
\noindent The {\bf strong physical passivity} property is the one that is most often used in the literature (see \cite{Cessenat}, \cite{Welters}). The proof of theorem \ref{Thmpassivity} shows in fact that {\bf passivity} implies {\bf strong physical passivity}. The converse is also true under additional assumptions. 
To demonstrate this, we will need a density lemma.
\begin{lema} \label{density} Let $L^2(\R^+)$ denote the subspace of causal functions of $L^2(\R)$ and $L^2_c(\R^+)$ the subspace of $L^2(\R^+)$ 
	functions with compact support. Any function of the form $|\widehat{f}|^2$ with $f \in L^2(\R)$ (in other words any non-negative integrable function) is the limit in $L^1(\R)$ of some sequence  $|\widehat{f}_n|^2$  with $f_n \in L^2_c(\R^+)$. 
	\end{lema}
\begin{proof} Let $L^2_c(\R)$ the dense subspace of $L^2(\R)$ of
	compactly supported functions. By density, there exists $(f_n)_{n=1}^{\infty} \subset L^2_c(\R)$ such that $f_n \rightarrow f \in L^2(\R)$. 
	Thus $\widehat{f}_n \rightarrow \widehat{f} \in L^2(\R)$. By construction, $\mbox{supp }f_n \subset [-T_n/2, T_n/2]$ so that
	$f_n^* (t) = f_n(t - T_n/2)$ has support in $[0, T_n]$ and thus belongs to $L^2_c(\R^+)$. Moreover,
	$$
	\widehat{f}^*_n (\omega) = e^{i\frac{\omega \, T_n}{2}} \, \widehat{f}_n (\omega) \quad \mbox{so that} \quad |\widehat{f}^*_n (\omega)| = |\widehat{f}_n (\omega)|.
	$$
Let us prove that $|\widehat{f}^*_n|^2$	converges to $|\widehat{f}|^2$ in $L^1(\R)$ which will conclude the proof. We write
$$
\int_\R \big| \,|\widehat{f}^*_n|^2 - |\widehat{f}|^2 \big|  \; d\omega \leq \big\| \; |\widehat{f}_n| - |\widehat{f}| \; \|_{L^2(\R)}  \; \big\| \; |\widehat{f}_n| + |\widehat{f}| \; \|_{L^2(\R)}  \leq C \; \big\| \; |\widehat{f}_n| - |\widehat{f}| \; \|_{L^2(\R)}.
$$
We finish the proof using the second triangular inequality and Plancherel's theorem:
$$
\int_\R \big| \,|\widehat{f}_n|^2 - |\widehat{f}|^2 \big|  \; d\omega \leq C \; \big\| \widehat{f}_n - \widehat{f}\|_{L^2(\R)} =  C \; \big\|{f}_n -{f}\|_{L^2(\R)}.
$$
\end{proof}
\begin{theo} \label{equivalence} 
	Assume that for each ${\bf x} \in \R^3$, ${\cal I}m \big(\omega \, \varepsilon({\bf x}, \omega )\big)$ and ${\cal I}m \big(\omega \, \mu({\bf x}, \omega )\big)$ are bounded functions of $\omega$. Let $\widehat{\chi}_e({\bf x}, \omega)$ and $\widehat{\chi}_m({\bf x}, \omega)$ be Fourier-Laplace functions of $L^1$ causal functions $t \mapsto \chi_e({\bf x}, t)$ and $t \mapsto \chi_m({\bf x},t)$. Assume furthermore that
\begin{equation} \label{hyposup}
	\lim_{|\omega|\rightarrow +\infty} \omega \, \big(\varepsilon({\bf x}, \omega ) - \varepsilon_0 \big)= 0,
	 \quad \lim_{|\omega|\rightarrow +\infty} \omega \, \big(\mu({\bf x}, \omega ) - \mu_0 \big)= 0,\quad \mbox{ for } \omega \in \, \C^{+}.
\end{equation}
Then, the strong passivity assumption (\ref{Strongpassivity}) implies passivity.
\end{theo}
\begin{proof} Notice that in (\ref{Strongpassivity}), $({\bf E},{\bf P})$ and $({\bf H},{\bf M})$ are not connected by Maxwell's equations. Hence, it suffices to show that (\ref{Strongpassivity}) for $({\bf E},{\bf P})$  implies ${\cal I}m \, \omega \, \varepsilon({\bf x}, \omega ) \geq 0$ in $\C^+$ (the proof is identical for $\mu$ instead of $\varepsilon$). We start from identity 
	 (\ref{int1bis}).
	 Since $t \mapsto \chi_e({\bf x}, t)$ belongs to $L^1(\R^+)$, the function $\omega \mapsto \widehat{\chi}_e({\bf x}, \omega)$ extends continuously to the real axis. Thus, this also holds for the function $\omega \mapsto \varepsilon({\bf x}, \omega )$, which is thus continuous and bounded (thanks to (\ref{hyposup})) along the real axis. Thanks to these properties, using Lebesgue's dominated convergence theorem, we can pass to the limit 
in (\ref{int1bis}) when $\eta$ tends to 0 to obtain
$$
\, \int_0^T  \partial_t {\bf P} \cdot {\bf E}  \; dt =
\int_{-\infty}^{+\infty} \; 
{\cal I}m \big(\omega \, \varepsilon({\bf x}, \omega )\big) \, |{\widehat{\bf E}_T(\omega)}|^2 \; d\omega	\; \geq \; 0.
$$
This being true for any $T$ and any ${\bf E} \in L^2(\R)$, using the density lemma \ref{density}, we get
$$
\int_{-\infty}^{+\infty} \; 
{\cal I}m \big(\omega \, \varepsilon({\bf x}, \omega )\big) \, g(\omega) \; d\omega \geq 0, \quad \forall \; g \in L^1(\R) \mbox{ such that } g \geq 0
$$
from which we immediately infer that
$$
\forall \; \omega \in \R, \quad {\cal I}m \big(\omega \, \varepsilon({\bf x}, \omega )\big) \geq 0.
$$
To extend this positivity result to the half-space $\C^+$, 
let us set, for any $R > 0$, 
$\Omega_R = \{ \omega \in \C^+ \l / \; |\omega| < R \},$
that we identify to an open set of $\R^2$. Let
$$u_{\bf x}(x,y) = {\cal I}m \big((x+iy) \, (\varepsilon({\bf x}, x+iy )-\varepsilon_0)\big), \quad (x,y) \in \R^2_+:= \R \times \R^+_*.$$ 
By analyticity of $\varepsilon({\bf x}, \omega )$ in $\C^+$, $u_{\bf x}$ is harmonic in $\R^2_+$ so that, in $\overline{\Omega}_R$, the minimum of $u(x,y)$ is attained on $\partial \Omega_R := [-R,R] \cup \Gamma_R, \; \Gamma_R = \big\{ R \, e^{i \theta}, \theta \in (0,\pi) \big\}$. Since $u_{\bf x}$ is non-negative on the real axis, we get
$$
\min_{(x,y) \, \in \, \overline{\Omega}_R} u_{\bf x}(x,y) \geq \min\big( 0, \min_{(x,y) \, \in \, \Gamma_R} u_{\bf x}(x,y)\big).
$$
On the other hand, due to (\ref{hyposup}), $\|u_{\bf x}\|_{L^\infty(\Gamma_R)} \rightarrow 0$ when $R \rightarrow + \infty$. Thus, for any $\delta > 0$ (arbitrarily small), there exists $R_\delta > 0$, with $R_\delta \rightarrow + \infty$ when $\delta \rightarrow 0$ such that
$
\sup_{(x,y) \, \in \, \Gamma_{R_\delta}} |u_{\bf x}(x,y)| < \delta, \quad \mbox{thus }  \min_{(x,y) \, \in \, \overline{\Omega}_{R_{\delta}}} u_{\bf x}(x,y) \geq - \, \delta,
$
and one easily concludes by making $\delta$ tend to 0 that $u_{\bf x}(x,y)\geq 0$ for all $(x,y)\in \R \times \R^+_*$ which implies that $\operatorname{Im}(\omega \, \varepsilon({\bf x},\omega))\geq \varepsilon_0 \operatorname{Im}\omega>0$ for all $\omega\in \C^{+}$.
\end{proof}
\noindent 
\begin{rem} The result of theorem \ref{equivalence} is likely to be valid under much weaker assumptions (removing in particular the $L^1$ assumption for $\chi_e$ or $\chi_m$), as  stated in the book \cite{Zemanian} and used e.g. in \cite{sumrules}, \cite{Welters}. 
	\end{rem}
	
\section{Local dispersive materials} \label{sec-local}
\subsection{Definition} \label{sec-defi}
We shall say that a dispersive material is {\bf local} if and only if 
$$
{\bf (LM)} \qquad \omega \mapsto  \varepsilon({\bf x}, \omega) \quad \mbox{and} \quad \omega \mapsto \mu({\bf x}, \omega) \quad \mbox{are (irreducible) rational fractions}.
$$
The term {\bf local} can be misleading since it
does not mean that the constitutive laws are local in time: \textbf{memory effects} are present a priori. However, they are of particular form, as  it will be explained in detail later.
\begin{defi} \label{defAdmissible}{\bf (Admissible local materials)} We will call local materials {\bf admissible} if and only if they are compatible with the conditions ${\bf (CP)}$, ${\bf (RP)}$ and ${\bf (HF)}$. The reader can easily verify that
$$
{\bf (ALM)} \quad  
\left\{
\begin{array}{l}
\dsp	\varepsilon({\bf x}, \omega) =  \varepsilon_0 \Big( 1 + \frac{P_e({\bf x}, -i \omega)}{Q_e({\bf x}, -i \omega)} \Big), \quad \mu({\bf x}, \omega) =  \mu_0 \Big( 1 + \frac{P_m({\bf x}, -i \omega)}{Q_m({\bf x}, -i \omega)} \Big), \mbox{ where }\\[12pt]
P_e({\bf x}, \cdot) , Q_e({\bf x}, \cdot) , P_m({\bf x}, \cdot) , Q_m({\bf x}, \cdot)  \mbox{ are polynomials with real coefficients }\\[12pt]
\mbox{that satisfy } d^oP_e({\bf x}, \cdot) < M_e := d^oQ_e({\bf x}, \cdot), d^oP_m({\bf x}, \cdot) < M_m := d^oQ_m({\bf x}, \cdot).
	\end{array} \right.
$$
\end{defi}

\begin{rem} For simplicity, we consider only the case where $M_e$ and $M_m$ do not depend on ${\bf x}$. 
	
\end{rem}
\noindent In the above framework, the relationship (\ref{Dispersive}) can be rewritten in terms of {\bf ordinary differential equations} (ODEs) in time, introducing the polarization ${\bf P}$ and magnetization ${\bf M}$ as in (\ref{defPM}). More precisely,
\begin{equation} \label{Dispersivelocal}
\left\{	\begin{array}{ll}
{\bf D}({\bf x}, t) = \varepsilon_0 \, {\bf E}({\bf x}, t) + {\bf P}({\bf x}, t) , \quad {\bf B}({\bf x}, t) = \mu_0 \,  {\bf H}({\bf x}, t)  + {\bf M}({\bf x}, t) , & (a)\\[12pt]
Q_e({\bf x}, \partial_t) \; {\bf P} = \varepsilon_0 P_e({\bf x}, \partial_t) \; {\bf E}, \quad Q_m({\bf x}, \partial_t) \; {\bf M} =\mu_0 P_m({\bf x}, \partial_t) \; {\bf H}, & (b)
\end{array} \right.
		\end{equation}
where (\ref{Dispersivelocal}(b)) is completed with 
properly chosen initial conditions compatible with (\ref{Dispersive}).
The above justifies the term {\bf local}, since differential operators 
are local in time (they 'see' only the behaviour of a function around a given time). 
Using the theory of linear ODEs, 
(\ref{Dispersivelocal}) can be expressed in the form (\ref{Dispersive_time}), where $t \mapsto \chi_e({\bf x}, t)$ and  $t \mapsto \chi_m({\bf x}, t)$ are linear combinations of exponentials, possibly multiplied by polynomials (the exponential rates are the poles of $\varepsilon({\bf x}, \cdot), \; \mu({\bf x}, \cdot)$ and the polynomial degrees are the multiplicities of these poles).
Notice that $t \mapsto \chi_e({\bf x}, t)$ and $t \mapsto \chi_m({\bf x}, t)$ do not necessarily belong to $L^1(\R^+)$ !\\[12pt]
In the following, we shall pay a particular attention to so-called {\bf lossless media} defined as follows.
\begin{defi} \label{defLossless}{\bf (Lossless local medium)} An admissible local medium is said to be {\bf lossless} if and only if the functions
$\omega \mapsto \varepsilon({\bf x}, \omega)$ and $\omega \mapsto \mu({\bf x}, \omega)$ are {\bf real-valued} along the {\bf real axis} (outside poles of course).
\end{defi}
\noindent Lossless local media are characterized by the following theorem.
\begin{theo} \label{ThmLossless} An admissible local material is lossless if and only if  $\varepsilon({\bf x}, \omega)$ and $\mu({\bf x}, \omega)$ are even in $\omega$, i.e.  the polynomials $P_e({\bf x}, \cdot)$, $Q_e({\bf x}, \cdot)$, $P_m({\bf x}, \cdot)$ and $Q_m({\bf x}, \cdot)$, are even in $\omega$.
	\end{theo}
\begin{proof}
Let us give the proof for $\varepsilon$. Let $\omega \in \R$ be such that $\omega$ and $-\omega$ are not poles of $\omega \mapsto \varepsilon({\bf x}, \omega)$. Using first the fact that $\omega\in\mathbb{R}$, the reality principle {\bf (RP)} and finally the fact that $\varepsilon({\bf x}, \omega)$ is real, we deduce 
$$
\varepsilon({\bf x}, -\omega) =  \varepsilon({\bf x}, -\overline{\omega}) = \overline{\varepsilon({\bf x}, \omega)} = \varepsilon({\bf x}, \omega).
$$
Since $\omega \mapsto \varepsilon({\bf x}, \omega)$ is rational, the above implies that $\omega \mapsto \varepsilon({\bf x}, \omega)$ is even on its domain of definition.
	\end{proof}
\noindent Satisfying ${\bf (ALM)}$ does not however guarantee the well-posedness of the evolution 
problem corresponding to (\ref{Maxwell}, \ref{Dispersivelocal}). To further investigate this question, as well as other problems such as wave dispersion, it is useful to look at the case of homogeneous local dispersive media. This is the subject of the next sections.\\[12pt]
{\bf Common examples of dispersive models.}
\begin{itemize}
\item {\bf Conductive media}. This is an example of dissipative (not lossless) medium. It corresponds to the case where $\mathbf{B}=\mu_0\mathbf{H}$ and  $\partial_t{\bf D} = \varepsilon_0 \, \partial_t{\bf E} + \sigma({\bf x}) \, {\bf E}$, where $\sigma({\bf x}) \geq 0$ is the conductivity, i.e. 
\begin{equation} \label{epsmuconductive}
\varepsilon({\bf x}, \omega) = 	\varepsilon_0 - \, \frac{\sigma({\bf x})}{i\omega}, \qquad \mu({\bf x}, \omega)=\mu_0.
	\end{equation}
\item {\bf Lorentz and Drude media}. For these media, the permittivity and permeability read
\begin{equation} \label{epsmuLorentz}
\varepsilon({\bf x}, \omega) = 	\varepsilon_0 \, \Big( 1 + \frac{\Omega_e({\bf x})^2}{\omega_e({\bf x})^2 -\omega^2} \Big), \quad
\mu({\bf x}, \omega) = 	\mu_0 \, \Big( 1 + \frac{\Omega_m({\bf x})^2}{\omega_m({\bf x})^2 -\omega^2} \Big),
\end{equation}
where $(\Omega_e({\bf x}), \omega_e({\bf x}), \Omega_m({\bf x}), \omega_m({\bf x}))$ are coefficients that characterize the medium. 
The reader will easily check that this medium is admissible and lossless. We shall see in section \ref{Generalpassive} that a natural generalization of (\ref{epsmuLorentz}) leads to a quite general class of materials, representative of all passive materials.\\[12pt]
In the case where the so-called resonance frequencies $\omega_e({\bf x})$ and $\omega_m({\bf x})$ vanish, one obtains the Drude material, which is (in some sense) the simplest dispersive lossless material. For it, 
\begin{equation} \label{epsmuDrude}
\varepsilon({\bf x}, \omega) = 	\varepsilon_0 \, \Big( 1- \frac{\Omega_e({\bf x})^2}{\omega^2} \Big), \quad
\mu({\bf x}, \omega) = 	\mu_0 \, \Big( 1- \frac{\Omega_m({\bf x})^2}{\omega^2} \Big).
\end{equation}
Finally, a lossy version of Lorentz material corresponds to the following constitutive laws  :
\begin{equation} \label{epsmulossyLorentz}
\left\{	\begin{array}{l}
\dsp \varepsilon({\bf x}, \omega) = 	\varepsilon_0 \, \Big( 1 + \frac{\Omega_e({\bf x})^2}{\omega_e({\bf x})^2 - i \, \alpha_e({\bf x}) \, \omega -\omega^2} \Big),
 \\[12pt]
\dsp \mu({\bf x}, \omega) = 	\mu_0 \, \Big( 1 + \frac{\Omega_m({\bf x})^2}{\omega_m({\bf x})^2 - i \, \alpha_m({\bf x}) \, \omega -\omega^2} \Big),
\end{array} \right.
\end{equation}
where the coefficients $\alpha_e({\bf x}) \geq 0$ and $\alpha_m({\bf x}) \geq 0$ play a role similar to the conductivity in (\ref{epsmuconductive}). In this case the poles of $\varepsilon(\mathbf{x}, .)$ and $\mu(\mathbf{x}, .)$ belong to the lower half-space $\mathbb{C}\setminus \mathbb{C}^+$.
\end{itemize}
\subsection{Homogeneous media} \label{sec-hom}
Let us consider now homogeneous local dispersive media occupying the whole space $\R^3$. Since $\varepsilon$ and $\mu$ do not depend on ${\bf x}$, the electromagnetic field is governed by the following system 
of evolution equations
\begin{equation} \label{MaxwellDispersivelocalhomogene}
\left\{	\begin{array}{lll}
	\varepsilon_0 \, \partial_t {\bf E} + {\bf rot} \, {\bf H} + \varepsilon_0 \, \partial_t {\bf P}= 0, \quad \mu_0 \partial_t {\bf H} - {\bf rot} \, {\bf E}  + \mu_0 \,  \partial_t {\bf M} = 0, & \quad {\bf x} \in \R^3, t >0, & (a)\\[12pt]
Q_e(\partial_t) \; {\bf P} =  \varepsilon_0 \,P_e(\partial_t) \; {\bf E}, \quad Q_m(\partial_t) \; {\bf M} = \mu_0 \,P_m( \partial_t) \; {\bf H},& \quad {\bf x} \in \R^3, t >0,& (b)
\end{array} \right.
		\end{equation}
where the polynomials $P_e , Q_e , P_m , Q_m$ have the properties explained in ${\bf (ALM)}$. Our main purpose is to
study the Cauchy problem, when (\ref{MaxwellDispersivelocalhomogene}) is completed by initial conditions
\begin{equation} \label{IC}
	{\bf E}({\bf x}, 0) = {\bf E}_0({\bf x}), \quad {\bf H}({\bf x}, 0) = {\bf H}_0({\bf x}), \quad({\bf E}_0,{\bf H}_0) \in L^2(\R^3)^3 \times L^2(\R^3)^3.
	\end{equation}
We are interested in the $L^2$-well-posedness, i.e. existence and uniqueness of a solution satisfying
\begin{equation} \label{regWPL2}
({\bf E}, {\bf H})	\in C^0(\R^+; L^2(\R^3)^3) \times  C^0(\R^+; L^2(\R^3)^3).
	\end{equation}
For what follows, it will be useful to introduce the notion of {\bf equivalent models}.
\subsubsection{Equivalent and non-degenerate models} \label{sec-equivalent}~\\[-20pt]
	\begin{defi} \label{defEM} {\bf (Equivalent models)} Two local dispersive models $(\varepsilon, \mu)$ and $({\varepsilon}^*, {\mu}^*)$ are said to be equivalent if and only if $
 \varepsilon(\omega) \mu(\omega) = {\varepsilon}^*(\omega){\mu}^*(\omega) \quad (\mbox{as rational fractions in } \omega).$
\end{defi}
\noindent The interest of this notion lies in the following result. 
\begin{theo} If the Cauchy problem associated to $(\varepsilon, \mu)$ is well posed, the Cauchy problem associated to any equivalent model $({\varepsilon}^*, {\mu}^*)$ is well posed too. In other words, to prove the well-posedness of the Cauchy problem for a given medium, it suffices to prove the well-posedness for any medium equivalent to it.  
	\end{theo}
\begin{proof} 
	Let $(\varepsilon^*, \, \mu^*)$ be a local dispersive media equivalent to $(\varepsilon, \mu)$. Let $\nu$ be a rational fraction such that 
	\begin{align*}
	\varepsilon^*=\nu\, \varepsilon, \qquad \mu=\nu^{-1}\,\mu^*.
	\end{align*}
	We assume that the Cauchy problem associated to $(\varepsilon, \mu)$ is well-posed, and wish to prove the well-posedness of the model $(\varepsilon^*, \; \mu^*)$. By linearity, it suffices to study the well-posedness for the Cauchy data of the form $(\mathbf{E}_0^*, 0)$, or $(0, \mathbf{H}_0^*)$. Let us consider the first case. 
We have, with obvious notations, 
	\begin{align*}
	\widehat{\mathbf{D}}^*=\varepsilon^*\widehat{\mathbf{E}}^*, \qquad \widehat{\mathbf{B}}^*=\mu^*\widehat{\mathbf{H}}^*.
	\end{align*}
  In particular, the Maxwell system (\ref{Maxwell}) in the medium $(\varepsilon^*,  \mu^*)$ with the initial data $(\mathbf{E}_0^*, 0)$ in the frequency domain reads (apply Laplace-Fourier transform and use (\ref{propFLT})) 
$$ 
  -i\omega \, \widehat{\mathbf{D}}^*-\mathbf{E}_0^*-{\bf rot} \, \widehat{\mathbf{H}}^*=0,\qquad -i\omega \, \widehat{\mathbf{B}}^*+{\bf rot} \,\widehat{\mathbf{E}}^*=0,
$$
which can be rewritten as follows, since $\nu$ is independent of the space variable, 
$$
-i\omega \, \varepsilon \, \nu \, \widehat{\mathbf{E}}^*-\mathbf{E}_0^*-{\bf rot} \, \widehat{\mathbf{H}}^*=0,\qquad  -i\omega\mu \, \widehat{\mathbf{H}}^*+{\bf rot} \,( \nu\widehat{\mathbf{E}}^*)=0.
$$
Defining $\widehat{\mathbf{E}}:=\nu\, \widehat{\mathbf{E}}^*$ and setting the initial data  $\mathbf{E}_0:=\mathbf{E}_0^*$, we obtain the following system:
$$
-i\omega \varepsilon \widehat{\mathbf{E}}-\mathbf{E}_0-{\bf rot} \, \widehat{\mathbf{H}}^*=0,\qquad  -i\omega\mu\hat{\mathbf{H}}^*+{\bf rot} \, \widehat{\mathbf{E}}=0.
$$
In the time domain, the above is reduced to the Cauchy problem for the local dispersive media $(\varepsilon, \; \mu)$ with respect to the unknowns $\widehat{\mathbf{E}}$ and $\widehat{\mathbf{H}}^*$. 
	\end{proof}
\noindent Thanks to the above property, we can restrict ourselves to the following {\bf non-degeneracy} property.
	\begin{defi} \label{defPolesSZeroes}{\bf (Non-degenerate local dispersive models)} A local dispersive model $(\varepsilon, \mu)$ is called {\bf non-degenerate } if and only if
$
\omega^2 \, \varepsilon(\omega) \mu(\omega)  \; \mbox{is an irreducible rational fraction},
$
or, equivalently, denoting by ${\cal P}_e$ (resp.  ${\cal P}_m$) the set of poles of $\varepsilon$ (resp. $\mu$) and  by ${\cal Z}_e$ (resp.   ${\cal Z}_m$) the set of zeros of $\omega \, \varepsilon$ (resp. $\omega \, \mu$), 
$$
{\cal P}_e \cap {\cal Z}_m = \emptyset, \quad {\cal P}_m \cap {\cal Z}_e = \emptyset.
$$
\end{defi}
\noindent From now on, we study only non-degenerate models. This is not restrictive due to the following result.
\begin{lema}
	Any local dispersive media is equivalent (definition \ref{defEM}) to a non-degenerate model.
	\end{lema}
\subsubsection{Plane waves. Well-posedness and stability.} 
\label{sec-WP}
To study  (\ref{MaxwellDispersivelocalhomogene}), let us concentrate on particular solutions ({\bf plane-wave} solutions) of (\ref{MaxwellDispersivelocalhomogene})(a,b) in  the form
\begin{equation} \label{planewaves}
\left\{	\begin{array}{ll}
{\bf E}(x,t) = \mathbb{E} \; \exp i ({\bf k} \cdot {\bf x} - \omega \; t) , \quad 
{\bf H}(x,t) = \mathbb{H} \; \exp i ({\bf k} \cdot {\bf x} - \omega \; t) \\[12pt]
{\bf P}(x,t) = \mathbb{P} \; \exp i ({\bf k} \cdot {\bf x} - \omega \; t) , \quad 
{\bf M}(x,t) = \mathbb{M} \; \exp i ({\bf k} \cdot {\bf x} - \omega \; t) \\[12pt]
{\bf k} \in \R^3, \quad \omega \in \C, \quad \big(\mathbb{E}, \mathbb{H}, \mathbb{P}, \mathbb{M} \big) \in\C^3 \times \C^3 \times \C^3\times \C^3.
\end{array} \right.
		\end{equation}
When $\omega = \omega_R + i \, \omega_I$, we can rewrite the plane wave solution (\ref{planewaves}) as follows (for the electric field)
\begin{equation} \label{planewave2}
{\bf E}(x,t) = \mathbb{E} \; \exp i ({\bf k} \cdot {\bf x} -\omega_R \; t )  \; e^{\omega_I t}. 
\end{equation}
It corresponds to a wave propagating in the direction of the wave vector ${\bf k}$ at the phase velocity $\omega_R / |{\bf k}|$ with an amplitude which varies in time proportionally to $e^{\omega_I \, t}$.\\[12pt]
By definition, when $\omega_I = 0$, the wave is called {\bf purely propagative}, when $\omega_I<0$ the wave is {\bf evanescent} in time, and when $\omega_I > 0$, the wave in {\bf unstable}.
\\[12pt]
In view of the time domain analysis of (\ref{MaxwellDispersivelocalhomogene}) as an evolution problem for $({\bf E}, {\bf H}, {\bf P}, {\bf M})$ in the space $L^2(\R^3)^4$, the correct point of view for looking at plane waves is to consider the {\bf wave vector} ${\bf k} \in \R^3$ as a given parameter and to look for the  related (complex) frequencies $\omega$ and corresponding amplitude vectors $\big(\mathbb{E}, \mathbb{H}, \mathbb{P}, \mathbb{M} \big)$. This approach is 
validated a posteriori by the use of the Fourier transform in space, the wave vector ${\bf k}$ being the {\bf dual variable} of the space variable ${\bf x}$.
Substituting (\ref{planewaves}) into (\ref{MaxwellDispersivelocalhomogene})(a,b) leads to
$$
\left\{	\begin{array}{lll}
{\bf k} \times \mathbb{H} = \omega \, \left(\varepsilon_0\,\mathbb{E} + \mathbb{P}\right), \quad 
{\bf k} \times \mathbb{E} = - \omega \, \left(\mu_0 \,\mathbb{H} +\mathbb{M}\right)\\[12pt]
Q_e(-i\omega) \; \mathbb{P} =\varepsilon_0 P_e(-i\omega) \; \mathbb{E}, \quad  Q_m(-i\omega) \; \mathbb{M} =\mu_0 P_m(-i\omega) \; \mathbb{H},
\end{array} \right.
$$
We can separate the solutions into two families:\\[12pt]
{\bf Purely magnetic or electric static modes}. These are solutions associated with $\omega \in {\cal P}_e: = \{\mbox{poles of } \varepsilon \equiv \mbox{zeros of } Q_e \}$ or $\omega \in {\cal P}_m:= \{\mbox{poles of } \mu \equiv \mbox{zeros of } Q_m \}$. We call these mode static because $\omega$ is independent of ${\bf k}$. In this case we have:
\begin{itemize}
	\item for $\omega \in {\cal P}_e$, for each ${\bf k}$, a three dimensional space of \textbf{amplitude vectors} corresponding to 
$$
(Magnetic~modes) \quad \mathbb{E} = 0, \quad \mathbb{P} = \omega^{-1} \, {\bf k} \times \mathbb{H}, \quad 
\mathbb{M} = - \mu_0\, \mathbb{H} , \quad \mathbb{H} \in \C^3.
$$
	\item for $\omega \in {\cal P}_m$, for each ${\bf k}$, a three-dimensional space of \textbf{amplitude vectors} corresponding to 
$$
(Electric~modes) \quad \mathbb{H} = 0, \quad \mathbb{M} = - \omega^{-1} \, {\bf k} \times \mathbb{E}, \quad 
\mathbb{P} = -\varepsilon_0 \, \mathbb{E} , \quad \mathbb{E} \in \C^3.
$$ 
\end{itemize}
{\bf Maxwell modes}. When  $\omega \notin {\cal P}_e \cup {\cal P}_m$, one can first eliminate $\mathbb{P}$ and $\mathbb{M}$ to 
obtain
\begin{equation}\label{eq.relmode}
\mathbb{P} = \big(\varepsilon(\omega) - \varepsilon_0 \big) \, \mathbb{E}, \ \;  \mathbb{M} = \big(\mu(\omega) - \mu_0 \big)\ \mathbb{H}, \; \ {\bf k} \times \mathbb{H} = \omega \, \varepsilon(\omega) \, \; \mathbb{E} \mbox{ and } {\bf k} \times \mathbb{E} = - \; \omega \, \mu(\omega) \, \mathbb{H}.
\end{equation}
From \eqref{eq.relmode}, we obtain the eigenvalue problem $
- \; {\bf k} \times ({\bf k} \times \mathbb{E}) = \omega^2 \, \varepsilon(\omega) \, \mu(\omega) \, \mathbb{E}
$
which we can solve to get
\begin{itemize}
\item[(i)] Either $\omega^2 \varepsilon(\omega) \, \mu(\omega) \, =0$ (curl-free static modes) and we have three subcases:
\begin{enumerate}
\item  if $\omega \varepsilon(\omega)=0$ and  $\omega \mu(\omega)\neq 0$, then ${\bf k} \times \mathbb{E}=0$  and $\mathbb{H} =0$  (1D space of solutions);
\item  if $\omega \varepsilon(\omega)\neq 0$ and  $\omega \mu(\omega)=0$, then ${\bf k} \times \mathbb{H}=0$ and $\mathbb{E} =0$  (1D space of solutions);
\item  if $\omega \varepsilon(\omega)= 0$ and  $\omega \mu(\omega)=0$, then ${\bf k} \times  \mathbb{E} =0$ and ${\bf k} \times \mathbb{H}=0$  (2D space of solutions).
\end{enumerate}
\item [(ii)] Either $\omega^2 \varepsilon(\omega) \, \mu(\omega) \neq 0$, one gets from \label{eq.EVP} and \eqref{eq.relmode} that ${\bf k} \cdot \mathbb{E} = 0$ and and $\mathbb{H} =-(\omega \mu(\omega))^{-1}{\bf k} \times \mathbb{E}$ (eigenspace of dimension $2$)
for $\omega$ and ${\bf k}$ being linked by the dispersion relation
\begin{equation} \label{relationdispersion}
\omega^2 \, \varepsilon(\omega) \, \mu(\omega)= |{\bf k}|^2 .
\end{equation}
\end{itemize}
\begin{rem} \label{remequivdispersion} Two equivalent media, in the sense of definition \ref{defEM}, have the same dispersion relation.
	\end{rem}
\begin{rem} \label{rempropdispersion}The dispersion equation (\ref{relationdispersion}) can be seen as a polynomial equation
in  $\omega$ with degree $N = M_e + M_m + 2$ (where we recall that $M_e$ and $M_m$ are the respective degrees of the polynomials $Q_e$ and $Q_m$, see definition \ref{defAdmissible})
whose coefficients are affine functions  in $|{\bf k}|^2$ and whose higher order term is independent of $|{\bf k}|$. As a consequence, this equation admits $N$ 
branches of solutions
\begin{equation} \label{branches}
|{\bf k}| \rightarrow \omega_j\big(|{\bf k}| \big), \quad 1 \leq j \leq N,
\end{equation}
where each function $|{\bf k}| \rightarrow \omega_j\big(|{\bf k}| \big)$ is continuous and piecewise analytic. Moreover,
it is known \cite{Kato} that the loss of analyticity can occur only at a values of $|{\bf k}|$ for which  $\omega_j\big(|{\bf k}| \big)$ is not a simple root of (\ref{relationdispersion}). 
\end{rem}
Let us study the $L^2-$well-posedness of (\ref{MaxwellDispersivelocalhomogene}) , i.e. let us look for solutions of (\ref{MaxwellDispersivelocalhomogene}) such that
\begin{equation} \label{regsolutions}
({\bf E}, {\bf H},{\bf P},{\bf M}) \in C^0(\R^+; L^2(\R^3))^4
\end{equation}
for given initial fields ${\bf E}_0 \equiv {\bf E}(\cdot,0) \in L^2(\R^3)^3$ and ${\bf H}_0 \equiv {\bf H}(\cdot,0) \in L^2(\R^3)^3$. 
\begin{defi} (Well-posedness and stability) \label{WPS} 
The problem (\ref{MaxwellDispersivelocalhomogene}) is {\bf well posed} if there exists a unique solution 
	satisfying (\ref{regsolutions}) and, for some $C(t) \geq 0,$
\begin{equation} \label{L2estimate}
\|{\bf E}(\cdot,t)\|_{L^2} + \|{\bf H}(\cdot,t)\|_{L^2} +\|{\bf P}(\cdot,t)\|_{L^2} +\|{\bf M}(\cdot,t)\|_{L^2} \leq C(t) \; \big( \, \|{\bf E}_0\|_{L^2} + \|{\bf H}_0\|_{L^2} \, \big)
\end{equation}
Otherwise, the problem is said {\bf strongly ill posed}.
If, in addition, $C(t) = C \; (1 + t^p)$ for some $C > 0$ and $p \in \mathbb{N}$,
which prevents any exponential blow-up, then the problem is said to be {\bf stable}.
\end{defi}	
\noindent Using Fourier analysis (in particular Plancherel's theorem), see e. g. \cite{kreiss}, it is not difficult to establish the
\begin{lema} \label{lemaWPS}
	The problem (\ref{MaxwellDispersivelocalhomogene}) is {\bf well posed} if and only if there exists $M \leq 0$ such that
\begin{equation} \label{condWP}
\forall \; 1 \leq j \leq N, \quad {\cal I}m \, \omega_j\big(|{\bf k}|\big) \leq M.
\end{equation}
The problem (\ref{MaxwellDispersivelocalhomogene}) is {\bf stable} if and only if 
\begin{equation} \label{condS}
{\cal P}_e \cup {\cal P}_m \subset \C \setminus \C^+ \quad \mbox{and} \quad \forall \; 1 \leq j \leq N, \quad {\cal I}m \, \omega_j\big(|{\bf k}|\big) \leq 0.
\end{equation}	
\end{lema}
	\begin{rem} \label{remzeroes} Looking at (\ref{relationdispersion}) when  ${\bf k} \rightarrow 0$ shows that, for stable media, ${\cal Z}_e \cup {\cal Z}_m \subset \C \setminus \C^+$ too.\end{rem}
\noindent Thus, {\bf strongly ill posed models} admit unstable plane waves whose rate of exponential blow-up can be arbitrarily large, while for {\bf unstable models} this rate must be uniformly bounded. 
\begin{theo} \label{thmWPCauchy}
For any local admissible material, the problem (\ref{MaxwellDispersivelocalhomogene}) is {\bf well posed}.
	\end{theo} 
\begin{proof}
	\noindent By continuity, ${\cal I}m \, \omega_j\big(|{\bf k}|\big)$ can blow up only 
	when $|{\bf k}| \rightarrow + \infty$. Inspecting (\ref{relationdispersion}), one sees that,
	\begin{itemize}
		\item Either $\omega_j\big(|{\bf k}|\big) = \pm \, c_0 \; |{\bf k}| + O(1) \; (|{\bf k}| \rightarrow + \infty)$, and thus ${\cal I}m \, \omega_j\big(|{\bf k}|\big)$ remain bounded,
	\item Either $\dsp \lim_{|{\bf k}| \rightarrow + \infty} \omega_j\big(|{\bf k}|\big)$ exists and belongs to ${\cal P}_e \cup {\cal P}_m$, so that ${\cal I}m \, \omega_j\big(|{\bf k}|\big)$ is bounded too.
	\end{itemize}
This proves well-posedness. 
	\end{proof}
	\begin{rem} \label{HFcondition}
		One sees here the mathematical importance of condition ${\bf (HF)}$. Assume for instance that
		$$
\varepsilon(\omega) \, \mu(\omega) \sim C_\infty \; \omega^q \quad (|\omega| \rightarrow + \infty) , \quad q \in \N^*, \quad C_\infty = \rho_\infty \, e^{i \theta_\infty}, \quad \rho_\infty > 0, \; \theta_\infty \in [0, 2 \pi[.	$$
In this case , (\ref{relationdispersion}) would admit $q+1$ solutions $\omega_\ell^\infty\big(|{\bf k}|\big)$ satisfying
$$
\omega_\ell^\infty\big(|{\bf k}|\big) \sim \rho_\infty \; \exp {i\Big(\frac{\theta_\infty + \ell \, \pi}{q+2} \Big)} \; |{\bf k}|^{\frac{2}{q+2}}
\quad ( |{\bf k}| \rightarrow + \infty)
$$
so that, at least for one of them, the imaginary part of $\omega_\ell^\infty\big(|{\bf k}|\big)$ would tend to $+ \infty$ !
		\end{rem}
\subsubsection{Non-dissipative media. Definition and first results.}\label{sec-nondissi}
Non-dissipative media are a particular sub-class of stable media.
\begin{defi}{\bf (Non-dissipative media)}. We shall say that a local medium is {\bf non-dissipative} if and only if all plane waves in such a medium are purely propagative. In other words, a medium is non-dissipative if and only if all solutions of the dispersion relation (\ref{relationdispersion}) are {\bf real}.\end{defi}
\noindent Let us establish a connection with the notion of {\bf lossless} medium (see definition \ref{defLossless}).
Let us set
\begin{equation} \label{defF}
{\cal F}(\omega): =  \omega^2 \, \varepsilon(\omega) \, \mu(\omega),  \quad \mbox{(so that (\ref{relationdispersion}) $\Longleftrightarrow {\cal F}(\omega) = |{\bf k}|^2$)}
\end{equation}
\begin{lema} \label{Petitlemme} A rational function $R(\omega)$ with real poles and zeros, for which $R(- \overline{\omega}) = \overline{R(\omega)}$, is even.
\end{lema}
\begin{proof} The property $R(- \overline{\omega}) = \overline{R(\omega)}$ implies that if $\omega$ is a pole (resp. a zero) of
	$R(\omega)$, $- \overline{\omega}$ is a pole (resp. a zero) too. Since the poles (resp. the zeros) are in addition real, they are symmetrically distributed with respect to the origin. In other words we can write (with obvious notation)
	$$
R(\omega) = A \; \frac{(\omega^2 - z_1^2) \cdots (\omega^2 - z_{N_z}^2)}{(\omega^2 - p_1^2) \cdots (\omega^2 - p_{N_p}^2)} \quad ( \mbox{with } A \in \R).
	$$
This finishes the proof.
\end{proof}
\begin{lema} \label{ND-Lossless} If a non-degenerate local medium is non-dissipative, the poles and the zeros of  ${\cal F}(\omega)$ are all real, and their multiplicity is less or equal to 2. In particular, $\omega=0$ is not a zero of $\mu(\omega)$ or $\varepsilon(\omega)$. 
	\end{lema}
\begin{proof}  Let $\omega^*$ be a zero of $\omega^2\,\varepsilon \, \mu$ of multiplicity  $m > 0$. For some non-zero $A_*$, we can write
	$$
\omega^2 \, \varepsilon(\omega) \, \mu(\omega) \sim A_* \; (\omega - \omega^*)^{m}	\text{ in the vicinity of } \omega^*.
	$$ 
Rewriting the dispersion equation (\ref{relationdispersion}) as $A_* \; (\omega - \omega^*)^{m} \big( 1 + O(\omega - \omega^*)\big) = |{\bf k}|^2$, one deduces (using the implicit function theorem) that for $|{\bf k}|\rightarrow 0$,  (\ref{relationdispersion}) admits $m$ branches of solutions $\omega_\ell(|{\bf k}|), 1 \leq \ell \leq m$:
$$
\omega_\ell(|{\bf k}|) = \omega^* + A_*  \; |{\bf k}|^{\frac{2}{m}}\; \exp {i \, \Big(\frac{\ell \, \pi}{m}}\Big) \; \big(1 + o(1) \big), \quad 1 \leq \ell \leq m
$$
Then, writing that $\omega_\ell(|{\bf k}|) \in \R$ shows that $\omega^* \in \R$ and $m \leq 2$ (as well as $A_* \in \R$).\\[12pt] 
In the same way, let $\omega^*$ be a pole of $\omega^2\,\varepsilon \, \mu$  of multiplicity  $m > 0$.  For some non-zero  $A_*$, we can write
	$$
\omega^2 \, \varepsilon(\omega) \, \mu(\omega) \sim A_*\; (\omega - \omega^*)^{-m}	
	$$ 
Rewriting the dispersion equation (\ref{relationdispersion}) as $A_* \; (\omega - \omega^*)^{-m} \big( 1 + O(\omega - \omega^*)\big) = |{\bf k}|^2$,  one deduces (using the implicit function theorem) that  for $|{\bf k}|\rightarrow +\infty$, (\ref{relationdispersion}) admits $m$ branches of solutions $\omega_\ell(|{\bf k}|), 1 \leq \ell \leq m$:
$$
\omega_\ell(|{\bf k}|) = \omega^* + A_*  \; |{\bf k}|^{-\frac{2}{m}}\; \exp {i \, \Big(\frac{\ell \, \pi}{m}}\Big) \; \Big(1 + o\big(1\big) \Big), \quad 1 \leq \ell \leq m.
$$
Again, writing that $\omega_\ell(|{\bf k}|) \in \R$ shows that $\omega^* \in \R$ and $m \leq 2$ (and $A_* \in \R$).	\end{proof}
\begin{coro} \label{coroND-Lossless} If a non-degenerate local medium is non-dissipative, the functions $\varepsilon(\omega)$ and $\mu(\omega)$ are even and their poles are real and their zeros are real. Moreover,
	\begin{itemize} 
		\item {(i}) The multiplicity of each zero or each non-zero pole is at most 2,
and equals 1 if such a pole or zero is shared by $\varepsilon(\omega)$ and $\mu(\omega)$. 
\item{(ii)} The multiplicity of $\,0\,$ as a pole of $\varepsilon(\omega)$ or $\mu(\omega)$ is at most $4$. 
\end{itemize}
As a consequence, $\varepsilon$ and $\mu$ are necessarily of the form 
\begin{equation} \label{NDformgene} 
\left\{ \begin{array}{l}
\dsp \varepsilon(\omega) = \varepsilon_0 \; \Big( 1 + \sum_{\ell = 0}^{N_e} \frac{a_{e,\ell}}{\omega_{e,\ell}^2 -\omega^2} + \sum_{\ell = 0}^{N_e} \frac{b_{e,\ell}}{(\omega_{e,\ell}^2 -\omega^2\big)^2}\Big), \quad (a_{e,\ell},b_{e,\ell}) \in \R^2, \quad 0 \leq \ell \leq N_e, \\[12pt] 
\dsp  \mu(\omega) = \mu_0 \; \Big( 1 + \sum_{\ell = 1}^{N_m} \frac{a_{m,\ell}}{\omega_{e,\ell}^2 -\omega^2} + \sum_{\ell = 1}^{N_m} \frac{b_{m,\ell}}{\big(\omega_{m,\ell}^2 -\omega^2\big)^2}\Big), \quad (a_{e,\ell},b_{e,\ell}) \in \R^2, \quad 0 \leq \ell \leq N_m.
  \end{array} \right.
\end{equation}
with $0 \leq \omega_{p,e}^0 < \cdots < \omega_{p,e}^{N_e}, \; 0 \leq \omega_{p,m}^0 <  \cdots < \omega_{p,m}^{N_m} \; $ and $b_{e,\ell} = b_{m,\ell'} = 0$ if $\omega_{e,\ell} = \omega_{m, \ell'}$.\\[12pt] In particular, the medium is lossless in the sense of definition \ref{defLossless}.\end{coro}
\begin{proof}
Because of the non-degeneracy assumption, each zero of $\varepsilon(\omega)$ or $\mu(\omega)$ is a zero of ${\cal F}(\omega)$. The same is true for the poles of $\varepsilon(\omega)$ or $\mu(\omega)$ which are different from 0. Thus, by lemma \ref{ND-Lossless}, all poles and zeros of $ \varepsilon(\omega)$ and $\mu(\omega)$ are real and lemma \ref{Petitlemme} ensures that $\varepsilon(\omega)$ and $\mu(\omega)$ are even. If $\omega^*$ is a zero of multiplicity $m_\varepsilon$  of $\varepsilon(\omega)$ and a zero of multiplicity $m_\mu$  of $\mu(\omega)$, first of all, by lemma \ref{ND-Lossless}, $\omega_*\neq 0$. Thus, its multiplicity as a zero of ${\cal F}(\omega)$ is $m_\varepsilon + m_\mu$, and lemma \ref{ND-Lossless} yields $m_\varepsilon + m_\mu \leq 2$, which shows property (i) for the zeros. The same reasoning applies to non-zero poles, which completes the proof of property (i).  If $0$ is a pole of multiplicity $m_\varepsilon$  of $\varepsilon(\omega)$ and a pole of multiplicity $m_\mu$  of $\mu(\omega)$, its multiplicity as a pole of ${\cal F}(\omega)$ is $m_\varepsilon + m_\mu - 2$.
Then (ii) follows from lemma \ref{ND-Lossless} again.
Taking into account {\bf (HF)}, formulas (\ref{NDformgene}) are obtained by the usual partial fraction expansion, the reality of the coefficients follow from the reality principle {\bf (RP)}, and the last condition is obtained from (i), (ii). \end{proof}

\subsubsection{Non-dissipative passive local materials.} \label{sec-nondissipassive}
The reciprocal of lemma \ref{ND-Lossless}, namely that any lossless material is non-dissipative, is not true. Consider
$$
\varepsilon(\omega) = 	\varepsilon_0 \, \Big( 1 + \frac{\Omega_e^2}{\omega^2} \Big), \quad
\mu(\omega) = 	\mu_0 .
$$ 
The dispersion relation (\ref{relationdispersion}) reads $\omega^2 = |{\bf k}| - \Omega_e^2$, and hence $\omega\notin \mathbb{R}$
for $|{\bf k}| < \Omega_e$. 
However, the reciprocal of Lemma \ref{ND-Lossless} holds true for passive materials.
\begin{lema} \label{lem_LL_ND} Any lossless local passive material is non-dissipative.
	\end{lema}
\begin{proof}
	Assume that (\ref{relationdispersion}) admits for some ${\bf k} \neq 0$, a non-real root $\omega$. Since $- \omega$ is a solution  too (cf. Theorem \ref{ThmLossless}), we can assume that ${\cal I}m \, \omega > 0$. Taking the real and imaginary part of (\ref{relationdispersion}), we thus get
	$$
	\begin{array}{l}
(a) \quad		{\cal R}e \big(\omega \, \varepsilon(\omega)\big) \, {\cal R}e \big(\omega \, \mu(\omega)\big) = |{\bf k}|^2 + {\cal I}m \big(\omega \, \varepsilon(\omega)\big) \, {\cal I}m\big(\omega \, \mu(\omega)\big) \\[12pt]
(b) \quad		{\cal R}e \big(\omega \, \varepsilon(\omega)\big) \, {\cal I}m \big(\omega \, \mu(\omega)\big) + {\cal I}m \big(\omega \, \varepsilon(\omega)\big) \, {\cal R}e \big(\omega \, \mu(\omega)\big) = 0
		\end{array}
		$$
Since $\omega \, \varepsilon(\omega)$ and $\omega \, \mu(\omega)$ are Herglotz functions (passivity), we deduce from (a) that ${\cal R}e \big(\omega \, \varepsilon(\omega)\big)$ and ${\cal R}e \big(\omega \, \mu(\omega)\big)$ are of the same sign while (b) says that they are of opposite signs. This is a contradiction.
	\end{proof}
\begin{theo}  \label{GenLorentz} [Representation of local lossless passive materials] The electric permittivity and magnetic permeability $\big(\varepsilon(\omega), \mu(\omega) \big)$  associated to a lossless passive local material are necessarily of the form (we speak of {\bf generalized Lorentz materials})
	\begin{equation} \label{Lorentzlaws}
\varepsilon(\omega)	= \varepsilon_0 \; \Big( 1 + \sum_{\ell = 1}^{N_e} \frac{\Omega_{e,\ell}^2}{\omega_{e,\ell}^2 -\omega^2}\Big), \quad \mu(\omega)	= \mu_0 \; \Big( 1 + \sum_{\ell = 1}^{N_m} \frac{\Omega_{m,\ell}^2}{\omega_{m,\ell}^2 -\omega^2}\Big).
		\end{equation}
		where the $\omega_{e,\ell}$'s and $\omega_{m,\ell}$'s are real, and the $\Omega_{e,\ell}$'s and $\Omega_{m,\ell}$'s are positive real numbers. Reciprocally, a medium associated with (\ref{Lorentzlaws}) is necessarily passive and lossless.
		\end{theo}
\begin{proof} By lemma \ref {lem_LL_ND}, we know that the medium is non-dissipative. 
	Thus, by corollary \ref{coroND-Lossless}, $\big(\varepsilon(\omega), \mu(\omega) \big)$ are of the form (\ref{NDformgene}). Moreover, since $ \omega \, \varepsilon(\omega)$ and $ \omega \, \mu(\omega)$ are Herglotz functions (cf. the definition \ref{defpassive} of passive media), 
we know by lemma \ref{lemHerglotz} that their real poles (i.e. their poles since all of them are real) are simple: in other words, except maybe for 
$\omega = 0$, the poles of $\varepsilon(\omega)$ and $\mu(\omega)$ are simple. Thus, the coefficients $b_{e,\ell}$ and 
$b_{m,\ell}$ appearing in (\ref{NDformgene}) vanish. Finally, using (\ref{NDformgene}), one can compute explicitly the residue of $\omega \varepsilon(\omega)$ and $\omega \mu(\omega)$ at each of their poles :
$$
{\cal R}es \big( \omega \, \varepsilon(\omega), \pm \, \omega_{e,\ell} \big) = \varepsilon_0 \;  \frac{a_{e,\ell}}{2} , \quad
{\cal R}es \big( \omega \, \mu(\omega), \pm \, \omega_{m,\ell} \big) = \mu_0 \;  \frac{a_{m,\ell}}{2},
$$
which shows, by lemma \ref{lemHerglotz},  that $a_{e,\ell}$ and $a_{m,\ell}$ are positive numbers, i. e. 
$a_{e,\ell} = \Omega_{e,\ell}^2$ and $a_{m,\ell} = \Omega_{m,\ell}^2$.\\[12pt]
Reciprocally, to prove the passivity of generalized Lorentz materials, we compute
$$
\omega \, \varepsilon(\omega)	= \varepsilon_0 \; \Big( \omega  + \sum_{\ell = 1}^{N_e} \Omega_{e,\ell}^2 \,  \frac{\omega}{\omega_{e,\ell}^2 -\omega^2}\Big) \equiv \varepsilon_0 \; \Big( \omega  + \sum_{\ell = 0}^{N_e} \; \Omega_{e,\ell}^2 \, \frac{\omega \, \omega_{e,\ell}^2 - \overline{\omega} \, |\omega|^2 }{|\,\omega_{e,\ell}^2 -\omega^2|^2}\Big)
$$
so that
$$
{\cal I}m \big(\omega \, \varepsilon(\omega))	= \varepsilon_0 \; ({\cal I}m \, \omega) \; \Big( 1 + \sum_{\ell = 0}^{N_e} \; \Omega_{e,\ell}^2 \, \frac{ \omega_{e,\ell}^2 + |\omega|^2 }{|\, \omega_{e,\ell}^2 -\omega^2|^2}\Big),
$$
which proves that $\omega \, \varepsilon(\omega)$ is a Herglotz function. The same holds for $\omega \, \mu(\omega)$.
\end{proof}
\begin{figure}[h] 
\centerline{
\includegraphics[height=4cm]{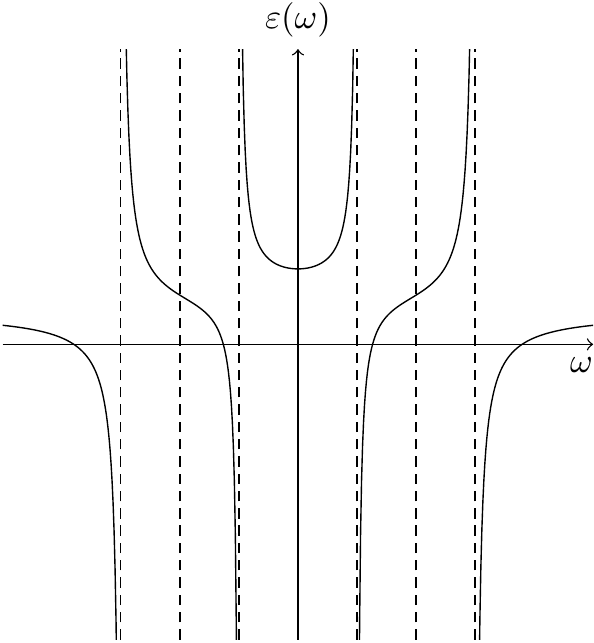} \hspace*{2cm} \includegraphics[height=4cm]{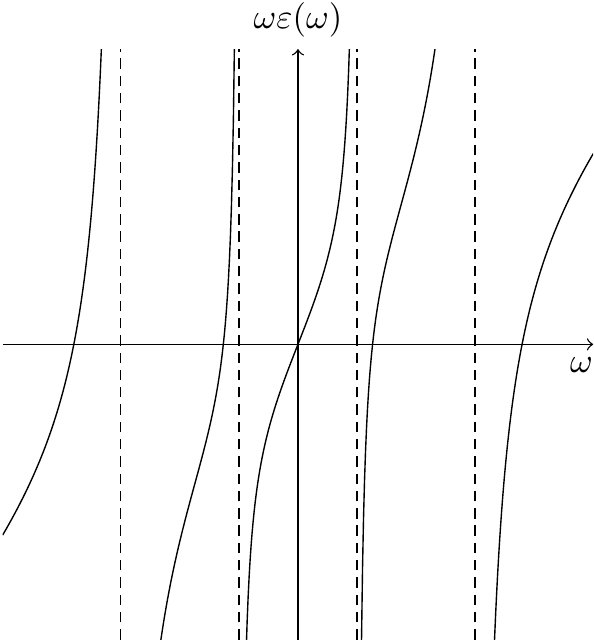}}
\caption{Left: a typical graph of $\varepsilon(\omega)$ when $0$ is not a pole. Right: illustration of the growing property}
\label{Fig1}
\end{figure}
\begin{rem} \label{positivity0}
For any Lorentz material, if $0$ is not a pole of $\varepsilon$ (resp. $\mu$),  $\varepsilon(0) \geq \varepsilon_0$ (resp. $\mu(0) \geq \mu_0$).
	\end{rem}
\noindent One can wonder whether a non-dissipative local material is necessarily passive. This is not the case as it can be 
guessed from the fact that the non-dissipativity is linked to a property of the product $\varepsilon \, \mu$ while passivity relies on a property for each of the functions $\varepsilon$ and $\mu$. Let us consider
the following example
$$
\varepsilon(\omega) = \varepsilon_0 \; \Big( 1  - \frac{\Omega_e^2}{\omega^2}\Big) \, \Big( 1  - \frac{\Omega_m^2}{\omega^2}\Big), \quad \mu(\omega) =   \mu_0.
$$
The corresponding medium is equivalent, in the sense of the definition \ref{defEM}, to the Drude medium
$$
\varepsilon(\omega) = \varepsilon_0 \; \Big( 1  - \frac{\Omega_e^2}{\omega^2}\Big) , \quad \mu(\omega) =  \mu_0 \; \Big( 1  - \frac{\Omega_m^2}{\omega^2}\Big).
$$
and is thus non-dissipative (it has the same dispersion relation). However this medium is not passive since
$$
\omega \, \varepsilon(\omega) \sim \frac{\Omega_e^2 \, \Omega_m^2}{\omega^3} \mbox{ when } \omega \rightarrow 0
$$
so that for $\omega = \rho \, e^{i \theta}, \; \theta\in \; ]0, \pi[ \, $, $\rho\rightarrow 0$, $\omega \, \varepsilon(\omega) \sim \rho \, e^{- 3i \theta}$, which lies in $\C^-$ when  $\theta \in \;  ] \pi/3, 2\pi/3 [ \, .$\\[12pt]
Nevertheless we have the following result, a proof of which is given in Appendix.
\begin{theo} \label{thm.equivpassive} A non-dissipative local material is necessarily  equivalent to a passive material (thus to a generalized Lorentz medium).
	\end{theo}
\noindent In a local passive non-dissipative material, Maxwell's equations can be rewritten, modulo the introduction of appropriate auxiliary unknowns, as the coupling of standard Maxwell's equations in the vacuum with a system of linear second order ODE's  (harmonic oscillators). The precise result is the following:
\begin{theo}\label{PDEmodel}  A PDE-system associated to dispersive Maxwell's equations in the media (\ref{Lorentzlaws}) reads :
\begin{equation} \label{Lorentzsystem}
\left\{	\begin{array}{lll}
\dsp	\varepsilon_0 \, \partial_t {\bf E} + {\bf rot} \, {\bf H} + \varepsilon_0 \, \sum_{\ell = 1}^{N_e}\Omega_{e, \ell}^2  \partial_t \bbP_\ell= 0, & \dsp \mu_0 \, \partial_t {\bf H} - {\bf rot} \, {\bf E}  + \mu_0 \,  \sum_{\ell = 1}^{N_m} \Omega_{m, \ell}^2\partial_t \bbM_\ell = 0,  \\[18pt]
\partial_t^2 \bbP_\ell + \omega_{e, \ell}^2 \,  \bbP_\ell = \,  {\bf E}, &  \partial_t^2 \bbM_\ell + \omega_{m, \ell}^2 \, \bbM_\ell =  \,  {\bf H}.
\end{array} \right.
	\end{equation}
~\\[-4pt]with  $\bbP_\ell ({\bf x},0) = \partial_t\bbP_\ell ({\bf x},0)=0, \, 0 \leq \ell \leq N_e \, , $ and $\bbM_\ell ({\bf x},0) = \partial_t\bbM_\ell ({\bf x},0)=0, \, 0 \leq \ell \leq N_m \, $.
		\end{theo}
		\begin{proof} With (\ref{Lorentzlaws}), we define the polarization ${\bf P}$ and the magnetization ${\bf M}$ (see (\ref{defPM})) as
\begin{equation} \label{defPellMell}
{\bf P} = \varepsilon_0 \, \sum_{\ell = 0}^{N_e} \; \Omega_{e,\ell}^2 \, \bbP_\ell, \quad {\bf M} = \mu_0 \, \sum_{\ell = 0}^{N_m} \;  \Omega_{m,\ell}^2 \, \bbM_\ell,
\end{equation}
where the Fourier-Laplace transforms of the $\bbP_\ell$'s and the $\bbM_\ell$'s satisfy
$$
\big({\omega_{e,\ell}^2 -\omega^2}\big) \, \widehat\bbP_\ell =  \widehat{\bf E}, \quad \big({\omega_{m,\ell}^2 -\omega^2}\big) \, \widehat\bbM_\ell =  \widehat{\bf H},
$$
It is then obvious to conclude.
\end{proof}

\noindent We finish this section with a characteristic property of local passive materials: the {\bf growing property}.
\begin{theo}\label{thm.growing} Any passive local material satisfies the growing property :
	\begin{equation} \label{growingproperty}
		\forall \; \omega \in \R \setminus \mathcal{P}_e, \quad \frac{d}{d\omega}(\omega \, \varepsilon) (\omega) > 0, \quad 
		\forall \; \omega \in \R \setminus \mathcal{P}_m, \quad \frac{d}{d\omega}(\omega \, \mu) (\omega) > 0.
		\end{equation}
Reciprocally, if a (non-degenerate) non-dissipative local material satisfies (\ref{growingproperty}), it is passive. 
\end{theo}
\begin{proof}
For the direct statement, we can use formula (\ref{Lorentzlaws}) and directly compute 
$$
\frac{d}{d\omega} \big(\omega \, \varepsilon \big) (\omega)	= \varepsilon_0 \; \Big( 1 + \sum_{\ell = 1}^{N_e} \; \Omega_{e,\ell}^2 \; \frac{d}{d\omega}\Big(\frac{\omega}{\omega_{e,\ell}^2 -\omega^2}\Big) \Big) \equiv \varepsilon_0 \; \Big( 1 + \sum_{\ell = 1}^{N_e} \; \Omega_{e,\ell}^2 \; \frac{\omega_{e,\ell}^2 + \omega^2}{\big(\omega_{e,\ell}^2 -\omega^2\big)^2}\Big)
$$
For the reciprocal statement, assume now that $(\omega\varepsilon(\omega))'>0$ for $\omega \notin {\cal P}_e$. This immediately implies that all zeros of $\omega\varepsilon(\omega)$ are simple and that $\omega\varepsilon(\omega)$ does not admit local minima or maxima. As a consequence, between two consecutive zeros of $\omega\varepsilon(\omega)$ there is one pole of $\omega\varepsilon(\omega)$
and between two consecutive poles there is one zero. Therefore, zeros and poles of $\omega\varepsilon(\omega)$ interlace along the real axis. Since $\varepsilon(\omega)\rightarrow \varepsilon_0$ as $\omega\rightarrow \infty$, the number of poles (counted with multiplicities) of $\omega\varepsilon(\omega)$ is smaller than the number of zeros of $\omega\varepsilon(\omega)$ by one. 
This, together with the fact that $\omega\varepsilon(\omega)$ has only simple zeros, implies that all the poles of $\omega\varepsilon(\omega)$ are simple too. So are the poles of $\varepsilon(\omega)$, with a possible exception of $\omega=0$, which can be a double pole (it cannot be a simple pole since $\varepsilon(\omega)$ is an even function). As $\varepsilon(\omega)$ is even,
real on the real axis and $\varepsilon(\omega)\rightarrow \varepsilon_0$ as $\omega\rightarrow \infty$, it admits the following partial fraction expansion
\begin{align*}
\varepsilon(\omega)=\varepsilon_0 \; \Big( 1+\sum\limits_{\ell=0}^{N_e}\frac{a_{e,\ell}}{\omega^2_{\ell}-\omega^2}\Big), \qquad \omega_{\ell}\in\mathbb{R}, \; a_{e,\ell}\in\mathbb{R}, \; \ell=0,\ldots,N_e.
\end{align*}
It remains to show that $a_{e,\ell}>0$  ( i. e. $a_{e,\ell} = \Omega_{e,\ell}^2$ ) for all $\ell=0,\ldots, N_e$. For this we compute explicitly 
\begin{align*}
(\omega\varepsilon(\omega))'=\varepsilon_0  \; \Big( 1+\sum\limits_{\ell=0}^{N_e} \frac{a_{e,\ell}(\omega^2+\omega_{\ell}^2)}{(\omega_{\ell}^2-\omega^2)^2}\Big).
\end{align*}
In the vicinity of $\omega=\omega_{\ell}$ the above expression is of the same sign as $a_{e,\ell}$ (this remains true for $\omega_{\ell}=0$), therefore $a_{e,\ell}>0$ for all $\ell=0,\ldots, n$.
\end{proof}
\subsection{Dispersion analysis of non-dissipative materials} \label{sec-dispersionanalysis}
\subsubsection{Introduction} \label{sec_introdispersion}
In this section we will study the solutions of the dispersion relation (\ref{relationdispersion}), seen again as an equation for $\omega$, ${\bf k}$ being a parameter. According to theorem \ref{thm.equivpassive} and remark \ref{remequivdispersion}, we can restrict ourselves to passive media, or, with theorem \ref{GenLorentz}, to generalized Lorentz materials associated with (\ref{Lorentzlaws}). In this case, for fixed ${\bf k}$, (\ref{relationdispersion}) is an equation of degree $2N_e + 2 N_m + 2$ in $\omega$ (more precisely of degree $N := N_e + N_m + 1$ in $\omega^2$). \\[12pt]
For the simplicity of the exposition and to avoid the treatment of multiple cases, we shall limit our discussion to the particular case where $0$ is not a pole of $\varepsilon$ nor $\mu$ (this is not restrictive, see remark \ref{othercases}):
\begin{equation} \label{hypononzero}
0 \notin {\cal P} := {\cal P}_e \cup {\cal P}_m.
	\end{equation}
Let us introduce a function  which will play a privileged role in the forthcoming analysis, namely
\begin{equation} \label{defD}
{\cal D}(\omega): = {\cal F}'(\omega)= (\omega \, \varepsilon)'(\omega) \; \big(\omega \, \mu(\omega)\big) + (\omega \, \mu)'(\omega) \, \big(\omega \, \varepsilon(\omega)\big) , \quad \omega \notin {\cal P}.
\end{equation}
In what follows, we shall refer repeatedly to the following technical lemma:
\begin{lema}\label{lemmetechnique} At any point $\omega \in \R \setminus {\cal P}$ where ${\cal F}(\omega) > 0$, ${\cal D}(\omega) \neq 0$ and has the same sign as $\omega \, \varepsilon(\omega) > 0$ (or $\omega \, \mu(\omega) > 0$). Let $I$ be an interval which does not intersect $\cal P$. If ${\cal F}$ is positive in $I$, ${\cal F}(\omega)$ is strictly monotonous in $I$: strictly increasing if $\omega \, \varepsilon(\omega) > 0$ in $I$, strictly decreasing if $\omega \, \varepsilon(\omega) < 0$ in $I$. As a consequence,  if ${\cal F}$ admits a (strict) local extremum at $\omega = \omega_{ext}$, ${\cal F}(\omega_{ext}) \leq 0$ and the number of points inside $I$ at which ${\cal F}$ changes its sign is at most equal to 2. 
	\end{lema}
\begin{proof} If ${\cal F}(\omega) > 0$ , $\omega \, \mu(\omega)$ and $\omega \, \varepsilon(\omega)$
are non-zero real numbers with the same sign. Due to (\ref{growingproperty}) and (\ref{defD}),  ${\cal D}(\omega)= {\cal F}'(\omega)$ has the same sign that $\omega \, \mu(\omega)$ and $\omega \, \varepsilon(\omega)$.\\[12pt]
If ${\cal F}$ attained a local extremum at a point $\omega_{ext}$ where ${\cal F}(\omega_{ext}) > 0$, there would exist a neighborhood of $\omega_{ext}$ in which ${\cal F}$ would be positive and non-monotonous, which would contradict the first part of the lemma. If ${\cal F}$ had three changes of sign inside $I$, it would attain a local maximum at a point where it is positive (see figure \ref{Fig1}), which would contradict the previous result.
\end{proof}
\begin{figure}[h] 
\centerline{
\includegraphics[width=10cm]{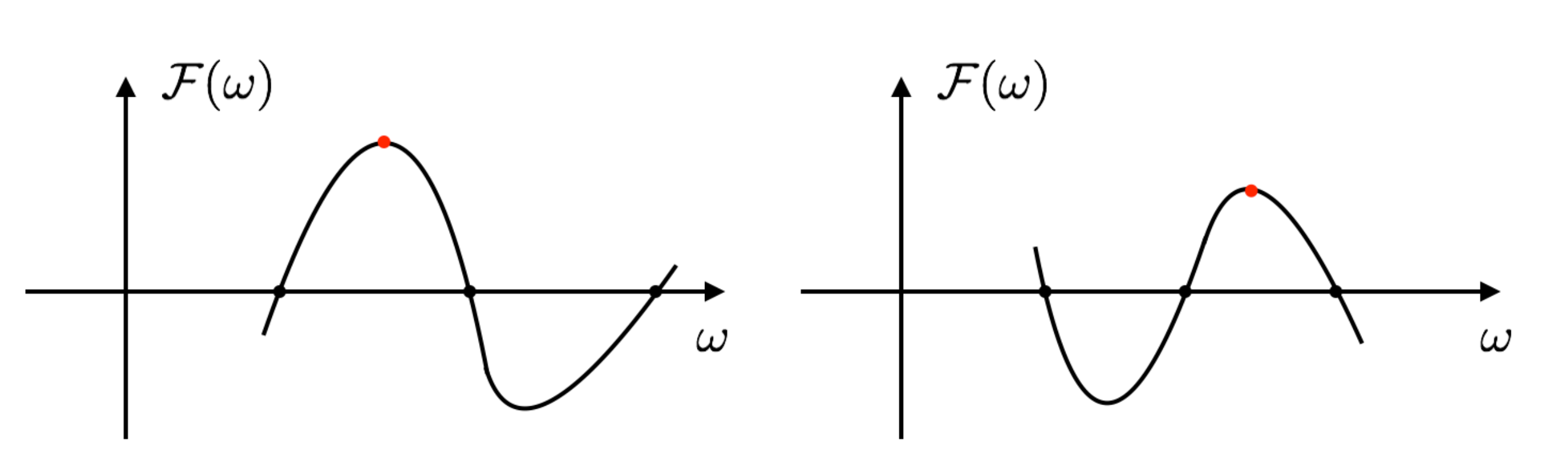}}
\caption{Impossibility of three sign changes: the 'bad' maxima are indicated by the red points}
\label{Fig1}
\end{figure}

\begin{rem}\label{rem.zero}
If $I$ is an interval which does not intersect $\cal P$ then by lemma \ref{propF}, $\cal F$ can admit at most two zeros (counted with their multiplicity) inside $I$.
\end{rem}

\subsubsection{General results. Spectral bands and gaps.} \label{sec-genemodes}
\begin{lema}\label{simplicity}
Assume (\ref{hypononzero}). All solutions $\omega$ of (\ref{relationdispersion}) for ${\bf k} \neq 0$ are simple and non-zero.
\end{lema}
\begin{proof} The fact that $\omega \neq 0$ follows from (\ref{hypononzero}) which implies ${\cal F} = 0$. The fact that it is a simple root, i. e. that ${\cal D}(\omega) \neq 0$, is a consequence of lemma \ref{lemmetechnique} since ${\cal F}(\omega)= {|{\bf k}|^2}$ implies ${\cal F}(\omega)>0$ . 
\end{proof}
\begin{coro} \label{smoothness} The set $\Omega\big({\bf k})$ of the solutions of the dispersion relation can be labeled as follows
\begin{equation} \label{repOmegak}
\Omega\big({\bf k}) = \{ \pm \; \omega_\ell\big(|{\bf k}|\big), 0 \leq \ell \leq N + 1 \}, \quad 0 < \omega_0\big(|{\bf k}|\big) < \omega_1\big(|{\bf k}|\big) < \cdots < \omega_{N+1}\big(|{\bf k}|\big).	
\end{equation} 
Moreover, each function $|{\bf k}| \mapsto \omega_\ell\big(|{\bf k}|\big)$ is analytic and strictly monotonous. 
\end{coro}  
\begin{proof}
(\ref{repOmegak}) immediately follows from 	lemma \ref{simplicity} together with the evenness of $\varepsilon$ and $\mu$. The analytic smoothness then follows from the simplicity of the solutions $\omega_\ell\big(|{\bf k}|\big)$ at any $|{\bf k}| > 0$ (use for instance the implicit function theorem). To prove the strict monotonicity of it suffices to notice that
$$
{\cal F} \big( \, \omega_\ell\big(|{\bf k}| \big) \, \big)  =  |{\bf k}|^2 \quad
\Longrightarrow \quad {\cal D} \big( \, \omega_\ell\big(|{\bf k}| \big) \, \big)\, \omega_\ell'\big(|{\bf k}|\big) = 2 \, |{\bf k}|
$$
which implies in particular that $\omega_\ell'\big(|{\bf k}|\big) \neq 0$ (as well as ${\cal D} \big( \, \omega_\ell\big(|{\bf k}| \big) \, \big) \neq 0$) for any $|{\bf k}| \neq 0$. 
\end{proof}
\noindent In what follows we shall define the {\bf spectrum of the medium} as the closure of the set of {\bf propagative frequencies}, i.e. as the closure of the set of frequencies at which there exists a propagative 
plane wave, in other words,
\begin{equation} \label{spectrum}
	{\cal S} = \mbox{closure } \{ \omega \in \R \setminus {\cal P} \; / \; {\cal F}(\omega) \geq 0 \}
\end{equation}
that can be rewritten as a finite union of closed intervals, called \textbf{spectral bands}: 
{\begin{equation} \label{eq:bands}
	{\cal S} = \bigcup_{\ell = 0}^{N+1} \pm \; {\cal B}_\ell, \quad  {\cal B}_\ell := \mbox{closure } \big\{ \; \omega_\ell\big(|{\bf k}| \big), |{\bf k}| \in [0, +\infty[ \; \big\}.
\end{equation}
The term spectrum is justified: in section \ref{Schrodinger} $\cal S$ appears as a spectrum of a certain self-adjoint operator.
\begin{lema} \label{lemma:bands} Two distinct spectral bands cannot overlap: the intersection of two bands is either empty or reduced to one of their extremities. Furthermore, the function $\omega_{N+1}\big(|{\bf k}| \big)$ is strictly increasing and the corresponding  band  $B_{N+1}$ is of the form $[z^*, +\infty[$ where $z^*$ is the largest positive zero of ${\cal F}(\omega)$.
\end{lema}
\begin{proof} First assume that there exists $m \neq \ell$ such that ${\cal B}_\ell$  and ${\cal B}_m$ overlap. Then, there would exist
	 non-zero ${\bf k}$ and ${\bf k}'$ such that  $\omega_\ell \big(|{\bf k}| \big) = \omega_m\big(|{\bf k}'| \big) \equiv \omega \in \R^+$.
	  Since ${\cal F}\big( \, \omega_\ell\big(|{\bf k}| \big)\, \big) = |{\bf k}|^2$ and ${\cal F}\big( \, \omega_m\big(|{\bf k}'| \big)\, \big) = |{\bf k}'|^2$, we deduce $|{\bf k}| = |{\bf k}'|$, which is impossible since $\ell \neq m$ (cf. (\ref{repOmegak})). \\[12pt]
The fact that ${\cal F}(\omega) \sim c_0^2 \, \omega^2$ when $\omega$ tends to $+ \infty$ shows that the image of the function $\omega_{N+1}\big(|{\bf k}| \big)$ is an interval of the form $[z^*, +\infty[$. By a contradiction argument, this shows that 
$\omega_{N+1}\big(|{\bf k}| \big)$ is strictly increasing. Therefore, $z^* = \omega_{N+1}(0)$ which implies ${\cal F}(z^*)$ = 0 and ${\cal F}(\omega) > 0 $ for $\omega > z^*$.
\end{proof} 
\noindent The set ${\cal G}$ of \textbf{non-propagative frequencies} is the open subset of $\R$ defined as:
\begin{equation} \label{spectrum}
{\cal G} = \R \setminus {\cal S}.
\end{equation}
From lemma \ref{lemma:bands}, we deduce that ${\cal G}$ is a finite union of open bounded intervals, called {\bf spectral gaps}.
\subsubsection{Description of dispersion curves} \label{sec-dispersioncurves}
To go further in the description of the spectral bands and dispersion curves $|{\bf k}| \rightarrow \omega_\ell\big(|{\bf k}| \big)$, it is useful to rename the positive poles of $\cal F$ as follows (double poles are not repeated)
$$
{\cal P}^+ := {\cal P} \cap \R_+^* = \{ 0 < p_1 < p_2 < \cdots < p_{N_d}\}, 
$$
and to introduce the disjoint intervals 
$$
I_0 = [0, p_1),  \; I_1 = (p_1, p_2),\ldots, \; I_q = (p_q, p_{q+1}), \cdots, \; I_{N_d} = (p_{N_d}, + \infty)
$$
then we will look for the solutions $\omega$ belonging to every $I_q, q = 0, \cdots, N_d$. We provide below a detailed explanation of our statements and in figures \ref{FigS1} - \ref{FigS2}, the illustrations to the explanations.\\[12pt]
{\bf Resolution of (\ref{relationdispersion}) in $I_0$}.\\[12pt]
Let us remind that $I_0 = [0, p_1)$.  Note that ${\cal F}(\omega) \sim \varepsilon(0) \, \mu(0) \, \omega^2$  with $\varepsilon(0) \, \mu(0) > 0$ (see remark \ref{positivity0}). Therefore, ${\cal F}(\omega)$ is increasing for small $\omega$. Since it cannot have a positive local maximum inside $I_0$ (cf lemma \ref{lemmetechnique}),  ${\cal F}(\omega)$ is increasing in $I_1$ (see figure \ref{FigS0}). Thus the whole interval $[0,p_1]$ coincides with the first spectral band ${\cal B}_0$, the function $\omega_0\big(|{\bf k}| \big)$ is strictly increasing and satisfies
$$
\omega_0\big(0 \big) = 0, \quad \dsp \lim_{|{\bf k}| \rightarrow + \infty} \omega_0\big(|{\bf k}| \big) = p_{1}.
$$
	\begin{figure}[h] 
	\centerline{
	\includegraphics[height=4cm]{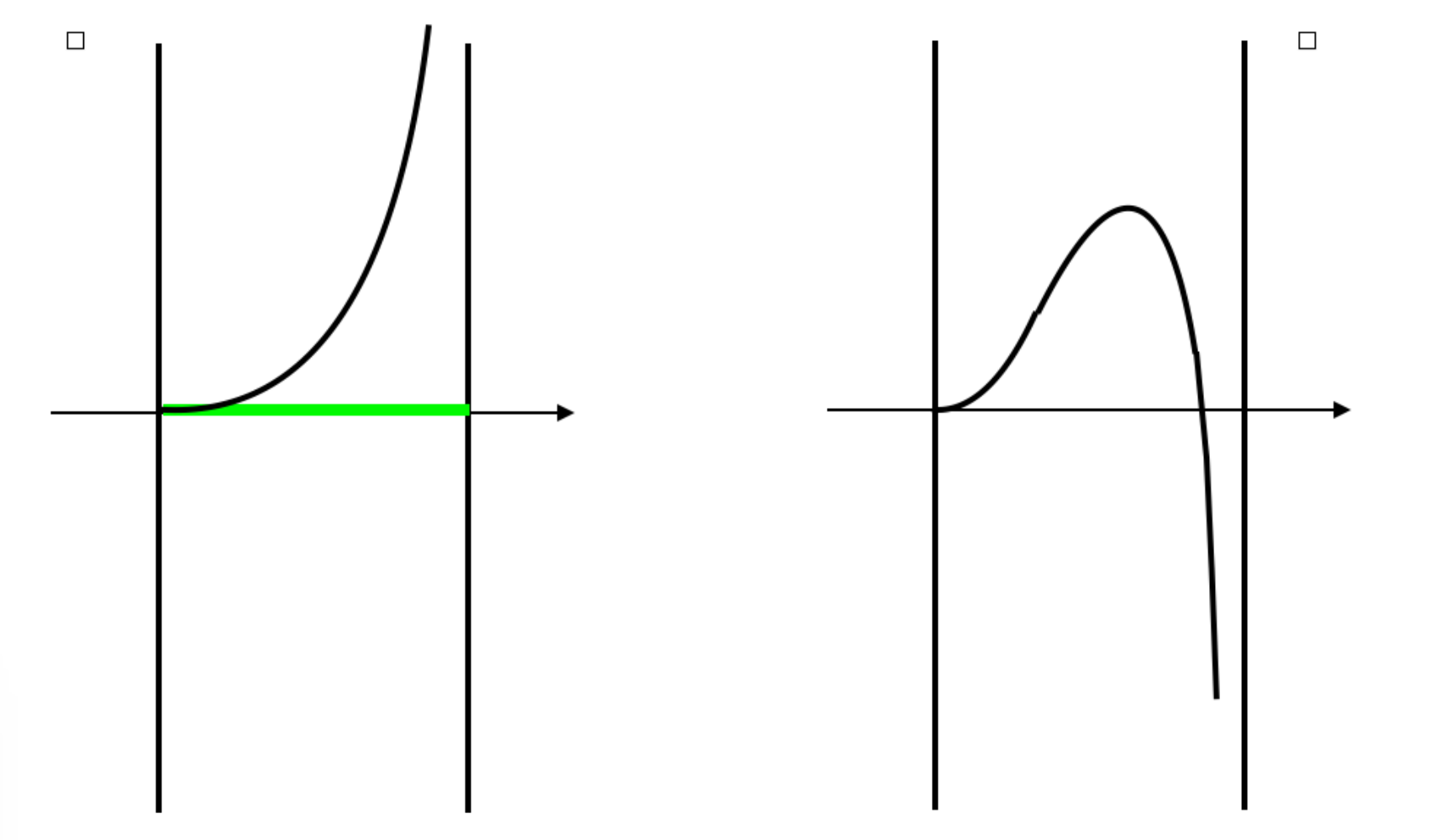}}
	\caption{Left: possible graph for $\omega \rightarrow {\cal F}(\omega)$ inside $I_1$ (a full spectral band, in green).
	Right: impossible scenario.}
	\label{FigS0}
	\end{figure}
{\bf Resolution of (\ref{relationdispersion}) in $I_q$ when $q \in \{1, \cdots, N_d-1\}$}.\\[12pt]
Let us write $\{1, \cdots, N_d-1\} = Q_- \cup Q_0 \cup Q_+$ where the disjoints sets $Q_- ,Q_0,Q_+$ are defined by 
$$
\left\{\begin{array}{l}
\dsp	q \in Q_- \quad \Longleftrightarrow \quad \lim_{\omega \rightarrow p_q^+} {\cal F}(\omega) = \lim_{\omega \rightarrow p_{q+1}^-} {\cal F}(\omega) =  - \infty \\[12pt]
\dsp	q \in Q_+ \quad \Longleftrightarrow \quad \lim_{\omega \rightarrow p_q^+} {\cal F}(\omega) = \lim_{\omega \rightarrow p_{q+1}^-} {\cal F}(\omega) = + \infty \\[12pt]
\dsp	q \in Q_0 \quad \Longleftrightarrow \quad \lim_{\omega \rightarrow p_q^+} {\cal F}(\omega) =\pm \, \infty \mbox{ and }  \lim_{\omega \rightarrow p_{q+1}^-} {\cal F}(\omega) = \mp \, \infty
	\end{array} \right.
	$$
Let us then distinguish three different cases:
\begin{itemize}
	\item[(i)] $ q \in Q_-$. In that case, we claim that (\ref{relationdispersion}) has no solution for ${\bf k} \neq 0$, in other words, that $I_q \subset {\cal G}$ or that $I_q\cap \cal S$ is a singleton. Indeed, by lemma \ref{lemmetechnique}, the maximum value  of ${\cal F}$ inside $I_q$ is non-positive (figure \ref{FigS01} (b)). Moreover, if it is zero, it corresponds to a double zero of $\cal{F}$, see remark \ref{rem.zero}. 
	\item[(ii)] $ q \in Q_0$. The number of sign changes of ${\cal F}$ inside $I_q$ is necessarily odd, hence equals to 1 by lemma \ref{lemmetechnique}. Thus, there exists a single (simple) zero $z_q \in I_q$ of ${\cal F}$ such that one of the following holds: \\[12pt]
	(ii.1) Either ${\cal F}$ is negative in $]p_q, z_q[$ and positive in $]z_q, p_{q+1}[$. This means that $]p_q, z_q[ \, \subset {\cal G}$ and, using lemma \ref{lemmetechnique}, that ${\cal F}$ is a strictly increasing bijection from $[z_q, p_{q+1})$ onto $[0, + \infty)$ (figure \ref{FigS01} (c)). Thus, inside $I_q$, equation (\ref{relationdispersion}) admits a unique branch of solutions $\omega = \omega_{\ell_q}\big(|{\bf k}| \big)$, defined from the inverse of the above bijection. This solution satisfies:
		$$
|{\bf k}| \longrightarrow \omega_{\ell_q}\big(|{\bf k}| \big) \mbox{ is strictly increasing,} \quad 	\omega_{\ell_q}\big(0 \big) = z_q, \quad \lim_{|{\bf k}| \rightarrow + \infty} \omega_{\ell_q}\big(|{\bf k}| \big) = p_{q+1}.
		$$
(ii.2) Either ${\cal F}$ is positive in $]p_q, z_q[$ and negative in $]z_q, p_{q+1}[$ (see figure \ref{Fig1} (c)). This means that $]z_q, p_{q+1}[ \, \subset {\cal G}$ and, using lemma \ref{lemmetechnique}, that ${\cal F}$ is a strictly decreasing bijection from $(p_q, z_q]$ onto $[0, + \infty)$ (figure \ref{FigS01}(d)). Hence, inside $I_q$, (\ref{relationdispersion}) admits a unique branch of solutions $\omega = \omega_{\ell_q}$, such that:
		$$
|{\bf k}| \longrightarrow \omega_{\ell_q}\big(|{\bf k}| \big) \mbox{ is strictly decreasing,} \quad 	\omega_{\ell_q}\big(0 \big) = z_q, \quad \lim_{|{\bf k}| \rightarrow + \infty} \omega_{\ell_q}\big(|{\bf k}| \big) = p_{q}.
		$$
In each of the above cases, the interval $I_q$ contains only one spectral band ${\cal B}_{\ell_q} \equiv [p_q, z_q]$ or $[z_q, p_{q+1}]$.
	\item[(iii)] $ q \in Q_+$. By lemma \ref{lemmetechnique}, the minimum value ${\cal F}_*(q)$ of ${\cal F}$ inside $I_q$ is non-positive (see figure \ref{FigS02}(a)). \\[12pt]
Assume that ${\cal F}_*(q) < 0$. By lemma \ref{lemmetechnique} again, ${\cal F}$ changes sign twice inside $I_q$ (figure \ref{FigS01}(a)): there exist two (simple) zeros $\{z_q^-, z_q^+\}$ of ${\cal F}$ such that $p_q < z_q^- < z_q^+ < p_{q+1}$ such that ${\cal F}$ is negative in $(z_q^-, z_q^+)$ (in other words $(z_q^-, z_q^+)$ is a particular band gap), ${\cal F}$ is a strictly decreasing bijection from $(p_q, z_q^-]$ onto $[0, + \infty)$ and ${\cal F}$ is a strictly increasing bijection from $[z_q^+, p_{q+1})$ onto $[0, + \infty)$. Thus, inside $I_q$, (\ref{relationdispersion}) admits two branches of solutions, $\omega = \omega_{\ell_q}$ and $\omega = \omega_{\ell_q+1}$ such that
		$$
		\begin{array}{lll}
\dsp |{\bf k}| \longrightarrow \omega_{\ell_q}\big(|{\bf k}| \big) \mbox{ is strictly decreasing,} & \quad 	\omega_{\ell_q}\big(0 \big) = z_q^-, & \quad \dsp \lim_{|{\bf k}| \rightarrow + \infty} \omega_{\ell_q}\big(|{\bf k}| \big) = p_{q}, \\[12pt]
\dsp |{\bf k}| \longrightarrow \omega_{\ell_q+1}\big(|{\bf k}| \big) \mbox{ is strictly increasing,} & \quad \omega_{\ell_q}\big(0 \big) = z_q^+, & \quad \dsp \lim_{|{\bf k}| \rightarrow + \infty} \omega_{\ell_q}\big(|{\bf k}| \big) = p_{q+1},
\end{array}		$$	
If ${\cal F}_*(q) = 0$, we are in a limit situation when $z_q^-= z_q^+ \equiv z_q $ is a double zero of ${\cal F}$. The situation is similar to the previous case, but there is no spectral gap inside $I_q$. In this interval, (\ref{relationdispersion}) still admits two branches of solution 
$\omega_{\ell_q}$ (decreasing) $\omega = \omega_{\ell_q+1}$ (increasing). Moreover, ${\cal B}_{\ell_q}\cap {\cal B}_{\ell_q+1}=\{z_q\}$.
	\end{itemize}
	\begin{figure}[h] 
	\centerline{
	\includegraphics[height=4cm]{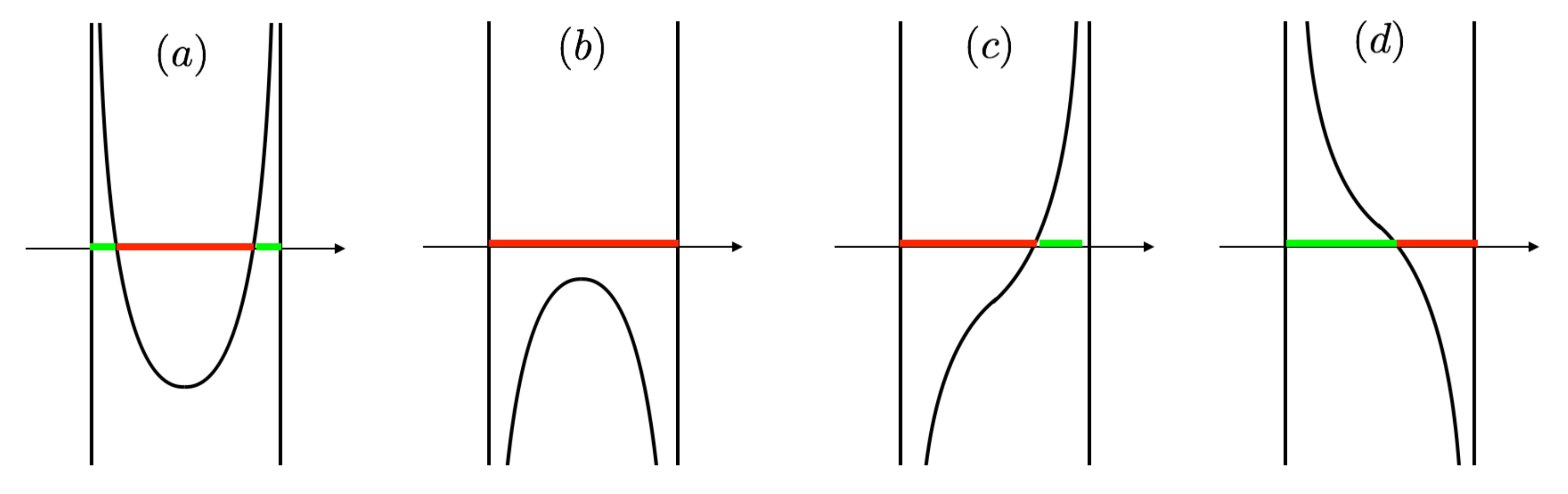}}
	\caption{Possible graphs for $\omega \rightarrow {\cal F}(\omega)$ inside $I_\ell, 1 \leq \ell \le N_d$.
	Red segments are band gaps, green segments are spectral bands.}
	\label{FigS01}
	\end{figure}
	\begin{figure}[h] 
	\centerline{
	\includegraphics[height=4cm]{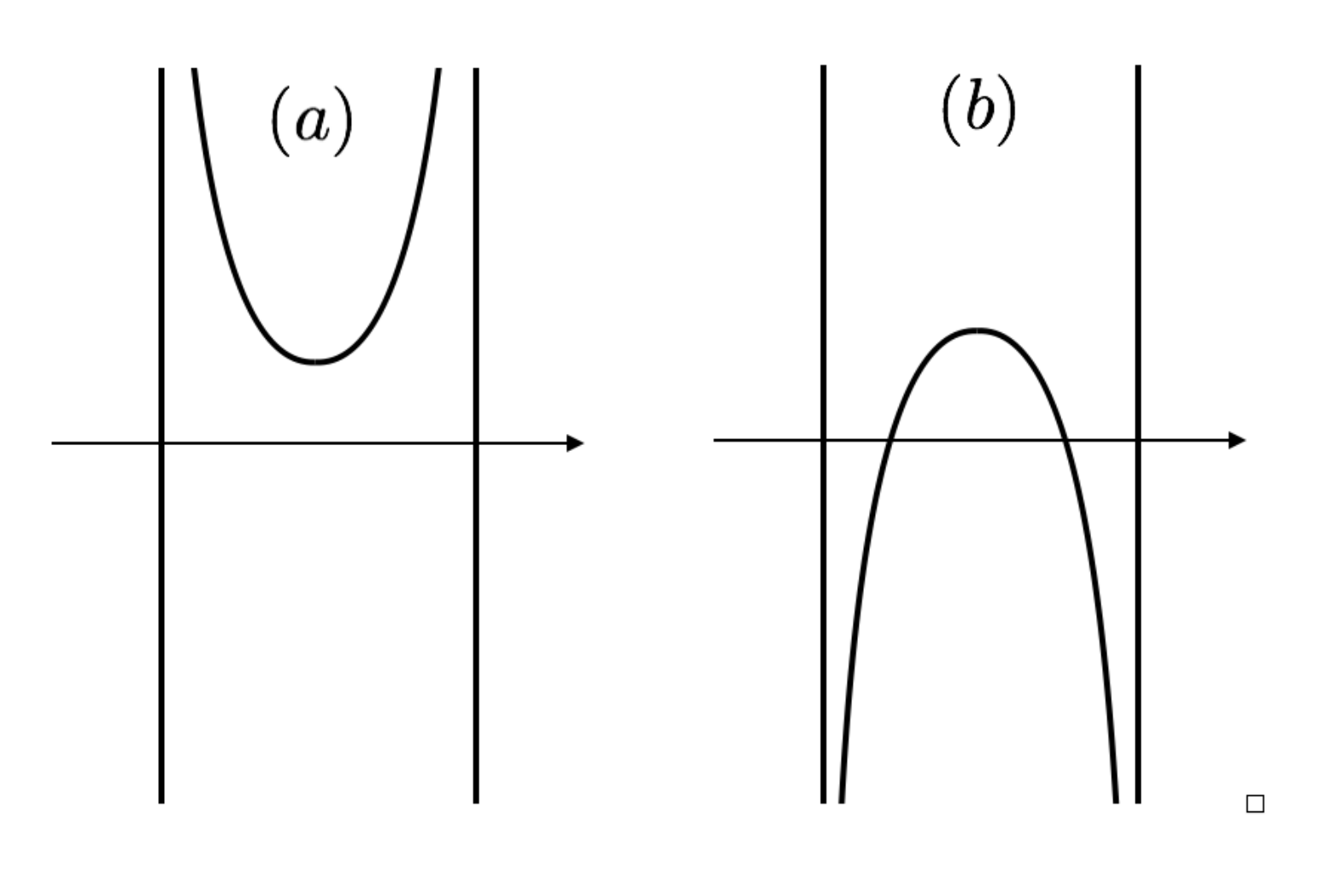}}
	\caption{Some impossible graphs for $\omega \rightarrow {\cal F}(\omega)$ inside $I_\ell, 1 \leq \ell \le N_d$.
	}
	\label{FigS02}
	\end{figure}
{\bf Resolution of (\ref{relationdispersion}) in $I_{N_d}$}.\\[12pt]
This case was partially treated by lemma \ref{lemma:bands}. The only scenarios for ${\cal F}(\omega)$ are the following:
\begin{itemize}
\item[(i)] Either $\dsp \lim_{\omega \rightarrow p_{N_d}^+} {\cal F}(\omega) = - \infty$. In this case,  ${\cal F}(\omega)$ changes sign exactly once inside $I_{N_d}$ (figure \ref{FigS1}(a)).\\[4pt] Thus, there exists a single  zero $z_{N_d} > p_{N_d}$ such that ${\cal F}(\omega)$ is negative in $(p_{N_d},z_{N_d})$ (this interval is a spectral gap) and positive in $(z_{N_d}, + \infty)$. Hence, ${\cal F}(\omega)$ is a strictly increasing bijection  from $(z_{N_d}, + \infty)$ onto $[0, +\infty)$. The corresponding  branch of solutions of (\ref{relationdispersion}), $\omega_{N+1}\big(|{\bf k}| \big)$, satisfies
$$
|{\bf k}| \longrightarrow \omega_{N+1}\big(|{\bf k}| \big) \mbox{ is strictly increasing,} \quad 	\omega_{N+1}\big(0 \big) = z_{N_d} , \quad  \omega_{N+1}\big(|{\bf k}| \big) \sim c_0 \, |{\bf k}|, \; (|{\bf k}| \rightarrow + \infty).
$$
\item[(ii)] Either $\dsp \lim_{\omega \rightarrow p_{N_d}^{+}} {\cal F}(\omega) = + \infty$. We are then in a situation similar to the case $q \in Q^+$ (figure \ref{FigS1}(b)). 
Let us denote ${\cal F}_*(N_d) \leq 0 $ the minimum value of ${\cal F}$ in $I_{N_d}$. 
\\[12pt] If ${\cal F}_*(N_d) < 0$, there exists two (simple) zeros $\{z_{N_d}^-, z_{N_d}^+\}$ of ${\cal F}$ with $p_{N_d} < z_{N_d}^- < z_{N_d}^+$ such that ${\cal F}$ is negative in $(z_{N_d}^-, z_{N_d}^+)$ (i.e. $(z_q^-, z_q^+)$ is a band gap), ${\cal F}$ is a strictly decreasing bijection from $(p_{N_d}, z_{N_d}^-]$ onto $[0, + \infty)$ and a strictly increasing bijection from $[z_{N_d}^+, + \infty )$ onto $[0, + \infty)$. Hence, inside $I_{N_d}$, (\ref{relationdispersion}) admits  two branches of solutions (the last two ones), $\omega_{N}$ and $\omega_{N+1}$ such that
		$$
		\begin{array}{lll}
\dsp |{\bf k}| \longrightarrow \omega_{N}\big(|{\bf k}| \big) \mbox{ is strictly decreasing,}  \quad 	\omega_{N}\big(0 \big) = z_{N_d}^-,  \quad \dsp \lim_{|{\bf k}| \rightarrow + \infty} \omega_{N} \big(|{\bf k}| \big) = p_{N_d}, \\[12pt]
\dsp |{\bf k}| \longrightarrow \omega_{N+1}\big(|{\bf k}| \big) \mbox{ is strictly increasing,}  \quad \omega_{N+1}\big(0 \big) = z_{N_d} , \quad  \omega_{N+1}\big(|{\bf k}| \big) \sim c_0 \, |{\bf k}| \mbox{ at } \infty ,
\end{array}		
$$
If ${\cal F}_*(N_d) = 0$, we are in a limit situation where $z_{N_d}^-, z_{N_d}^+ \equiv z_ {N_d}$ is a double zero of ${\cal F}$. The situation is similar to the previous case, but there is no spectral gap inside $I_q$:  ${\cal B}_{N}\cap{\cal B}_{N+1}=\{z_ {N_d}\}$.
	\begin{figure}[h] 
	\centerline{
	\includegraphics[height=4cm]{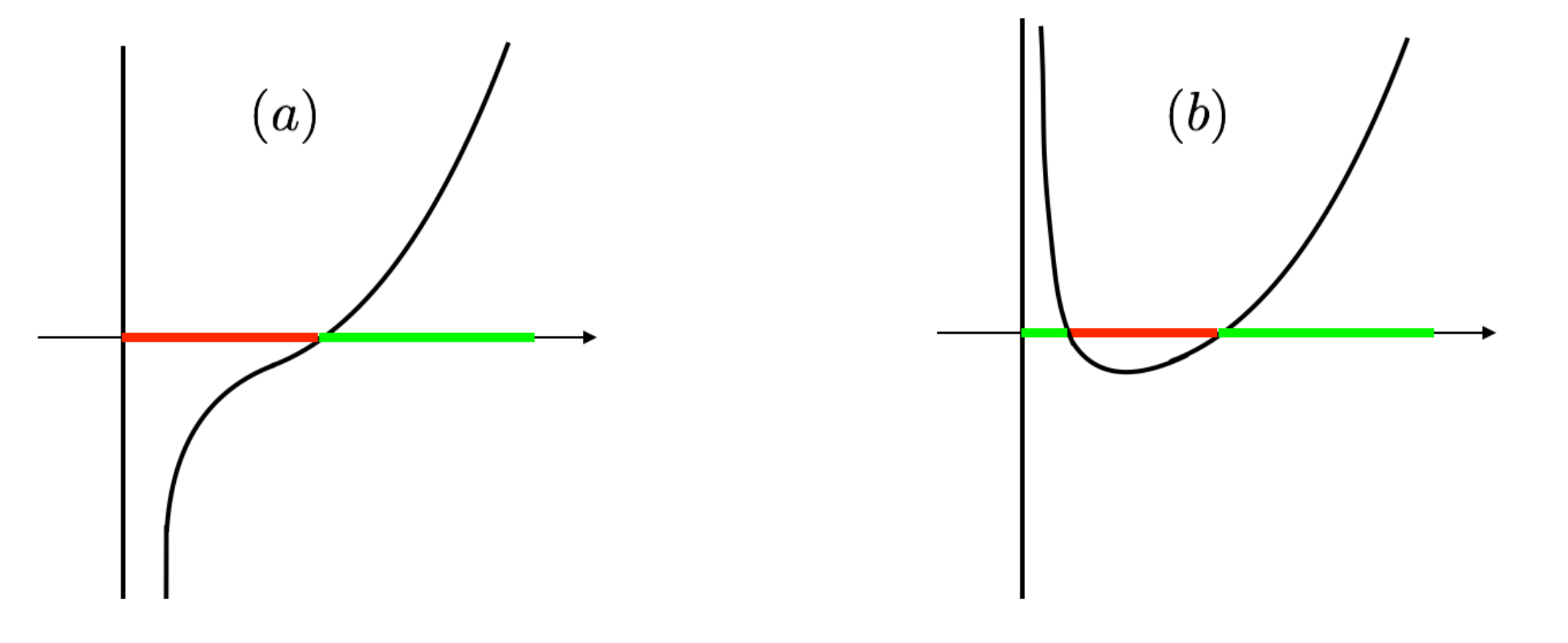}}
	\caption{Possible graphs for $\omega \rightarrow {\cal F}(\omega)$ inside $I_{N_d}$.
	Red segments are band gaps, green segments are spectral bands.}
	\label{FigS1}
	\end{figure}
	\begin{figure}[h] 
	\centerline{
	\includegraphics[height=4cm]{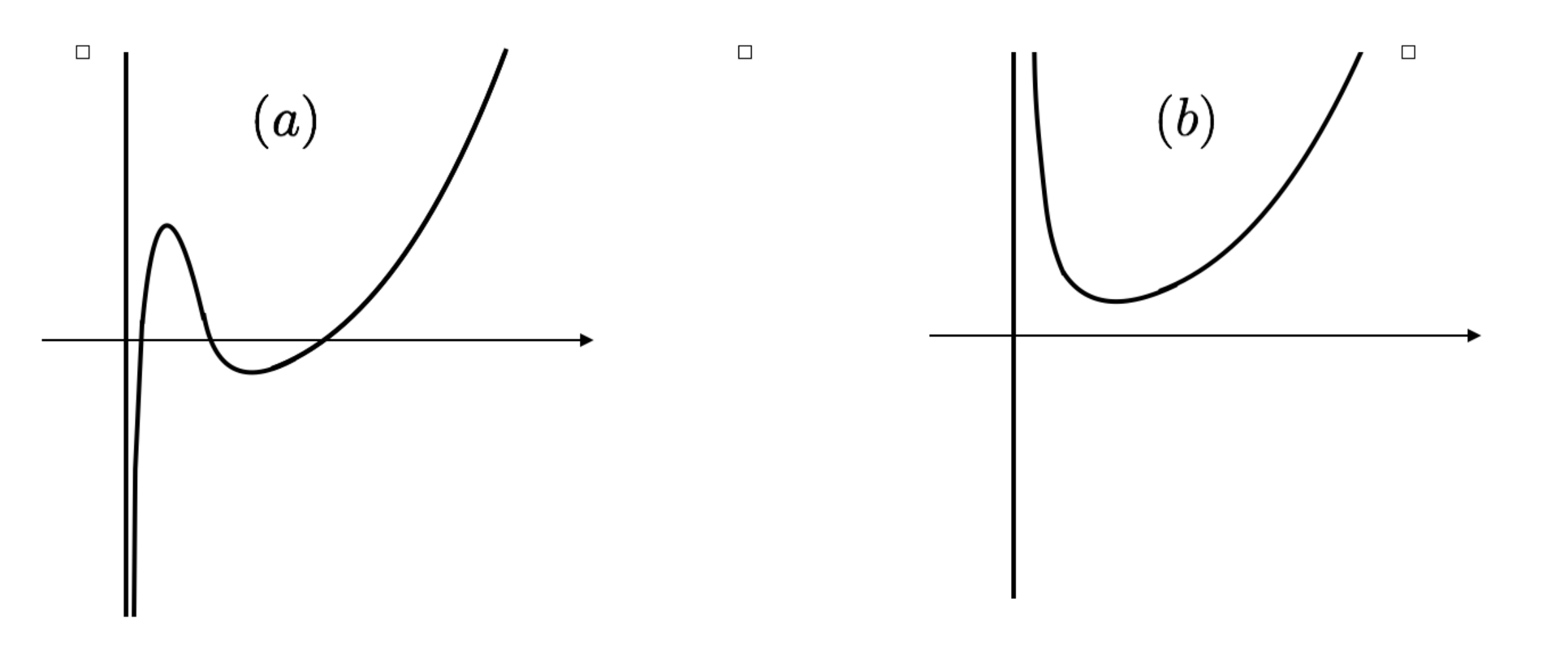}}
	\caption{Some impossible graphs for $\omega \rightarrow {\cal F}(\omega)$ inside $I_{N_d}$.
	}
	\label{FigS2}
	\end{figure}
\end{itemize}
\begin{rem} \label{othercases} There is a clear similarity between the spectral analysis of dispersive local materials and
	the spectral analysis of periodic media \cite{Kuchment}, especially in the 1D case \cite{Eastham}. The main difference is that, in the latter case, there is a countable infinity of spectral bands (and in most cases, a countable infinity of spectral bands) and that these
bands systematically alternate as positive / negative / positive / $\cdots$
	\end{rem}
\begin{rem} \label{othercases} If (\ref{hypononzero}) is not satisfied, it is easy to verify that most of the above results remain true. The only change concerns the first spectral band: the lower bound of ${\cal B}_0$ can be positive (when only one of the functions $\varepsilon$ or $\mu$ admits $0$ as a pole) and the first mode $\omega_0$ does not need to be increasing (cf. Drude media).
	\end{rem}
\subsubsection{Forward and backward modes. Negative index.} \label{sec-Negative}
According to the previous section, the modes $ \pm\, \omega_\ell({\bf k})$ of a non-dissipative local material can be split into two categories:
\begin{itemize} 
	 \item the {\bf forward} (or {\bf direct}) modes for which  $\omega'_\ell\big(|{\bf k}|\big) \, \omega_\ell\big(|{\bf k}|\big) > 0$ for any $|{\bf k}| >0$, i. e. for which the {\bf phase velocity} and  the {\bf group velocity} have the same sign, 
	\item the {\bf backward} (or {\bf inverse}) modes for which  $\omega'_\ell\big(|{\bf k}|\big) \, \omega_\ell\big(|{\bf k}|\big) < 0$ for any $|{\bf k}| >0$, i. e. for which the {\bf phase velocity} and  the {\bf group velocity} have opposite signs. 
\end{itemize}
\begin{rem} \label{remgroup} For 3D linear wave propagation, the phase and group velocities associated with a family of (propagative) plane waves obeying a dispersion relation $\omega = \omega({\bf k})$ (where $\omega(\cdot)$ is a smooth real-valued function in $\R^3$), the phase and group velocities are defined as vector fields, namely, $\omega({\bf k}) \, {\bf k} / |{\bf k}|^2$ and $\nabla_{\mathbf{k}} \omega({\bf k})$. For isotropic media (studied in this work), $\omega(\cdot)$ is a function of $|{\bf k}|$, and the phase and group velocities are thus proportional to ${\bf k}$. Thus phase and group velocity can be viewed as a scalar quantities.
	\end{rem}
{\bf Example.} Let us consider the following Lorentz model
\begin{equation} \label{Lorentz1}
	\varepsilon(\omega) = \frac{\omega^2 - 16}{\omega^2 - 1}, \quad  \mu(\omega) = \frac{\omega^2 - 25}{\omega^2 - 4},
	\end{equation}
In figure \ref{FigBands} we show the corresponding modes (computed numerically). There are 3 modes, corresponding to 3 spectral bands and 2 band gaps. The 1st and 3rd modes are forward, the 2nd one is backward.
	\begin{figure}[h] 
	\centerline{
	\includegraphics[height=4cm]{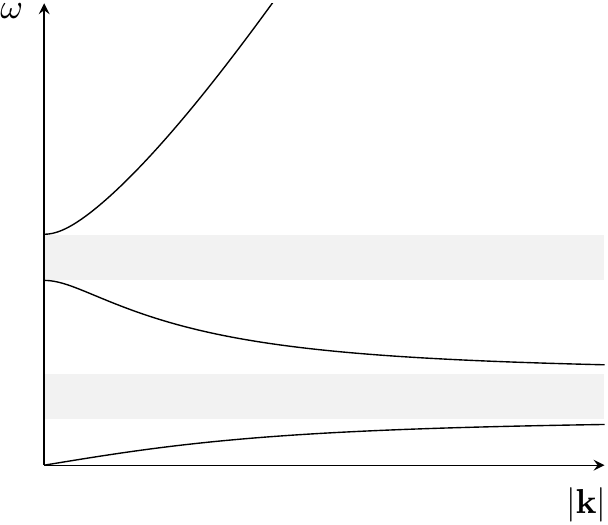}}
	\caption{Dispersion curves for the model (\ref{Lorentz1}). Band gaps are in grey.
	}
	\label{FigBands}
	\end{figure}
~\\[12pt]
In the following we shall denote by ${\cal I}_f$ the set of indices (always non-empty) $\ell \in \{0, \cdots, N+1\}$ corresponding to forward modes $\pm \, \omega_\ell$ and by ${\cal I}_f$ the set of indices $\ell \in \{0, \cdots, N+1\}$ corresponding to forward modes. We can split accordingly the spectrum ${\cal S}$ of the material as
$$
{\cal S} = {\cal S}_f\cup {\cal S}_b, \quad  {\cal S}_f= \bigcup_{j \in {\cal I}_f} {\cal B}_\ell, \quad  {\cal S}_b= \bigcup_{\ell \in {\cal I}_b} {\cal B}_\ell,
$$
where ${\cal S}_n$ is by definition the set of {\bf forward  frequencies} and ${\cal S}_b$ is by definition the set of {\bf backward frequencies}. The following result gives a simple characterization of the two sets.
\begin{theo} 
For a non-dissipative local material one has the characterization
\begin{equation} \label{charND}	
 {\cal S}_f= \mbox{closure}\{ \omega \in \R \setminus \mathcal{P} \; / \; \omega \, {\cal D}(\omega) > 0 \}, \qquad  {\cal S}_b = \mbox{closure}\{ \omega \in \R \setminus \mathcal{P} \; / \; \omega \, {\cal D}(\omega) < 0 \}.
	\end{equation}
If, moreover, the material is passive (which is always true up to equivalence), then
\begin{equation} \label{charNDbis}	
\left\{ \begin{array}{l}
{\cal S}_f = \mbox{closure}\{ \omega \in \R \setminus \mathcal{P}\; / \; \varepsilon(\omega) > 0 \mbox{ and } \mu(\omega) > 0 \}, \\[12pt] {\cal S}_b = \mbox{closure}\{ \omega \in \R \setminus \mathcal{P} \; / \; \varepsilon(\omega) < 0 \mbox{ and } \mu(\omega) < 0 \}.
\end{array} \right.
\end{equation}
\end{theo}
\begin{proof} The first characterization follows from the following formula for the group velocity associated with the mode $\omega_\ell\big(|{\bf k}|\big)$ (see the proof of corollary \ref{smoothness} ):
	\begin{equation} \label{groupvelocity}
		\big[ \, \omega_\ell\big(|{\bf k}| \big) \, {\cal D} \big( \, \omega_\ell\big(|{\bf k}| \big) \, \big)\big] \, \big[ \, \omega_\ell\big(|{\bf k}| \big) \, \omega_\ell'\big(|{\bf k}|\big)\big]  = 2 \, |{\bf k}| \, \omega_\ell\big(|{\bf k}| \big)^2.
	\end{equation}
The second part of the theorem follows from the observation (already done in the proof of lemma \ref{lemmetechnique}) that, because of the growing property for passive materials, 
$\omega \, {\cal D} (\omega)$ appears as a linear combination of $\omega^2 \, \varepsilon(\omega)$ and $\omega^2 \, \mu(\omega)$ with positive coefficients (see (\ref{defD})). Since inside ${\cal S}$, $\omega \,\varepsilon(\omega)$ and $\omega \, \mu(\omega)$ 
have the same sign, 
the sign of $\omega {\cal D} (\omega)$ corresponds to the common sign of $\varepsilon(\omega)$ and $ \mu(\omega)$. \end{proof}
\begin{rem} \label{remNegative} The second characterization is the one that is often used in the literature for defining backward frequencies or backward modes. However, rigorously speaking, it is valid only for passive materials. \end{rem}
\noindent Whereas forward modes always exist (the last band is always forward, see lemma \ref{lemma:bands}), backward modes may or may not exist (see for instance figure \ref{Fig5}). This justifies the following definition.
\begin{defi} ({\bf Negative index material}) A negative index material is a non-dissipative local material in which there exist backward modes, i.e. for which the set ${\cal I}_b$ (or, equivalently, the set ${\cal S}_b$) is non-empty.
\end{defi}
	\begin{figure}[h]
	\centerline{
	\includegraphics[height=4cm]{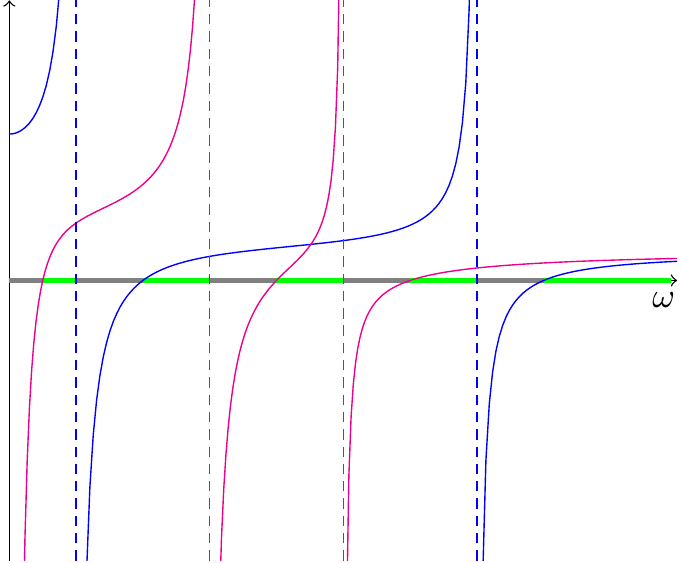} \hspace*{2cm}\includegraphics[height=4cm]{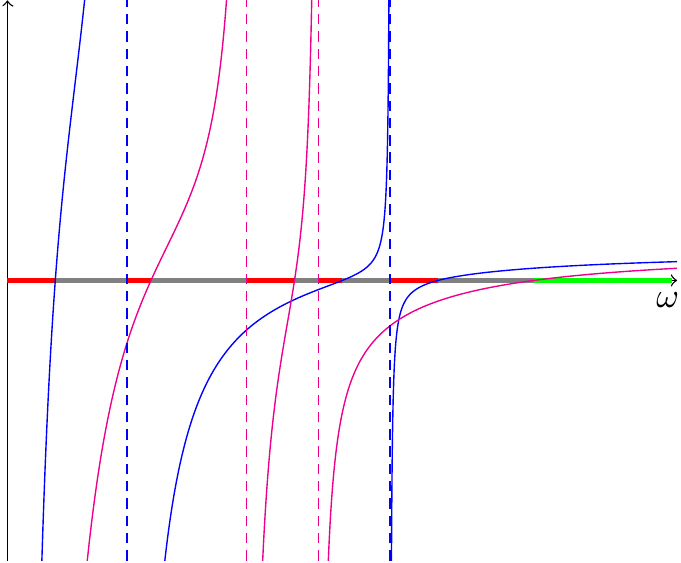}}
	\caption{Graphs of $\varepsilon(\omega)$ (blue) and $\mu(\omega)$ (magenta) for two different generalized Lorentz models with $N_e = N_m =2$. On the horizontal axis, grey segments represent band gaps, red segments backward bands and forward segments positive bands. For the left picture, all modes (there are 6 of them) are forward. For the right picture, all modes, except the last one, are backward.}
	\label{Fig5}
	\end{figure}
\subsection{Energy analysis of local passive materials.} \label{sec_EnergyLorentz}
\noindent The well-posedness and stability of (\ref{Lorentzsystem})  can be recovered with the help of energy techniques (which presents the advantage to be generalizable to variable coefficients).
\begin{theo} \label{thm.Energy} Any sufficiently smooth solution of  (\ref{Lorentzsystem}) satisfies the energy identity $\dsp \frac{d}{dt} \, {\cal E}_{tot}(t) = 0$, where
	\begin{equation} \label{energyidentity}
\left\{	\begin{array}{lll}
\dsp {\cal E}_{tot}(t) :=  {\cal E}(t) + \sum_{\ell = 0}^{N_e} {\cal E}_{e, \ell}(t) + \sum_{\ell = 0}^{N_m} {\cal E}_{m, \ell}(t), \quad \dsp {\cal E}(t) :=  \frac{1}{2}\int_{\R^3} \big( \varepsilon_0 \, |{\bf E}|^2  + \mu_0 \, |{\bf H}|^2 \, \big) \; d{\bf x}, \\[12pt]
\dsp  
{\cal E}_{e, \ell}(t) :=  \frac{\varepsilon_0}{2} \,\sum_{\ell = 0}^{N_e}  \; \int_{\R^3}  \Omega_{e,\ell}^2 \, \big(|\partial_t \bbP_\ell|^2  + \omega_{e, \ell}^2 \, |\bbP_\ell|^2 \big) \; d{\bf x}, \\[12pt] 
\dsp {\cal E}_{m, \ell}(t) :=   \frac{\mu_0}{2} \,\sum_{\ell = 0}^{N_m}  \; \int_{\R^3}  \Omega_{m,\ell}^2 \, \big(|\partial_t \bbM_\ell|^2  + \omega_{m, \ell}^2 \, |\bbM_\ell|^2 \big) \; d{\bf x}.
\end{array} \right.
	\end{equation}
		\end{theo}
\begin{proof} Using (\ref{Lorentzsystem}) (second line, first equation), we compute that
$$
	\int_{\R^3}  \partial_t {\bf P} \cdot {\bf E} \; d{\bf x} = \varepsilon_0 \, \sum_{\ell = 0}^{N_e} \int_{\R^3}  \Omega_{e,\ell}^2 \, \partial_t \bbP_\ell \cdot {\bf E} \; d{\bf x} =  \frac{\varepsilon_0}{2} \,\sum_{\ell = 0}^{N_e}  \; \frac{d}{dt} \int_{\R^3}  \Omega_{e,\ell}^2 \, \big(|\partial_t \bbP_\ell|^2  
+ \omega_{e, \ell}^2 \, |\bbP_\ell|^2 \big) \; d{\bf x}.
$$	
In the same way, we have $\int_{\R^3}  \partial_t {\bf M} \cdot {\bf H} \; d{\bf x} =\frac{d}{dt} \sum_{\ell = 0}^{N_m} {\cal E}_{m, \ell}(t)$. 
To conclude it suffices to substitute the above two equalities in (\ref{idEnergy}).
	\end{proof}
	\begin{rem} \label{remPassivityEnergy} The above theorem also permits us to recover the physical passivity of Lorentz media since
$$ {\cal E}(t) \leq {\cal E}(t) + \sum_{\ell = 0}^{N_e} {\cal E}_{e, \ell}(t) + \sum_{\ell = 0}^{N_m} {\cal E}_{m, \ell}(t) =  {\cal E}(0) + \sum_{\ell = 0}^{N_e} {\cal E}_{e, \ell}(0) + \sum_{\ell = 0}^{N_m} {\cal E}_{m, \ell}(0) = {\cal E}(0), 
		$$
the last equality resulting from the zero-initial conditions for the $\bbP_\ell$'s and $\bbM_\ell$'s.\end{rem}
\noindent As announced in remark \ref{decroissance}, ${\cal E}(t)$ is not a decreasing function of time in general. Let us consider the case of a Drude model (\ref{epsmuDrude}) with 
$\xi_m = 0$, i. e. $\mu=\mu_0$. Assume for simplicity that ${\bf H}_0 = 0$ and that ${\bf E}_0 \in L^ 2(\R^3)$ with $\mbox{div}\,{\bf E}_0 = 0$, so that at each time $t \geq 0$, $\mbox{div}\, {\bf E}(\cdot,t) = 0$. Then one easily checks that the electric field satisfies the (vectorial) Klein-Gordon equation  
\begin{align*}
\partial_{tt}{\bf E}- c_0^2 \, \Delta {\bf E}+\xi_e^2 \,{\bf E}=0, \quad {\bf E}(\cdot,0) = {\bf E}_0, \partial_t{\bf E}(\cdot,0) = 0.
\end{align*}
Using the Fourier transform in space ( ${\bf E}({\bf x},t) \rightarrow \widehat{\bf E}({\bf k},t) \; $ and ${\bf H}({\bf x},t) \rightarrow \widehat{\bf H}({\bf k},t) \; $), we obtain 
\begin{align*}
\widehat{\bf E}({\bf k},t)=\widehat{\bf E}_0({\bf k}) \; \cos\xi_e({\bf k})t, \quad \widehat{\bf H}({\bf k},t)=i \; \mu_0^{-1} \; \big({\bf k} \times \widehat{\bf E}_0({\bf k})\big) \; \frac{\sin\xi_e({\bf k}) t}{\xi_e({\bf k})},\quad  \xi_e({\bf k}):= \big(\xi_e^2 + c_0^ 2 \, |{\bf k}|^ 2 \big)^{\frac{1}{2}}.
\end{align*}
Using Plancherel's theorem, $\mu_0^{-1} = \varepsilon_0 \, c_0^2$ and $|{\bf k} \times \widehat{\bf E}_0({\bf k})| = |{\bf k}| \, |\widehat{\bf E}_0({\bf k})|$ (since ${\bf k} \cdot \widehat{\bf E}_0({\bf k}) = 0$), we get after some manipulations
\begin{align*}
\mathcal{E}(t)= \mathcal{E}(0) - \frac{\varepsilon_0 \, \xi_e^2}{2} \int_{\R} |\widehat{\bf E}_0({\bf k})|^2 \; \frac{(\sin\xi_e({\bf k}) t)^2}{\xi_e({\bf k})^2}  \; d\bf k.
\end{align*}
Thus, as soon as $\xi_e > 0$, one has the strict inequality $\mathcal{E}(t)< \mathcal{E}(0)$ for any $t > 0$. Moreover,
\begin{align*}
\mathcal{E}'(t)=- \frac{\varepsilon_0 \, \xi_e^2}{2} \int_{\R} |\widehat{\bf E}_0({\bf k})|^2 \; \frac{\sin 2 \, \xi_e({\bf k}) t}{\xi_e({\bf k})}  \; d\bf k.
\end{align*}
Next, we play with the initial field replacing $\widehat{\bf E}_0({\bf k})$ by $\widehat{\bf E}_0^\delta({\bf k}) := \widehat{\bf E}_0({\bf k} / \delta)$ where $\delta$ is devoted to be small : this corresponds to concentrating the Fourier transform of the initial data near ${\bf k} = 0$. Denoting $\mathcal{E}^\delta(t)$ the electromagnetic energy of the corresponding solution $({\bf E}^\delta, {\bf H}^\delta)$, we thus have
\begin{align*}
\mathcal{E}_\delta'(t)=- \frac{\varepsilon_0 \, \xi_e^2}{2} \int_{\R} |\widehat{\bf E}_0({\bf k}/\delta)|^2 \; \frac{\sin 2 \, \xi_e({\bf k}) t}{\xi_e({\bf k})}  \; \md{\bf k} 
\end{align*}
that is to say (using the change of variable ${\bf k} = \delta \boldsymbol{\xi}$), 
\begin{equation} \label{Eprime}
\mathcal{E}_\delta'(t) = - \frac{\varepsilon_0 \, \xi_e^2 \, \delta^3}{2} \; t \; \int_{\R} |\widehat{\bf E}_0(\boldsymbol{\xi})|^2 \; \Phi \big(\xi_e(\delta \boldsymbol{\xi}) t\big)  \; \md \boldsymbol{\xi} \quad \mbox{where } \Phi(x) := \frac{\sin2x}{x}.
\end{equation}
Writing $\Phi \big(\xi_e(\delta \boldsymbol{\xi})  t\big) = \Phi \big(\xi_e t) + \left[ \, \Phi \big(\xi_e(\delta \boldsymbol{\xi})  t\big) - \Phi \big(\xi_e t) \, \right]$, this can be rewritten as
\begin{equation} \label{Eprime2}
\mathcal{E}_\delta'(t)= - \frac{\varepsilon_0 \, \xi_e^2 \, \delta^3}{2} \left( \frac{\sin 2 \, \xi_e t}{\xi_e} \right) \; \|{\bf E}_0\|_{L^2}^2 \;   + \frac{\varepsilon_0 \, \xi_e^2 \, \delta^3}{2} \; t \; R_\delta(t)
\end{equation}
where we have set
$ \dsp
R_\delta(t) : = \int_{\R} |\widehat{\bf E}_0(\boldsymbol{\xi})|^2 \; \Big[\Phi(\xi_e \, t) - \Phi \big(\xi_e(\delta \boldsymbol{\xi}) t\big)\Big] \; \md \boldsymbol{\xi}.
$\\[12pt]
Let $C$ be the Lipschitz constant of $\Phi$ in $\R$. If $E_0 \in H^1(\R^3)^3$, since  $\big|\xi_e(\delta \boldsymbol{\xi}) - \xi_e\big| \leq \delta^ 2 \, |\boldsymbol{\xi}|^2 / 2 \xi_e$,  we have
\begin{equation} \label{estRdelta}
|R_\delta(t)| \leq Ct \; \int_{\R} |\widehat{\bf E}_0(\boldsymbol{\xi})|^2 \; \big[ \xi_e(\delta \boldsymbol{\xi}) - \xi_e \big] \; \md \boldsymbol{\xi} \leq Ct \; \frac{\delta^2}{2 \xi_e} \int_{\R} |\boldsymbol{\xi}|^2|\widehat{\bf E}_0(\boldsymbol{\xi})|^2  \; \md \boldsymbol{\xi} \equiv Ct \; \frac{\delta^2}{2 \xi_e} |{\bf E}_0|_{H^1}^2.
\end{equation}
Now, we prove that $\mathcal{E}_\delta'(t)$ can be non negative. Indeed, we first deduce from (\ref{Eprime2}) that
$$
\mathcal{E}_\delta'(T_n)= \frac{\varepsilon_0 \, \xi_e^2 \, \delta^3}{2} \; \Big( \frac{\|{\bf E}_0\|_{L^2}^2}{\xi_e}+ T_n \;R_\delta(T_n) \Big), \quad 
T_n = \frac{(2n+1)\pi}{4 \xi_e}, n \in \N^*, \quad (\sin 2 \, \xi_e T_n = -1).$$
Thus, thanks to (\ref{estRdelta}), 
$ \dsp
\mathcal{E}_\delta'(T_n) \geq \frac{\varepsilon_0 \, \xi_e \, \delta^3}{2} \; \Big( \|{\bf E}_0\|_{L^2}^2 - C \; \delta^2 T_n^2 \; |{\bf E}_0|_{H^1}^2 \Big).
$
Thus, for any $N \in \N^*$ 
$$
\mbox{as soon as  } C \, \delta^2 T_N^2\,|{\bf E}_0|_{H^1}^2 < \|{\bf E}_0\|_{L^2}^2, \quad \mathcal{E}_\delta'(T_n) > 0, \quad n = 1, \cdots N.
$$
\section{Maxwell's equations in general passive media} \label{Generalpassive}
\subsection{A representation of electric permittivity and magnetic permeability in passive media}  \label{Representationformula}
The representation of lossless passive local media as generalized Lorentz media (cf. theorem \ref{GenLorentz}) is, as we shall see, representative of general passive (even lossy) media. This is a consequence of a following well-known representation theorem for Herglotz functions, known as the {\bf Nevanlinna's representation theorem}. In this section, we assume the familiarity of the reader with basics of measure theory on $\mathbb{R}$ \cite{Mat-99}.
\begin{lema}\label{thm.Nevanlinna} [ Nevanlinna's theorem]
A necessary and sufficient condition for $f$ to be a Herglotz function is given by the following representation:
\begin{equation}\label{eq.defhergl}
f(\omega)=\alpha \, \omega +\beta + \displaystyle \int_{\R} \left(  \frac{1}{\xi-\omega}- \frac{\xi}{1+\xi^2}\right)\md \nu( \xi), \  \mbox{ for } \omega \in \C^{+},
\end{equation}
where $\alpha\in\R^{+}$, $\beta\in \R$ and $\nu$ is a positive regular Borel measure for which  
\begin{equation} \label{finiteness}
 \dsp \int_{\R} \
 \md \nu(\xi)/(1+\xi^2) < + \infty.
 \end{equation}
Moreover, $\alpha, \beta$ and $\nu$ are related to $f$ via the following formulas
\begin{equation}  \label{alphabeta} 
 \alpha= \lim_{y \rightarrow +\infty}\displaystyle \frac{f(i \, y)}{i \, y}, \quad \beta={\cal R}e \, f(i), 
\end{equation}
and the measure $\nu$ is given by
\begin{equation}  \label{measure} 
\left\{ \begin{array}{ll}
(i) & \dsp \forall \; a \in \R, \quad \nu(\{a\})=\lim_{\eta\to 0^{+}}\eta \; {\cal I}mf(a+  i \eta),\\[10pt] 
(ii) & \dsp \forall \; a \leq b,  \quad \frac{\nu\big([a,b]\big)+\nu\big((a,b)\big)}{2} =\lim_{\eta\rightarrow 0^{+}} \frac{1}{\pi} \int_{a}^{b}  {\cal I}m \, f(x+ i \, \eta) \, dx .
\end{array} \right.
\end{equation}
\end{lema}
\begin{rem} \label{remNevanlinna}
The reader will easily check that the integrand in the right hand side of (\ref{eq.defhergl}) is in $O(\xi^{-2})$ for large $\xi$ so that (\ref{finiteness}) ensures the existence of the integral. If, in addition, 
\begin{equation} \label{finitenessbis}
	\dsp \int_{\R} |\xi| \, \md \nu(\xi)/(1+\xi^2) < + \infty,
	\end{equation} 
we get $f(\omega)=\alpha \, \omega+\gamma  + \displaystyle \int_{\R} \frac{\md \nu( \xi) }{\xi-\omega}$ with $\gamma \in \R$.
\end{rem}
\begin{rem} \label{remNevanlinna2} The formula (\ref{measure}) provides the measure $\nu$ of any interval $[a,b)$, $(a,b]$, $(a,b)$ or $[a,b]$. Thus it defines completely $\nu$ as a Borel measure \cite{Mat-99}.
\end{rem}	
\begin{rem} \label{eq.hergprop} Let $\operatorname{supp}(\nu)$ be the support \cite{Mat-99} of $\nu$ in (\ref{eq.defhergl}). As the
	support of a measure is closed, $I = \R \setminus supp(\nu)$ is open. Using (\ref{eq.defhergl}), 
	the Herglotz function $f$ can be continuously extended on $I$. This extension is real-valued. Moreover,
	$f$ has an analytic extension $f_e$ on $\R \setminus \operatorname{supp}(\nu)$ by the Schwarz reflection principle:
	$f_e(z) = f(z)$ on $\C^+\cup I$ and $f_e(z) =\overline{ f(\overline{z})}$ on $\C^-$. 
	Along $I$, the zeros of $f$ are simple (lemma \ref{lemHerglotz}) and $f$ satisfies the growing property $f'(\omega) > a$ (simply differentiate   (\ref{eq.defhergl})). 
\end{rem}
\noindent The proof of lemma \ref{thm.Nevanlinna} can be found in appendix \ref{Nevanlinna}. 
An important corollary of lemma \ref{thm.Nevanlinna} is
\begin{theo}\label{thm.herglotz}
Let $\varepsilon$ and $\mu$ be the electric permittivity and magnetic permeability of a homogeneous 
passive medium. There exists two positive regular Borel measures $\nu_e$ and $\nu_m$ on $\R$, that are symmetric (i.e. $\nu_e(-B)= \nu_e(B)$ and $\nu_m(-B)= \nu_m(B)$ for any Borel set $B$) and satisfy (\ref{finiteness}), such that 
\begin{equation}\label{expepsmu}
\varepsilon(\omega)=\varepsilon_0 \; \Big( 1- \displaystyle \int_{\R}\frac{\md \nu_e( \xi)}{\omega^2-\xi^2} \Big) , \quad \mu(\omega)=\mu_0 \; \Big( 1- \displaystyle \int_{\R}\frac{\md \nu_e( \xi)}{\omega^2-\xi^2} \Big)\ \mbox{ for } \omega \in \C^{+}.
\end{equation}
\end{theo}
\begin{proof}
We give the proof for $\varepsilon$. It is obviously the same for $\mu$. By passivity, $f(\omega) = \omega \, \varepsilon(\omega) / \varepsilon_0 $ is a Herglotz function. Thus, using lemma \ref{thm.Nevanlinna} and the high frequency condition {\bf(HF)}, we can write 
$$
\omega \, \varepsilon(\omega) =\varepsilon_0 \; \Big(  \omega + \beta_e +\displaystyle \int_{\R} \left(  \frac{1}{\xi-z}- \frac{\xi}{1+\xi^2}\right)\md \nu_e( \xi) \Big) , \, \ \mbox{ for } \omega \in \C^{+}.
$$
By the reality principle  {\bf (RP)}, $f(i)$ is purely imaginary, and thus $\beta_e = \varepsilon_0^{-1} \, {\cal R}e \, f(i)=0$. Hence 
\begin{equation}\label{eq.epsrep}
\varepsilon(\omega)=\varepsilon_0 \; \Big(  1+\omega^{-1}\displaystyle \int_{\R} \left(  \frac{1}{\xi-\omega}- \frac{\xi}{1+\xi^2}\right)\md \nu_e( \xi) \, \ \mbox{ for } \omega \in \C^{+}.
\end{equation}
By the reality principle {\bf(RP)} again, one has
$
\dsp \varepsilon(\omega)=\frac{1}{2}\Big(\varepsilon(\omega)+\overline{\varepsilon(-\overline{\omega})}\Big).
$
Thus, using \eqref{eq.epsrep}, we compute:
\begin{eqnarray}
\varepsilon(\omega)=\varepsilon_0 \; \Big(  1+\frac{1}{2\omega}\displaystyle \int_{\R}  \Big(\frac{1}{\xi-\omega}-\frac{1}{\xi+\omega} \Big)\,\md \nu_e( \xi) \Big) = \varepsilon_0  \; \Big( 1+\int_{\R}\frac{\md \nu_e( \xi)}{\xi^2 - \omega^2} \Big).
\end{eqnarray}
The symmetry of $\nu_e$ follows from (\ref{measure}) since, by {\bf(RP)}, ${\cal I}m f(-\overline{\omega}) = {\cal I}m f(\omega)$. 
\end{proof}
\noindent One recovers generalized Lorentz materials (\ref{Lorentzlaws}) with finite sums of Dirac measures (in symmetric form) :
\begin{equation}\label{LorentzMeasures}
	\nu_e = \frac{1}{2} \; \sum_{\ell = 1}^{N_e} {\Omega_{e,\ell}^2} \, \big( \delta_{\omega_{e,\ell}} + \delta_{-\omega_{e,\ell}}\big), \quad \nu_m = \frac{1}{2} \; \sum_{\ell = 1}^{N_m} {\Omega_{m,\ell}^2} \, \big( \delta_{\omega_{m,\ell}} + \delta_{-\omega_{m,\ell}}\big).
	\end{equation}
Other similar passive materials are obtained with non-compactly supported discrete measures such as 
\begin{equation}\label{DiscreteMeasures}
	\nu_e = \frac{1}{2} \; \sum_{\ell = 1}^{+\infty} {\Omega_{e,\ell}^2} \,  \big( \delta_{\omega_{e,\ell}} + \delta_{-\omega_{e,\ell}}\big), \quad \nu_m = \frac{1}{2} \; \sum_{\ell = 1}^{+\infty} \, {\Omega_{m,\ell}^2} \big( \delta_{\omega_{m,\ell}} + \delta_{-\omega_{m,\ell}}\big).
	\end{equation}
where $\omega_{e,\ell}$ and $\omega_{m,\ell}$ are two sequences of positive real numbers satisfying 
\begin{equation}\label{Additional}
 \sum_{\ell = 1}^{+\infty} \frac {\Omega_{e,\ell}^2 }{1 + \omega_{e,\ell}^2}  < + \infty, \quad \quad 
 \sum_{\ell = 1}^{+\infty} \frac {\Omega_{m,\ell}^2}{1 + \omega_{m,\ell}^2}  < + \infty, \quad ( \Longleftrightarrow (\ref{finiteness})).
 \end{equation}
The functions $\varepsilon(\omega)$ and $\mu(\omega)$ are not rational, but  meromorphic functions 
with poles $\pm \, \omega_{e,\ell}$ and $\pm \, \omega_{m,\ell}$, $1 \leq \ell \leq + \infty$, defined by the following series,
whose convergence (outside poles) is ensured by (\ref{Additional}) :
	\begin{equation} \label{InfiniteLorentzlaws}
\varepsilon(\omega)	= \varepsilon_0 \; \Big( 1 + \sum_{\ell = 1}^{+ \infty} \frac{\Omega_{e,\ell}^2}{\omega_{e,\ell}^2 -\omega^2}\Big), \quad \mu(\omega)	= \mu_0 \; \Big( 1 + \sum_{\ell = 1}^{+ \infty} \frac{\Omega_{m,\ell}^2}{\omega_{m,\ell}^2 -\omega^2}\Big).
		\end{equation}
In the particular case where $ \dsp \omega_{e,\ell}=\omega_{m,\ell} = \frac{2\ell-1}{2}\pi$,  $\Omega_{e,\ell}^2=2a_e$  and $\Omega_{m,\ell}^2=2a_m$, we get
\begin{align*}
	\varepsilon(\omega)=\varepsilon_0(1+a_e \, \omega^{-1}\tan(\omega)), \qquad 
	\mu(\omega)=\mu_0(1+a_m \, \omega^{-1}\tan(\omega)), \qquad a_m, \; a_e>0.
	\end{align*}
Such functions appear naturally in the mathematical theory of metamaterials via high contrast homogenization \cite{bouchitte2009homogenization}, \cite{bouchitte2010homogenization}, \cite{Zhikov} (see remark \ref{rem-homogenization} for a concrete and relatively simple example).\\[12pt]
Measures with an absolutely continuous part (related to losses, see section \ref{Sec_Dissipative}) will be considered later.
\begin{rem} \label{rem-homogenization} {\bf Example of high contrast homogenization.}
Let us consider the case of a 2D transverse magnetic medium (the magnetic field is a 2D scalar function) and heterogeneous non-dispersive Maxwell's equations (see (\ref{Dielectric})). Let us study a family of problems depending on
a small parameter $\delta > 0$, given by
$$
\mu({\bf x}) = \mu_0, \quad \varepsilon({\bf x}) = \varepsilon^\delta({\bf x}).
$$
The scalar magnetic field $H^\delta$ is solution of the time harmonic model at a given frequency $\omega$:
$$
rot \Big( \frac{1}{\varepsilon^\delta} \, rot \, H^\delta \Big) - \mu_0 \, \omega^2 \, H^\delta = f, \quad \mbox{ in $D$ bounded $\subset \R^2$}
$$
completed for instance with absorbing boundary conditions (omitted here) on $\partial D$. 
The function $\varepsilon^\delta({\bf x}) : \R^2 \rightarrow \R_+^*$ is $\delta$-periodic and piecewise constant, with high contrast. More precisely,
$$
\R^2 = \bigcup_{{\bf j} \in \Z^2} \delta \; \big[ \, {\bf j} + \overline{C}_0 \, \big], \quad C = \; (0,1) \times (0,1) 
$$
where the reference cell is made of two parts $\overline{C} = \overline{C}_{int} \cup \overline{C}_{ext}, \; \overline{C}_{int} \subset C, \; C_{int} \cap C_{ext} = \emptyset,$ so that
$$
\forall \; \widehat{\bf x} \in C, \quad \forall \; {\bf j} \in \Z^2, \quad 	\varepsilon^{\delta} \big( \delta \big[ \, {\bf j} + \widehat{\bf x} \, \big] \big) = \varepsilon_{ref}^\delta(\widehat{\bf x}), \quad \varepsilon_{ref}^\delta(\widehat{\bf x}) = \varepsilon_0 \mbox{ in } {C}_{ext}, \quad \varepsilon_{ref}^\delta(\widehat{\bf x})= \delta^{-2} \; \varepsilon_0 \mbox{ in } {C}_{int} \; .
$$
Then it can be shown that $H^\delta \rightarrow H^{hom}$, weakly in $L^2(D)$, where $H^{hom}$ satisfies the homogenized model
$$
rot \Big( \frac{1}{\varepsilon_0} \, rot \, H^\delta \Big) - \mu_{eff}(\omega) \, \omega^2 \, H^\delta = f, \quad \mbox{ in $D$ }
$$
where the function $\dsp\mu_{eff}(\omega)= \mu_0 \; \Big( 1 + \omega^2 \, \sum_{n=1}^{+\infty} \frac{ \big|\langle \varphi_{n}\rangle \big|^2}{ \omega_{n}^2 - \omega^2} \Big) , \, \langle \varphi_{n}\rangle = \int_{C_{int}} \varphi_{n} \, $, with
$$
- \Delta \varphi_{n} =  \omega_{n}^2 \; \varphi_{n}, \quad \mbox{in } C_{int}, \quad \varphi_{n}|_{\partial C_{int}} = 0, \quad \int_{C_{int}} |\varphi_{n}|^2 = 1, \qquad n\geq 1.
$$
can be shown to be of the form (\ref{LorentzMeasures}).
	\begin{figure}[h]
	\centerline{
	\includegraphics[height=4cm]{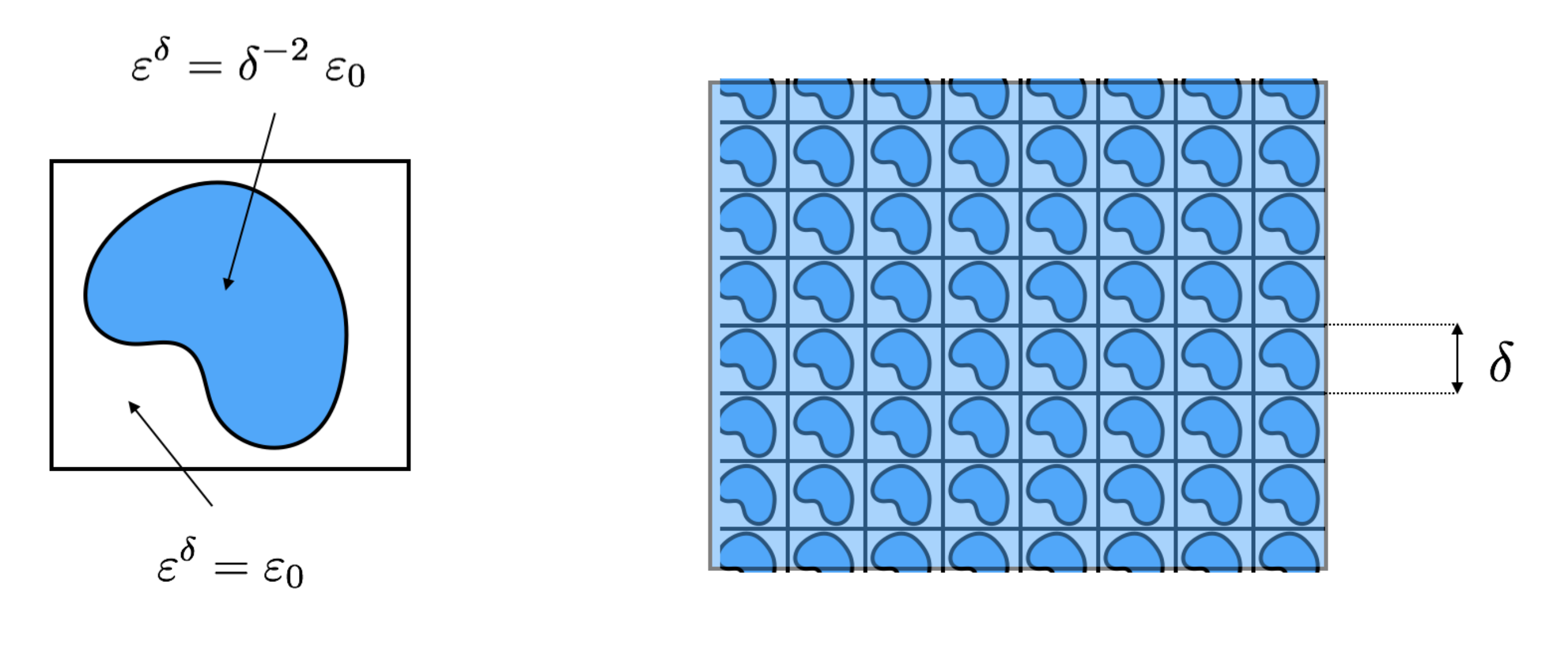}}
	\caption{A high contrast periodic medium (right). The periodicity cell (left).}
	\label{Fig5}
	\end{figure}
\end{rem}
\subsection{An augmented self-adjoint PDE model for Maxwell's equations in passive media} \label{Augmented}
We state below the generalization of theorem \ref{PDEmodel} for passive non-Lorentz materials. The idea of augmented 
models, with an auxiliary unknown depending on a (a priori) continuum of auxiliary variables was developed in \cite{Tip}, \cite{GralakTip}, \cite{Figotin}, \cite{Figotin1}, \cite{Figotin2}. The pioneering idea goes back to Lamb (\cite{Lamb}) in 1900.
\begin{theo}\label{PDEmodelgene}  A PDE-like model for dispersive Maxwell's equations with permittivity and permeability given by (\ref{eq.defhergl}) is (we consider here the Cauchy problem):
\begin{equation} \label{Lorentzsystemgene}
\hspace*{-0.3cm}\left\{	\begin{array}{lll}
\dsp \mbox{Find } \; \left\{ \begin{array}{l} {\bf E}({\bf x}, t) : \R^3 \times \R^+ \rightarrow \R^3, \quad  \quad  \quad  \quad  \quad  \; {\bf H}({\bf x}, t) : \R^3 \times \R^+ \rightarrow \R^3, \\[12pt]
 \bbP({\bf x}, t ;\xi) : \R^3  \times \R^+ \times \R \rightarrow \R^3,  \; \quad {\bf M}({\bf x}, t ;\xi) : \R^3  \times \R^+ \times \R \rightarrow \R^3, \end{array} \right. \mbox{ s. t.  } \\[24pt]
\dsp	\varepsilon_0 \, \partial_t {\bf E} + {\bf rot} \, {\bf H} + \varepsilon_0 \, \int_{\R} \partial_t \bbP (\cdot ; \xi) \, \md \nu_e(\xi) = 0, \quad ({\bf x}, t) \in \R^3 \times \R^+\\[12pt]
\dsp	\mu_0 \, \partial_t {\bf H} - {\bf rot} \, {\bf E} + \mu_0 \, \int_{\R} \partial_t {\bbM} (\cdot ; \xi) \, \md \nu_m(\xi) = 0, \quad ({\bf x}, t) \in \R^3 \times \R^+ \\[18pt]
\partial_t^2 \bbP (\cdot ; \xi) + \xi^2 \, \bbP (\cdot ; \xi) =   {\bf E},  \quad ({\bf x}, t ;\xi) \in \R^3  \times \R^+ \times \R, \\[12pt] \partial_t^2 \bbM (\cdot ; \xi) + \xi^2 \, \bbM (\cdot ; \xi) =   {\bf H}, \quad ({\bf x}, t ;\xi) \in \R^3  \times \R^+ \times \R \\[12pt]
{\bf E}({\bf x}, 0) =  {\bf E}_0({\bf x}),  \quad {\bf H}({\bf x}, 0) =  {\bf H}_0({\bf x}), \quad \bbP({\bf x}, 0 ;\xi) = 0, \quad \bbM({\bf x}, 0 ;\xi) = 0.
\end{array} \right.
	\end{equation}
		\end{theo}
		\begin{proof} The same as for theorem \ref{PDEmodel}, the index $\ell$ being replaced by the variable $\xi$ and sums over $\ell \in \{0, \cdots N_e\}$ (resp. $\{0, \cdots N_{m}\}$) by integrals over $\mathbb{R}$ with respect to $\md \nu_e(\xi)$ (resp. $\md \nu_m(\xi)$ ), in particular, the polarization ${\bf P}$ and the magnetization ${\bf M}$ (see (\ref{defPM})) are given by (also compare to (\ref{defPellMell}))\\[12pt]
\hspace*{3cm}$ \dsp 
{\bf P} = \varepsilon_0 \, \int_{\R} \bbP (\cdot ; \xi) \, \md \nu_e(\xi), \quad {\bf M} = \mu_0 \, \int_{\R} \bbM (\cdot ; \xi) \, \md \nu_m(\xi).
$		
\end{proof}
\noindent Like for (\ref{Lorentzsystem}), an energy conservation result holds for  (\ref{Lorentzsystemgene}) and implies physical passivity (cf. remark \ref{remPassivityEnergy}). 
\begin{theo} \label{thm.Energygene} Any smooth enough solution of  (\ref{Lorentzsystem}) satisfies the energy identity $\dsp \frac{d}{dt} \, {\cal E}_{tot}(t) = 0$ where
	\begin{equation} \label{energyidentitygene}
\left\{	\begin{array}{l}
\dsp {\cal E}_{tot}(t) :=  {\cal E}(t) + {\cal E}_{e}(t) + {\cal E}_{m}(t), \quad  \dsp {\cal E}(t) :=  \frac{1}{2}\int_{\R^3} \big( \varepsilon_0 \, |{\bf E}|^2  + \mu_0 \, |{\bf H}|^2 \, \big) \; d{\bf x}, \\[12pt]
\dsp  
{\cal E}_{e}(t) :=  \frac{1}{2} \, \varepsilon_0\int_{\R} \int_{\R^3} \Big( \, |\partial_t \bbP({\bf x}, t ;\xi)|^2  +  \xi^2 \, |\bbP({\bf x}, t ;\xi)|^2 \, \Big) \; d{\bf x} \; \md \nu_e(\xi),\\[12pt]
\dsp  
{\cal E}_{m}(t) :=   \frac{1}{2} \, \mu_0 \int_{\Lambda_m} \int_{\R^3} \Big( \, |\partial_t \bbM({\bf x}, t ;\xi)|^2  +  \xi^2 \, |\bbM({\bf x}, t ;\xi)|^2 \, \Big) \; d{\bf x} \; \md \nu_m(\lambda), 
\end{array} \right.
	\end{equation}
		\end{theo}
\noindent Behind the above energy conservation result is hidden the fact that (\ref{Lorentzsystemgene}) can be rewritten as an abstract Schr\"odinger equation involving some self-adjoint operator on an appropriate Hilbert space, see section \ref{Schrodinger}. 
\begin{proof} It is almost the same as for theorem \ref{thm.Energy}.
	\end{proof}
\noindent All these mathematical models give rise to {\bf finite propagation velocity}, bounded by
the speed of light $c_0$.
\begin{theo} \label{thm.Energygene2} Assume that the Cauchy data for (\ref{Lorentzsystem}), namely ${\bf E}_0$  and  ${\bf H}_0$
	have compact support included in the ball $B(0,R)$ of center $0$ and radius $R>0$. Then, at each time $t > 0$, the solution of (\ref{Lorentzsystem}) satisfies
\begin{equation} \label{propsupport}
\mbox{supp } {\bf E}(\cdot, t) \cup \mbox{supp } {\bf H}(\cdot, t)	\subset B(0, R + c_0 \, t) \quad (\mbox{i.e. } {\bf E}({\bf x}, t) = {\bf H}({\bf x}, t) = 0 \; \mbox{ for } \;  |{\bf x}|> R + c_0 \, t).	
\end{equation}
\end{theo}
\begin{proof} 
Let $\mathbf{e}_e({\bf x}, t ;\xi)=\, |\partial_t \bbP({\bf x}, t ;\xi)|^2  +  \xi^2 \, |\bbP({\bf x}, t ;\xi)|^2$, 
$\mathbf{e}_m({\bf x}, t ;\xi)=  \, |\partial_t \bbM({\bf x}, t ;\xi)|^2  +  \xi^2 \, |\bbM({\bf x}, t ;\xi)|^2$.
	We use the method of energy in moving domains. Let us consider $\Omega_{R,t} = \{ {\bf x} \in \R^3 / |{\bf x}|> R + c_0 \, t \}$ (a moving domain). The idea is to prove that the positive-valued function 
\begin{eqnarray*}
\begin{array}{lll} t \; \mapsto \; {\cal E}_{R,tot}(t)& := & \dsp \frac{1}{2} \int_{\Omega_{R,t}} \big( \varepsilon_0 \, |{\bf E}({\bf x},t)|^2  + \mu_0 \, |{\bf H}({\bf x},t)|^2 \, \big) \; d{\bf x}\\[12pt]
	 & +  & \dsp \frac{1}{2} \int_{\R} \int_{\Omega_{R,t}} \varepsilon_0 \,{\bf e}_e ({\bf x}, t ;\xi)  \; d{\bf x}\; \md \nu_e(\xi) +  \frac{1}{2}\int_{\R}\int_{\Omega_{R,t}}  \varepsilon_0 \,{\bf e}_m ({\bf x}, t ;\xi) d{\bf x}  \; \md \nu_m(\xi), 
	 	\end{array}
\end{eqnarray*}
i.e. the total energy contained in $\Omega_{R,t}$ at time $t$, is a decreasing function of time. Since it vanishes at $t = 0$
(property of the Cauchy data), it is identically $0$ from which the conclusion follows. The main difference with the proofs of theorems \ref{thm.Energygene} and \ref{thm.Energy} is that boundary terms have to be handled in the integration by parts as
well as the fact that the domain $\Omega_{R,t}$ is moving. Doing so (details are left to the reader), one obtains 
$$
\begin{array}{lll} \dsp \frac{d{\cal E}_{R,tot}}{dt}(t)& = & \dsp - \frac{1}{2} \int_{\Gamma_{R,t}} \big( \varepsilon_0 \, c_0 \, |{\bf E}({\bf x},t)|^2  + \mu_0 \, c_0 \,|{\bf H}({\bf x},t)|^2 \, - {\bf n} \times {\bf E}({\bf x},t) \cdot {\bf H}({\bf x},t) \big) \; d{\bf x}\\[12pt]
	 & -  & \dsp \frac{c_0}{2} \int_{\Gamma_{R,t}} \int_{\Omega_{R,t}} \varepsilon_0 \,{\bf e}_e ({\bf x}, t ;\xi)  \; d{\bf x}\; \md \nu_e(\xi) -  \frac{c_0}{2}\int_{\R}\int_{\Gamma_{R,t}}  \mu_0 \,{\bf e}_m ({\bf x}, t ;\xi) d{\bf x}  \; \md \nu_m(\xi) \; \end{array}
	 $$
with $\Gamma_{R,t} := \partial \Omega_{R,t}$ with unit normal vector ${\bf n}$. This is easily proven to be negative since $\varepsilon_0 \mu_0 \, c_0^2 = 1$.
\end{proof}
\subsection{Reinterpretation as a Schr\"odinger equation and related spectral theory} \label{Schrodinger}

\subsubsection{A Schr\"odinger evolution equation}
In this section, for technical reasons, we use the following assumption (valid for many applications):
\begin{equation} \label{finiteness2}
	\int_{\R} \md \nu_{e}(\xi) < + \infty, \quad \int_{\R} \md \nu_{m}(\xi) < + \infty.
\end{equation} 
It holds e.g. for Drude dissipative models. 
Modulo the introduction of the additional unknowns $\widetilde\bbP := \partial_t \bbP$ and $\widetilde\bbM := \partial_t \bbM$, (\ref{Lorentzsystem}) can be rewritten as a Schr\"{o}dinger equation of the form
\begin{equation}\label{eq.schro}
\frac{\md \, {\bf U}}{\md\, t} + i\,\mathbb{A} \, {\bf U}={\bf F}, \quad {\bf U} = \big( {\bf E}, {\bf H}, \bbP, \widetilde\bbP, \bbM, \widetilde\bbM\big)
\end{equation}
where the \emph{Hamiltonian} $\bbA$ is an unbounded operator on the Hilbert space:
$$
\boldsymbol {\mathcal{H}}=  L^2(\R^3)^3 \times L^2(\R^3)^3 \times  \boldsymbol {\mathcal{  V}_{e}} \times \boldsymbol {\mathcal{H}_{e}}\times  \boldsymbol {\mathcal{V}_{m}} \times \boldsymbol { \mathcal{H}_{m} }.
$$ 
with the Hilbert spaces (equipped with their natural inner product)
$$ 
 \boldsymbol {\mathcal{H}_{s} }= L^2\big(\R, L^2(\R^3)^3 ; \md \nu_{s} \big), \
 \boldsymbol {\mathcal{V}_{s}}= L^2\big(\R, L^2(\R^3)^3;\,  \xi^2 \, \md \nu_{s} \big)  \ \mbox{ for } s=e \mbox{ or } m,
$$
so that the space  $\boldsymbol{\mathcal{H}}$-inner product is given by
\begin{equation}\label{eq.defprodscal}
\begin{array}{lll}
\dsp ({\bf U}, {\bf U}_*)_{\boldsymbol{\cal H}} =  \dsp  \int_{\R^3} \big(\varepsilon_0 \, {\bf E} \cdot {\bf \overline{E}}_*  + \mu_0 \, {\bf H} \cdot {\bf \overline{H}}_* \big) \, \md{\bf x} \hspace*{-0.2cm}  & + \,  
 \varepsilon_0 \dsp \int_{\R} \int_{\R^3} \big(\xi^2 \,  \bbP \cdot \overline{\bbP}_* +   \,  \widetilde \bbP \cdot \overline{ \widetilde \bbP}_*\big) \, \md {\bf x}\,  \md \nu_e(\xi) \\[12pt]
    &  + \mu_0 \, \dsp \int_{\R} \int_{\R^3} \big( \xi^2 \,  \bbM \cdot \overline{\bbM}_* +   \,  \widetilde \bbM \cdot \overline{ \widetilde \bbM}_*\big) \, \md {\bf x}\,  \md \nu_m(\xi)  . 
\end{array}
\end{equation}
Setting 
$\boldsymbol{\mathcal{D}_{s}}= \{\bbP \in  \boldsymbol {\mathcal{V}_{s} } \, \mid \, \xi^2 \, \bbP \in  \boldsymbol {\mathcal{H}_{s} }\}$ and 
$\boldsymbol{\mathcal{ \widetilde D}_{s}}=  \boldsymbol {\mathcal{H}_{s} } \cap \ \boldsymbol {\mathcal{V}_{s} }, s=e \mbox{ or } m,$
the domain of $\bbA$ is
$$
D(\bbA)= H(\operatorname{\bf rot},\R^3) \times H(\operatorname{\bf rot},\R^3) \times  \boldsymbol{\mathcal{D}_{e}} \times \boldsymbol{ \mathcal{ \widetilde D}_{e}}\times  \boldsymbol{\mathcal{ D}_{m} }\times  \boldsymbol{\mathcal{\widetilde D}_{m} },
$$ 
and the operator $\bbA$ is defined in block form by
\begin{equation}\label{eq.opA}
\bbA := \ i\, \begin{pmatrix}
0 &\varepsilon_0^{-1}\,\mbox{\bf rot} & 0&-\boldsymbol{{\cal S}_e} & 0 & 0\\[2pt]
- \mu_0^{-1}\,\mbox{\bf rot} & 0 &0 & 0& 0& -\boldsymbol{{\cal S}_m} \\[2pt]
0 & 0 &0  & {\bf I_e} &0 &0\\[2pt]
\boldsymbol{{\cal I}_e} & 0 &-\, \xi^2 {\bf I_e} &0 &0 &0\\[2pt]
0 & 0 &0  & 0 &0 & {\bf I_m} \\[2pt]
0 & \boldsymbol{{\cal I}_m}& 0 & 0 & -\, \xi^2{\bf I_m}  & 0
\end{pmatrix},
\end{equation}
\begin{equation}\label{eq.defopblock}
\hspace*{-3cm} \mbox{where we have set } \quad \left\{\begin{array}{lllll}
&\dsp \boldsymbol{{\cal S}_e} & \hspace*{-0.5cm} :   \boldsymbol {\mathcal{H}_{e} } \rightarrow L^2(\R^3)^3, &  (\boldsymbol{{\cal S}_e}\widetilde\bbP)({\bf x}) := \int_{\R} \widetilde\bbP({\bf x}, \xi) \, \md \nu_e(\xi),\,\\[5pt]
 &\boldsymbol{{\cal I}_e} & \hspace*{-0.5cm} : L^2(\R^3)^3  \rightarrow \boldsymbol {\mathcal{H}_{e} }, & (\boldsymbol{{\cal I}_e} {\bf E})({\bf x},\xi) := {\bf E}({\bf x}),\\[5pt]
&\dsp {\bf I_e} & \hspace*{-0.5cm}:{\boldsymbol{\cal \widetilde  D}_e} \subset \boldsymbol {\mathcal{H}_{e} } \rightarrow \boldsymbol {\mathcal{V}_{e} }, & ({\bf I_e}\widetilde\bbP)({\bf x}, \xi) := \widetilde\bbP({\bf x}, \xi),\,\\[5pt]
&  \xi^2{\bf I_e} & \hspace*{-0.5cm}: \boldsymbol{\cal {D}}_e \subset  \boldsymbol {\mathcal{V}_{e} } \rightarrow \boldsymbol {\mathcal{H}_{e} }, & (\xi^2{\bf I_e})\bbP({\bf x}, \xi) := \xi^2\bbP({\bf x}, \xi),\,
\end{array} \right.
\end{equation}
and where the operators $\boldsymbol{{\cal S}_m}$, $\boldsymbol{{\cal I}_m}$, $\dsp {\bf I}_m$ and $\xi^2{\bf I}_m$ are defined similarly, replacing  $e$ by $m$, ${\bf E}$ by ${\bf H}$ and $\bbP$ by $\bbM$.
Notice that condition (\ref{finiteness2}) is needed for the definition of the operators $\boldsymbol{{\cal I}_e}$ and $\boldsymbol{{\cal I}_m}$. Moreover, 
${\boldsymbol {\cal S}_e} \in {\cal L}(\boldsymbol{\cal H}_e, L^2(\R^3)^3)$ is the adjoint of $ \boldsymbol{{\cal I}_e} \in {\cal L}( L^2(\R^3)^3, \boldsymbol{\cal H}_e)$ and ${\boldsymbol {\cal S}_m} \in {\cal L}(\boldsymbol{\cal H}_m, L^2(\R^3)^3)$ is the adjoint of $\boldsymbol{{\cal I}_m} \in {\cal L}( L^2(\R^3)^3, \boldsymbol{\cal H}_m)$, $\xi^2 {\bf I_e}$ is the adjoint of ${\bf I_e}$ and  $\xi^2{\bf I_m}$ is the adjoint of ${\bf I_m}$. Using in particular these properties (for the symmetry of $\bbA$) and the fact that the domain of the adjoint of $\bbA$ coincides with $D(\bbA)$ (the details are left to the reader), one shows that
\begin{theo}\label{prop.autoadjoint}
The operator $\bbA: D(\bbA)\subset\mathcal{H} \longmapsto \mathcal{H}$ is self-adjoint.
\end{theo}
\noindent From the semi-group theory (or  Hille-Yosida's theorem \cite{Pazy}), it follows that $\bbA$ is the generator
of a unitary semi-group. From this we obtain the following corollary.
\begin{coro}
Given any initial data ${\bf U}_0 \in D(\bbA)$, the evolution problem (\ref{eq.schro}) admits a unique solution
$$
{\bf U}(t) \in C^1(\R^+; {\cal H}) \cap C^0(\R^+; D(\bbA)) \quad (\mbox{ often denoted  } {\bf U}(t) = e^{-i\bbA t}{\bf U}_0)
$$
which satisfies $\|{\bf U}(t)\|_{\cal H} = \|{\bf U}_0\|_{\cal H}$.
\end{coro}
\begin{rem} By definition of the norm $\|\cdot\|_{\cal H}$, the conservation of $\|{\bf U}(t)\|_{\cal H}$ is nothing but the conservation of energy (cf. theorem \ref{thm.Energygene}).
	\end{rem}
\subsubsection{The reduced Hamiltonian}
Our goal is to compute the spectrum $\sigma(\bbA)$ of  $\bbA$.
As the medium is homogeneous, the spectral theory of the operator $\bbA$ defined in (\ref{eq.opA}) reduces, using space Fourier transform $\bbF$, see (\ref{eq.deffour}), to the one of a family of self-adjoint operators $(\bbA_{\bf{k}} )_{\bf{k}\in \R^3}$, reduced Hamiltonians, defined on functions which depend only on the variable $\xi.$ 
The knowledge of the spectra $\sigma(\bbA_{\bf{k}})$ will lead to an expression of $\sigma(\bbA)$.
\begin{equation}\label{eq.deffour}
\bbF u(\bf{k}) := \frac{1}{(2\pi)^{\frac{3}{2}}} \int_{\R^3} u( \bf{x})\, \me^{-i {\bf k } \cdot \bf{x} }\, \md \bf{x} \quad \forall u \in L^1(\R^3)\cap L^2(\R^3).
\end{equation}
For functions of both variables $\bf{x}$ and $\xi,$ we still denote by $\bbF$ the partial Fourier transform in the $\bf{x}$ variable. In particular, the partial Fourier transform of an element $ {\bf U} \in  \boldsymbol{\mathcal{H}}$ is such that
\begin{equation}\label{eq.defH1D}
\bbF  {\bf U}(\bf{k},\cdot) \in \mathcal{H} :=\C^3 \times \C^3  \times  \mathcal{V}_{e} \times  \mathcal{H}_{e} \times  \mathcal{V}_{m} \times  \mathcal{H}_{m} \quad \mbox{for a.e. } {\bf k} \in \R^3,
\end{equation}
where
$
 \mathcal{H}_{s} = L^2 (\C^3,\md \nu_{s})\mbox{ and } \mathcal{V}_{s}=L^2(\C^3,\xi^2\md \nu_{s}), \, \mbox{ for } s=e \mbox{ or } m.
$
The Hilbert space $\mathcal{H}$ is endowed with the inner product $(\cdot\,,\cdot)_{\mathcal{H}}$ defined as (\ref{eq.defprodscal}) for $(\cdot\,,\cdot)_{\boldsymbol{\mathcal{H}}}$ but without the integration in $\bf{x}$.
Applying $\bbF$ to the Schr\"{o}dinger equation (\ref{eq.schro}) leads us to introduce a family of operators $(\bbA_{\bf{k}} )_{\bf{k}\in \R^3}$ in $\mathcal{H}$ related to $\bbA$ by 
\begin{equation}\label{eq.AtoAk}
\mathcal{F}(\bbA {\bf U})(\cdot, \bf {k}) = \bbA_{\bf k}  \, \mathcal{F} {\bf U}(\cdot\,, {\bf k})  \quad \mbox{for a. e. } \bf{k}  \in \R^3.
\end{equation}
The domain of $\bbA_{\bf k}$ is $D(\bbA_{\bf {k}}) := \C^3 \times \C^{3} \times \mathcal{D}_{e} \times  \mathcal{ \widetilde D}_{e}\times  \mathcal{ D}_{m} \times \mathcal{\widetilde D}_{m}$ and
\begin{equation}\label{eq.opAk}
\quad\bbA_{\bf{k}}:= \ i \, \begin{pmatrix}
0 &\varepsilon_0^{-1}\, i \, \bf{k} \times & 0&-\mathcal{S}_{{e}} & 0 & 0\\[2pt]
- \mu_0^{-1}\,  i \, \bf{k} \times & 0 &0 & 0& 0& -\mathcal{S}_{{m}} \\[2pt]
0 & 0 &0  & \mathrm{I}_e &0 &0\\[2pt]
\mathcal{I}_e& 0 &-\xi^2 \mathrm{I}_{e} &0 &0 &0\\[2pt]
0 & 0 &0  & 0 &0 & \mathrm{I}_m \\[2pt]
0 & \mathcal{I}_m& 0 & 0 & -\xi^2 \mathrm{I}_{m} & 0 
\end{pmatrix}.
\end{equation}
The operators $\mathcal{S}_s$, $\mathcal{I}_{s}$, $\mathrm{I}_s$ and $\xi^2 \, \mathrm{I}_s$ for $s=e$ or $m$, and their corresponding domains, are defined as their bold version in (\ref{eq.defopblock}) but for functions of the variable $\xi$ only. It is easy to prove the
\begin{theo}
$\bbA_{\bf k}: D(\bbA_{\bf k})\subset\ \mathcal{H} \to \mathcal{H}$ is self-adjoint for all ${\bf k} \in \R^3.$
\end{theo}

\subsubsection{Resolvent of the reduced Hamiltonian}
In this paragraph, we compute the resolvent  $R_{\bf k}(\omega):=(\bbA_{\bf {k}}-\omega )^{-1}, \omega \notin  \R$ in order to obtain the spectrum of $\bbA_{\bf{k}}$.  Given ${\bf G}=({\bf e} , \, {\bf h}, \, {\bf p}, {\bf \widetilde{p}}, {\bf m}, {\bf \widetilde{m}})^{\top}\in\mathcal{H}$,  and setting ${\bf U}=({\bf E}, {\bf H}, {\bf \bbP}, \widetilde{ \bbP}, \bbM, \widetilde{ \bbM})^{\top},$ we have 
\begin{equation}\label{eq.linearres}
{\bf U} = R_{\bf k}(\omega)\,{\bf G} \quad \Longleftrightarrow \quad (\bbA_{\bf{k}}-\omega ) \,{\bf U}={\bf G}.
\end{equation}
Solving the linear system \eqref{eq.linearres} is straightforward (details are left to the reader) and leads to an explicit expression of $R_{\bf k}(\omega)$
(proposition \ref{prop.res}). Denoting ${\bf u} \cdot {\bf v}$ the inner product in $\C^3$, we set
\begin{equation}
\forall \; {\bf k} \neq 0 \mbox{ in } \R^3, \; \widehat{\bf {k}}={\bf k}/|\bf{k}|, \quad \forall \; {\bf u} \in \C^3, \; {\bf u}_{ \parallel}({\bf k}) :=  (\widehat{\bf {k}} \cdot {\bf u}) \,  \widehat{\bf {k}} \mbox{  and  } {\bf u}_{ \perp}({\bf k}) = {\bf u} - {\bf u}_{ \parallel}({\bf k}) .
\end{equation}
Then, from ${\bf G}$, we first define the vector fields of the variable $\xi$ (note that these depend linearly on ${\bf G}$)
$$
{\bf \widetilde{G}_p}= \frac{\omega \, {\bf \widetilde{ p }} -i \, \xi^2 {\bf p} }{\xi^2-\omega^2}, \,  \ {\bf \widetilde{G}_m}= \frac{\omega \, {\bf \widetilde{ m }} -i \, \xi^2 {\bf m} }{\xi^2-\omega^2} \ \mbox{ and } \ {\bf G_h}=-\mu_0 \,  \frac{i \, \mathcal{S}_{m}{\bf \widetilde{G}_m} + {\bf h}}{\omega \mu(\omega)} \, ,
$$
and introduce the operators, where $\mathcal{F}(\omega)=\omega^2 \varepsilon(\omega) \, \mu(\omega)$ has been defined in \eqref{defF}:
\begin{equation*}
\begin{array}{llll}
\bbS_{\omega,{\bf k}} {\bf G}:=& \dsp \hspace*{-0.3cm} -  \varepsilon_0 \Big[ \, \frac{i \, \mathcal{S}_{e} \bf{\widetilde{G}}_{p ,\parallel} ({\bf k}) +e_{\parallel}({\bf k}) }{\omega \varepsilon(\omega)} +  \omega \mu(\omega) \, \frac{i \, \mathcal{S}_{e} \bf{\widetilde{G}}_{p ,\perp} ({\bf k}) +e_{\perp}({\bf k}) }{\mathcal{F}(\omega)-|{\bf k}|^2} \, \Big]  -  \mu_0 \, {\bf k} \times  \frac{i \, \mathcal{S}_{m} \bf{\widetilde{G}}_{m} ({\bf k}) +{\bf h} }{\mathcal{F}(\omega)-|{\bf k}|^2},  \\[12pt]
\bbV_{\omega,{\bf k}}{\bf E}:=& \dsp \hspace*{-0.3cm} \Big( \, {\bf E}, \, \frac{{\bf k} \times {\bf E}}{\omega \mu(\omega)}, \, \frac{{\bf E}}{\xi^2-\omega^2}, -\frac{ i\, \omega {\bf E} }{\xi^2-\omega^2}, \, \frac{1}{{\xi^2-\omega^2}} \frac{{\bf k} \times {\bf E}}{\omega \mu(\omega)}, \,  - \frac{i \omega  }{\xi^2-\omega^2} \frac{{\bf k} \times {\bf E}}{\omega \mu(\omega)}\,  \Big)^t,\\ [12pt]
\bbT_{\omega,{\bf k}}{\bf G}:=& \dsp \hspace*{-0.3cm}  \Big( \, 0, \, {\bf G_h}, \, \frac{\omega {\bf p}  +i \, {\bf \widetilde{ p}}}{\xi^2-\omega^2}  , \, {\bf \widetilde{G}_p} , \, \frac{{\bf G_h}}{\xi^2-\omega^2} + \frac{\omega {\bf m } +i \, {\bf \widetilde{m} }}{\xi^2-\omega^2},  - \frac{i \omega \, {\bf G}_{h} }{\xi^2-\omega^2}+ {\bf  \widetilde{G}_m} \, \Big)^t.
\end{array}
\end{equation*}
The operator $\bbS_{\omega,{\bf k}}$ is continuous from $\mathcal{H}$ into $\C^3$, $\bbV_{\omega,{\bf k}}$ is continuous from $\C^3$ to $\mathcal{H}$ and  $\bbT_{\omega,{\bf k}}$ is continuous from $ \mathcal{H}$ to $\mathcal{H}$.
We point out that $\omega \varepsilon(\cdot)$ and $\omega \mu(\cdot)$, as non-zero Herglotz functions, can not vanish in the upper-half plane $\C^+$ (see lemma \ref{lemHerglotz}). Also, $\mathcal{F}(\omega)=|{\bf k}|^2$ has no solution $\omega$ in $\C^+$  by the same argument as in lemma \ref{lem_LL_ND}. Thus, the operators $\bbS_{\omega,{\bf k}}$, $\bbV_{\omega,{\bf k}}$ and $\bbT_{\omega,{\bf k}}$ are well-defined. The main result of this section is 
\begin{prop}\label{prop.res}
Let ${\bf k} \in \R^{3}\setminus \{0\}$. The resolvent of the self-adjoint operator $\bbA_{k}$ can be expressed as
\begin{equation} \label{resolvante}
R_{{\bf k}}(\omega) = \bbT_{\omega,{\bf k}} + \bbV_{\omega,{\bf k}}\,\bbS_{\omega,{\bf k}}, \quad\forall \; \omega \in \C\setminus \R.
\end{equation}
\end{prop}
\subsubsection{Spectrum of the reduced Hamiltonian}
 We shall use the following characterizations of $\sigma(\bbA_{\bf {k}})$ and the point spectrum $\sigma(\bbA_{\bf {k}})$ (see \cite{Ges-00}). First, as
$
\|R_{{\bf k}}(\omega)\|=\mathrm{d}(\omega, \sigma(\bbA_{\bf {k}}))^{-1}
$
where $\mathrm{d}(\cdot , \sigma(\bbA_{\bf {k}}))$ denotes the distance function to $\sigma(\bbA_{\bf {k}})$, 
$\omega_0 \in \R$ belongs to $\sigma(\bbA_{\bf {k}})$ if and only if $\|R_k(\omega_0+i \eta)\|$ blows up  when $\eta \searrow 0$. Second,
$\omega_0$ belongs to  $\sigma_p(\bbA_{\bf {k}})$ if and only if there exists ${\bf G}\in \mathcal{H}$ such that $\lim_{\eta \searrow 0}\eta \operatorname{Im}(R_{\bf k}(\omega_0+i\eta) {\bf G}, {\bf G})_{\mathcal{H}}>0$.
\subsubsection*{Case 1: $\omega_0 \in J=\operatorname{supp}(\nu_e)\cup \operatorname{supp}(\nu_m)$}~\\[-6pt]
Assume that $\omega_0 \in \operatorname{supp}( \nu_{e})$ (a similar proof applies for $\omega_0 \in \operatorname{supp}( \nu_{m})$) and let us show  that $\omega_0 \in \sigma(\bbA_{\bf k})$. Since $\sigma(\bbA_{\bf k})$ is a closed set, it is not restrictive to assume that $\omega_0 \neq 0$, except if $\omega_0=0$ is an isolated punctual mass ( this case is considered at the end of this paragraph). 
Let $p(\xi) \in L^2(\R; d \nu_e)$ be a scalar odd function (to be fixed later). Consider the particular choice
$${\bf G}=(0, 0, p(\xi) \, \widehat{\bf {k}},0,0,0) \in \mathcal{H}.$$ 
Thanks to \eqref{resolvante} and using again the expression of the operators, one computes that
\begin{equation}\label{eq.formres}
R_{\bf k}(\omega) {\bf G}=\Big( 0 , 0, \frac{\omega \,  p(\xi)\widehat{{\bf k}}}{\xi^2 -\omega^2},- \frac{i \xi^2  p(\xi)\widehat{{\bf k}}}{\xi^2 -\omega^2},0,0\Big)^t. 
\end{equation}
In particular
$ \dsp
\Big\| \, \frac{(\omega_0 + i \eta) \,   p(\xi)\widehat{{\bf k}} }{\xi^2 -(\omega_0 + i \eta)^2} \, \Big\|_{ \mathcal{V}_e}^2 \leq \|R_{\bf k}(\omega_0+i\eta) {\bf G}\|_{\mathcal{H}}^2 
$
i. e., using the definition of the norm in ${\cal V}_e$,
\begin{equation} \label{test}
|\omega_0 + i \eta|^2 \int_{\R} \frac{\xi^2 |p(\xi)|^2 \md \nu_e(\xi) }{|\xi^2-(\omega_0 + i \eta)^2|^2}  \leq  \|R_{\bf k}(\omega_0+i\eta) {\bf G}\|_{\mathcal{H}}^2.
\end{equation}
As $\omega_0 \in \operatorname{supp}(\nu_{e})$, by definition of the support \cite{Mat-99}, for any $\delta > 0$, the Borel set $A_\delta=\{\xi \in \R \mid |\xi^2-\omega_0^2| \leq \delta^2 \}$ satisfies $\nu_e(A_\delta) > 0$. Let us choose in (\ref{test}) $p = p_\delta$ odd such that $p_\delta^2 = {\bf 1}_{A_\delta}$ (this always possible as soon as $0\notin A_\delta$ which is the case for $\delta$ small enough since $\omega_0\neq0$). For such $\delta>0$ 
\begin{equation} \label{test2}
\forall \; \eta \in (0, \eta_0), \;  \quad |\omega_0 + i \eta|^2 \int_{A_\delta} \frac{\xi^2  \, \md \nu_e(\xi) }{|\xi^2-(\omega_0 + i \eta)^2|^2}  \leq  \|R_{\bf k}(\omega_0+i\eta) {\bf G}\|_{\mathcal{H}}^2.
\end{equation}
By definition of $A_\delta$, for $\xi \in A_\delta, \xi^2 + \omega_0^2 \leq \delta^2 + 2\omega_0^2$ so that
$$
|\xi^2-(\omega_0 + i \eta)^2|^2 =   (\xi^2 - \omega_0^2)^2 + \eta^4 + 2(\xi^2 + \omega_0^2 )\eta^2 \leq 
 \delta^{4} + \eta^4 + 2(\delta^2 + 2\omega_0^2 ) \eta^2.
$$
Thus (\ref{test2}) yields
$
\dsp  \frac{|\omega_0 + i \eta|^2}{\delta^{4} + \eta^4 + 2(\delta^2 + 2\omega_0^2 ) \eta^2} \int_{A_\delta}\xi^2  \md \nu_e(\xi) \leq  \|R_{\bf k}(\omega_0+i\eta) {\bf G}\|_{\mathcal{H}}^2
$ for $\delta$ small enough.\\[12pt]
Since $ \dsp \| {\bf G}\|_{\mathcal{H}}^2=\dsp \int_{A_\delta}\xi^2  \md \nu_e(\xi)>0$, making $\delta \searrow  0$ leads to $\dsp \frac{|\omega_0 + i \eta|^2}{\eta^4 + 4\omega_0^2  \eta^2} \leq  \|R_{\bf k}(\omega_0+i\eta)\|^2.$\\[12pt]
Thus (since $\omega_0 \neq 0$),  we obtain that  $\|R_k(\omega_0+i \eta)\|$ blows up  when $\eta \searrow 0$, which concludes the proof. \\[12pt]
Finally, let us pay a particular attention to the sets
\begin{equation} \label{defPePm}
{\cal P}_e = \{ \omega \in \R \; / \; \mbox{ such that } \nu_e\big( \{\omega\}\big) >0\}, \quad {\cal P}_m = \{ \omega \in \R \; / \; \mbox{ such that } \nu_m\big( \{\omega\}\big) >0\}.
\end{equation}
Note that in the case of local media, ${\cal P}_e$ and ${\cal P}_m$ are nothing but the sets of poles of $\varepsilon(\omega)$ and $\mu(\omega)$.
Let $\omega_0\in \mathcal{P}_e$ (a similar proof works for $\mathcal{P}_m$). As $\nu_e$ is symmetric, one deduces that $\nu_e(\{-\omega_0\})\neq 0$.
If one chooses ${\bf G}=(0, 0, p \, \hat{{\bf k}},0,0,0) \in \mathcal{H}$ with $ p=\bf{1}_{\{\omega_0\}}-\bf{1}_{\{-\omega_0\}}$, one obtains:
$$
\lim_{\eta \searrow 0}\eta \operatorname{Im}(R_{\bf k}(\omega_0+i\eta) {\bf G}, {\bf G})_{\mathcal{H}}=\lim_{\eta \searrow 0}\operatorname{Im}\left(\frac{2 \eta \, (\omega_0+i \, \eta) \, \omega_0^2}{-2i  \eta \, \omega_0+\eta^2\, \omega_0^2}\right)=\omega_0^2>0,
$$
which implies  $\omega_0 \in \sigma_p(\bbA_{\bf {k}})$. When $\omega_0=0\in {\cal P}_e $, one let the reader to show that if one takes
${\bf G}=(0, 0,0,  \bf{1}(\{0\})\, \hat{{\bf k}},0,0,0) \in \mathcal{H}$, one has $\lim_{\eta \searrow 0}\eta \operatorname{Im}(R_{\bf k}(\omega_0+i\eta) {\bf G}, {\bf G})_{\mathcal{H}}=1$ and thus $0\in \sigma_p(\bbA_{\bf {k}})$.
\subsubsection*{Case 2: $\omega_0 \in  I=\R \setminus J, $}~\\[-6pt]  
First note that the presence of $\xi^2 - (\omega_0 + i \eta)^2$ in denominators of the expression of $\bbT_{\omega_0 + i \eta,{\bf k}}$ and $\bbV_{\omega_0 + i \eta,{\bf k}}$ does not matter since $\omega_0$ is outside the supports of the measures $\nu_e$ and $\nu_m$ (and also $-\omega_0\not \in J$, since $J$ is symmetric with respect to the origin). Next, using remark \ref{eq.hergprop}, we know that the Herglotz functions $\omega \varepsilon(\cdot)$ and $\omega \mu(\cdot)$ have a continuous extension from the upper-half plane to the open set $I$. These extensions are real-valued and analytic on $I$. Thus, it is readily seen from \eqref{resolvante} and the expression of $\bbT_{\omega,{\bf k}}$, $\bbV_{\omega,{\bf k}}$ and $\bbS_{\omega,{\bf k}}$ (details are left to the reader) that the resolvent $R_{\bf k}(\omega_0 +i\, \eta)$  blows up when $\eta \to 0$ if and only if $\omega_0$ is a zero of the functions $\omega \varepsilon(\omega)$, $\omega \mu(\omega)$ or $\mathcal{F}(\omega)-|{\bf k}|^2$.
According to the notation used for local media (see definition \ref{defPolesSZeroes} and corollary \ref{smoothness}), we set
\begin{equation} \label{defsets}
\mathcal{Z}_{e}=\{ \omega \in I \mid \omega \varepsilon(\omega)=0\}, \quad  \mathcal{Z}_{m}=\{ \omega \in I \mid \omega \mu(\omega)=0\}  \mbox{  and  } \Omega({\bf k}) = \{ \omega \in I \mid \mathcal{F}(\omega)=|{\bf k}|^2\}.
\end{equation}
By evenness and analyticity of $\varepsilon(\omega)$ and $\mu(\omega)$ along $I$, the sets ${Z}_{e}$, $\mathcal{Z}_{m}$ and $\Omega({\bf k})$ are symmetric with respect to the origin and made of isolated points in $I$. The main difference with the case of local media is that one cannot exclude the fact that these sets be infinite (but always countable)!
$\mathcal{Z}_{e}$ and $\mathcal{Z}_{m}$ contains only simple zeros by lemma \ref{lemHerglotz}. For ${\bf k}\neq0$, the zeros of $\mathcal{F}(\cdot)-|{\bf k}|^2$ are simple too. The argument is the same as for local media  (lemma \ref{simplicity}). Thanks to these properties, it is clear that $\|R_k(\omega_0+i \eta)\|$ remains bounded for any $\eta > 0$ when $\omega_0 \notin \mathcal{Z}_{e} \cup \mathcal{Z}_{m} \cup \Omega({\bf k})$ while one can show that for any $\omega_0 \in \mathcal{Z}_{e} \cup \mathcal{Z}_{m} \cup \Omega({\bf k})$, $\|R_k(\omega_0+i \eta)\| \geq C_0 \; \eta^{-1}, C_0>0$ (again details are left to the reader). Thus $\omega_0 \in \sigma(\bbA_{\bf k})$ and is thus an eigenvalue since it is an isolated point in $\sigma(\bbA_{\bf k})$ \cite{Kato}.\\[12pt] 
Let us summarize what we have obtained in the following 
\begin{theo}\label{prop.specAk}
Let $\bf{k}\in \R^3\setminus \{0\}$, $J = \operatorname{supp}(\nu_e)\cup \operatorname{supp}(\nu_m)$ and $I=\R \setminus J$. Let $\mathcal{Z}_e$ , $\mathcal{Z}_m$, $\mathcal{P}_e$ , $\mathcal{P}_m$ and $\Omega({\bf k})$ defined by (\ref{defsets}) and (\ref{defPePm}). Then
\begin{equation}\label{spectreAk}
\sigma(\bbA_{\bf k})=J \cup \Omega({\bf k}) \cup \mathcal{Z}_e  \cup \mathcal{Z}_m, \quad 
\sigma_p(\bbA_{\bf {k}})  \cap I =  \mathcal{Z}_{e} \cup \mathcal{Z}_{m} \cup \Omega({\bf k}), \quad \mathcal{P}_e  \cup \mathcal{P}_m \subset \sigma_p(\bbA_{\bf {k}}). 
\end{equation}
\end{theo}

\subsubsection{Spectrum of the Hamiltonian}
\begin{theo}\label{th.spec}
Let $J = \operatorname{supp}(\nu_e)\cup \operatorname{supp}(\nu_m)$, $I=\R \setminus J$ and $\mathcal{Z}_e$ , $\mathcal{Z}_m$, $\mathcal{P}_e$ and  $\mathcal{P}_m$  defined by (\ref{defsets}) and (\ref{defPePm}).  Let ${\cal F}(\omega) := \omega^2 \, \varepsilon(\omega) \, \mu(\omega) : I \rightarrow \R$. Then the spectrum $\sigma(\bbA)$ and the point spectrum $\sigma_p(\bbA)$ of the operator $\bbA$ satisfy:
\begin{equation}\label{eq.specA}
(i) \; \sigma(\bbA)=J \cup {\mathcal{F}^{-1}(\R^{+})} , \quad (ii) \ \mathcal{P}_e \cup \mathcal{P}_m  \cup  \mathcal{Z}_e  \cup \mathcal{Z}_m\subset \sigma_p(\bbA) \quad (iii) \; \sigma_p(\bbA)\cap I = \mathcal{Z}_e \cup \mathcal{Z}_m.
\end{equation}
\end{theo}
\begin{proof}
The relation \eqref{eq.AtoAk} means that the self-adjoint operator $\bbA_{\oplus}=\bbF^{-1}\bbA \bbF$ is decomposable on the family of self-adjoint operators $(\bbA_k)_{k\in \R^3}$ with respect to the Lebesgue measure $\nu$ on $\R^3$ \cite{Ree-78}. As the Fourier transform $\bbF$ is unitary, $\bbA$ and $\bbA_{\oplus}$ have the same spectrum and punctual spectrum. Thus, using  theorem XIII.85 of \cite{Ree-78} which relates the spectrum of $\bbA_{\oplus}$ and the spectra of the operators $\bbA_{{\bf k}}$, one deduces that:
\begin{eqnarray}
\sigma(\bbA)&=&\{ \omega \in \R \mid \forall \varepsilon>0 \mid  \nu( \{ {\bf k} \in \R^3 \mid (\omega-\varepsilon,\omega+\varepsilon) \cap \sigma(\bbA_{\bf k})\neq \varnothing \})>0 \} ,\label{eq.specdecomp} \\[3pt]
\sigma_p(\bbA)&=&\{ \omega \in \R \mid  \nu( \{{\bf k} \in \R^3 \mid \omega\in  \sigma_p(\bbA_{\bf k})\})>0 \} \label{eq.specpuncdecomp} .
\end{eqnarray}

{\bf Step 1: $\sigma(\bbA) \subset  J \cup \mathcal{F}^{-1}(\R^+)$.  } As ${\cal F}: I \rightarrow \R$ is continuous, $\mathcal{G}=\mathcal{F}^{-1}(\R^-_*)$ is an open subset of $I$ which, by proposition \ref{prop.specAk}, does not intersect $\sigma(\bbA_{{\bf k}})$ for all ${\bf k} \neq 0$. Thus, one deduces from \eqref{eq.specdecomp} that $\mathcal{G}$ does not intersect $\sigma(\bbA)$, in other words that $\sigma(\bbA) \subset \R \setminus \mathcal{G} = J \cup \mathcal{F}^{-1}(\R^+)$. \\[12pt]
{\bf Step 2 :} $J \subset \sigma(\bbA)$ and $\mathcal{P}_e \cup \mathcal{P}_m  \cup  \mathcal{Z}_e  \cup \mathcal{Z}_m\subset \sigma_p(\bbA)$ (i. e. (\ref{eq.specA})(ii)). Indeed, by proposition \ref{prop.specAk} again,  $J \subset \sigma(\bbA_{{\bf k}})$ for any ${\bf k}$, i.e. for any $\omega \in J$ and $\varepsilon > 0$,  $ \omega \in (\omega-\varepsilon,\omega+\varepsilon) \cap \sigma(\bbA_{\bf k})$. Hence 
$
\{ {\bf k} \in \R^3 \mid (\omega-\varepsilon,\omega+\varepsilon) \cap \sigma(\bbA_{\bf k})\neq \varnothing \} = \R^3
$ and thus, by \eqref{eq.specdecomp}, $J \subset \sigma(\bbA)$. Similarly, proposition \ref{prop.specAk} says that $\mathcal{Z}_{e} \cup \mathcal{Z}_{m} \cup \mathcal{P}_e  \cup \mathcal{P}_m \subset \sigma_p(\bbA_{\bf {k}})$ for any ${\bf k} \in \R^3$. With \eqref{eq.specpuncdecomp} this implies the inclusion (\ref{eq.specA})(ii). \\[12pt]
As $\mathcal{F}^{-1}(\R^+) = \mathcal{F}^{-1}(\R^+_*) \cup \mathcal{F}^{-1}(0) \equiv \mathcal{F}^{-1}(\R^+_*) \cup \mathcal{Z}_{e} \cup \mathcal{Z}_{m}$,  if we prove that $\mathcal{F}^{-1}(\R^+_*) \subset \sigma(\bbA)$, since $\mathcal{Z}_{e} \cup \mathcal{Z}_{m} \subset \sigma(\bbA)$, we shall have proven that  $\mathcal{F}^{-1}(\R^+) \subset \sigma(\bbA)$, that is to say (\ref{eq.specA})(i) with step 2. Finally, if we prove that  $\mathcal{F}^{-1}(\R^+_*)$ does not contain any eigenvalue, we shall have proven (\ref{eq.specA})(iii). These observations lead us to the last step of our proof.\\[12pt]
{\bf Step 3 :}  $\mathcal{F}^{-1}(\R^+_*) \subset \sigma(\bbA)$ but contains no eigenvalue.
Let $\omega\in \mathcal{F}^{-1}(\R^+_*)$, then $\mathcal{F}(\omega)=R$ for $R>0$. Since $S_{R}=\{{\bf k} \in \R^3\, /  |{\bf k}|^2=R\}$ has zero Lebesgue measure, one deduces from proposition \ref{prop.specAk} and  \eqref{eq.specpuncdecomp}, that $\omega$ is not an eigenvalue . As $\mathcal{F}$ is $C^{\infty}$ on $I$ and $\mathcal{F}'(\omega)\neq 0$ (same proof as for lemma \ref{lemmetechnique}), there exists two open sets $U_{\omega}\subset I$ and $V_{R}\subset \R^+_*$ such that $\omega \in U_\omega$ and $R \in V_R$ and $\mathcal{F}:U_{\omega} \to V_{R}$ admits a $C^{\infty}$ inverse. Thus for all $\varepsilon>0$ such that $(\omega-\varepsilon, \omega+\varepsilon)\subset U_{\omega}$, $\mathcal{F}\big((\omega-\varepsilon, \omega+\varepsilon)\big)$ is an open set which contains $(R-\eta,R+\eta)\subset V_R$ for $\eta$ small enough. Let us set $C_{R,\eta}=\{{\bf k} \in \R^3, R-\eta\leq  |{\bf k}|^2\leq R+\eta\}$. For any ${\bf k} \in C_{R,\eta}$, $|{\bf k}|^2 \in \mathcal{F}\big((\omega-\varepsilon, \omega+\varepsilon)\big)$, i. e. ${\cal F}(\omega_{\bf k}) = |{\bf k}|^2$ for some 
$\omega_{\bf k} \in (\omega-\varepsilon, \omega+\varepsilon)$, that is to say, see (\ref{defsets}), that $\omega_{\bf k} \in \Omega({\bf k})$, thus $\omega_{\bf k} \in \sigma(\bbA_{\bf k})$ by proposition \ref{prop.specAk}. Thus, $(\omega -\varepsilon, \omega+\varepsilon)\cap \sigma(\bbA_{\bf k})$ contains $\omega_{\bf k}$ for any ${\bf k} \in C_{R, \eta}$ which means that $C_{R, \eta} \subset \{ {\bf k} \in \R^3 \mid (\omega-\varepsilon,\omega+\varepsilon) \cap \sigma(\bbA_{\bf k})\neq \varnothing \}$. Since $\nu(C_{R, \eta})>0$, one concludes with \eqref{eq.specdecomp} that $\omega\in \sigma(\bbA)$.
\end{proof}
\begin{rem}
The complementary of the support $I$ is an open set. Thus it can be decomposed as a countable union of disjoint open intervals: $I=\cup_{n=0}^{N}(a_n,b_n)$ (where $N$ can be finite or infinite). All these intervals are symmetric with respect to $0$. On each interval $(a_n,b_n)$, 
 $\mathcal{F}$ is a real-valued analytic function. By adapting \ref{lemmetechnique}, we prove that it is strictly monotonous wherever it is positive, and by adapting lemma \ref{estimnpmnz}, it vanishes at most at two points. One has
\begin{itemize}
\item $(a_n,b_n)\cap \sigma(\bbA)=\{ \lambda \in (a_n,b_n) \mid \mathcal{F}(\omega)\geq 0 \}$ and $(a_n,b_n)\cap \mathcal{G}=\{ \lambda \in (a_n,b_n) \mid \mathcal{F}(\omega)<0 \}$,
\item There is at most two different eigenvalues of $\mathcal{Z}_e \cup \mathcal{Z}_m$ in each interval.
\end{itemize}
Using all these properties, it is possible to sketch the different possible graphs for $\cal{F}$ in each $(a_n,b_n)$ as in section \ref{sec-dispersionanalysis} for local media. Indeed, the graphs are similar up to the difference that $\cal{F}$ can admit also a finite limit at the border $(a_n,b_n)$ when it is positive and no limit when it is negative.
\end{rem}
\begin{rem} For local Lorentz materials \eqref{Lorentzlaws}, one has (see (\ref{LorentzMeasures}))
$$
\mbox{supp }(\nu_e)=\mathcal{P}_e, \quad \mbox{supp }(\nu_m)=\mathcal{P}_m, \quad \sigma(\bbA)=\mathcal{F}^{-1}(\R^{+}) \cup \mathcal{P}_e \cup \mathcal{P}_m , \quad \sigma_p(\bbA)=\mathcal{P}_e \cup \mathcal{P}_m  \cup \mathcal{N}_e \cup \mathcal{N}_m.
$$
One can make a link with the Fourier / plane wave analysis of local media, performed in section \ref{sec-hom}. In particular, one sees that coincides with the set ${\cal S}$ defined by (\ref{spectrum}). Moreover, it is worthwhile mentioning that, in the point spectrum, $\mathcal{P}_e \cup \mathcal{P}_m$ corresponds to the static electric and magnetic modes and $\mathcal{Z}_e \cup \mathcal{Z}_m$ to the curl-free static modes, as they have been defined in section \ref{sec-WP}.
\end{rem}
\subsection{The case of lossy passive media: an example} \label{Sec_Dissipative}
As seen previously, we can find a conservative augmented formulation for any passive systems, in particular dissipative ones.
However, dissipation can be obtained only if measures $\nu_e$ or $\nu_m$ has a continuous part. This is the case for a dissipative Drude model (\ref{epsmulossyLorentz}) with $\omega_e = \omega_m=0$ and $\alpha_e,\, \alpha_m,\, \Omega_e,\, \Omega_m > 0$:
\begin{equation} \label{epsmulossyLorentz}
\dsp \varepsilon(\omega) = 	\varepsilon_0 \, \Big( 1 - \frac{\Omega_e^2}{ i \, \alpha_e \, \omega + \omega^2} \Big),
\quad 
\dsp \mu( \omega) = \mu_0 \, \Big( 1 - \frac{\Omega_m^2}{i \, \alpha_m \, \omega +\omega^2} \Big).
\end{equation}
\subsubsection{Dissipative and conservative formulations}
The dissipative formulation is obtained by introducing the polarization ${\bf P}$ and the magnetization ${\bf M}$:
\begin{equation} \label{dissipativesystem}	
\left\{ \begin{array}{ll}
	\varepsilon_0 \, \partial_t \mathbf{E}+\varepsilon_0\Omega_e^2\partial_t\mathbb{P}- {\bf rot} \, \mathbf{H}=0,& \quad
	\mu_0 \, \partial_t \mathbf{H} + \mu_0\Omega_m^2\partial_t\mathbb{M} +{\bf rot } \, \mathbf{E}=0,\\[12pt]
	\partial_{t}^2\mathbb{P}+ \alpha_e \,  \partial_t \mathbb{P}= \mathbf{E},& \quad
	\partial_{t}^2\mathbb{M}+ \alpha_m \,  \partial_t \mathbb{M}= \mathbf{H},
\end{array} \right.
\end{equation}
To this system, we naturally associate the following energy
\begin{equation}\label{defenergylocal}
\dsp {\cal E}_{loc}(t) :=  {\cal E}(t) +\frac{\varepsilon_0\Omega_{e}^2}{2} \, \int_{\R^3} |\partial_t {\mathbb{ P}}|^2 \md x + 
\frac{\mu_0\Omega_{m}^2}{2}\, \int_{\R^3} |\partial_t {\mathbb{ M}}|^2 \md x 
\end{equation}
	where $\dsp {\cal E}(t) :=  \frac{1}{2}\int_{\R^3} \big( \varepsilon_0 \, |{\bf E}|^2  + \mu_0 \, |{\bf H}|^2 \, \big) \; d{\bf x}$, the standard electromagnetic energy. One easily checks that
\begin{equation} \label{energydecay}
\frac{d}{dt}\, \mathcal{E}_{loc}(t) + \alpha_e\varepsilon_0 \Omega_e^2\,  \int_{\R^3} |\partial_t\mathbb{P}|^2d\mathbf{x} +  \alpha_m \mu_0 \Omega_m^2\,  \int_{\R^3} |\partial_t\mathbb{M}|^2d\mathbf{x} = 0,	
\end{equation}
which proves the decay in time of the energy ${\cal E}_{loc}(t)$. The conservative augmented formulation corresponds to  absolutely continuous measures (with respect to Lebesgue's measure) $\nu_e$ and $\nu_m$  defined by 
\begin{equation}\label{measuresDrude}
	\md \nu_e(\xi) = \frac{\alpha_e \, \Omega_e^2}{\pi} \; \frac{\md \xi}{\xi^2 + \alpha_e^2}, \quad \md \nu_m(\xi) = \frac{\alpha_m \, \Omega_m^2}{\pi} \; \frac{\md \xi}{\xi^2 + \alpha_m^2}.
	\end{equation}
This can be deduced from (\ref{measure}) but also from the following Nevalinna representation (left to the reader) 
\begin{equation} \label{NevanlinnaDrude}
1 - \frac{\Omega^2}{ i \, \alpha \, \omega + \omega^2 }= 1 + \frac{\alpha \, \Omega^2 }{\pi} \int\limits_{-\infty}^{+\infty}\frac{d\xi}{(\xi^2-\omega^2)(\xi^2+\alpha^2)} \; .
\end{equation}
Note that the measures $\nu_e$ and $\mu_e$ have finite masses, respectively $\Omega_e^2$ and $\Omega_m^2$, and their support is all $\R$ so that the spectrum of the corresponding Hamiltonian $\bbA$ is the whole real line (see theorem \ref{th.spec}).
\subsubsection{Numerical simulations} \label{numerics}
We have performed 2D simulations in the $(x,y)$-plane, considering the 2D Maxwell system for the electromagnetic field $(E_x, E_y,  H_z)$. We use the scaling $\varepsilon_0=\mu_0=1$ and consider $\Omega_e = \Omega_m =1$ and two sets of values for the absorption parameters: $\alpha_e = \alpha_m =0.1$ and $\alpha_e = \alpha_m =1$. The domain of computation is the square $[-0.5, \; 0.5]\times [-0.5, \; 0.5]$. On the boundary of the domain, a perfectly conducting boundary condition $\mathbf{E}\times \mathbf{n}=0$ is used. The system is initialized with $E_x(\mathbf{x},0)=E_y(\mathbf{x},0)=\mathrm{e}^{-300 (x^2+y^2)}$ and $H_z(\mathbf{x},0)=\mathrm{e}^{-200 (x^2+y^2)}$, while for the rest of the unknowns zero initial conditions are used. 
We compared results obtained separately with the two systems (\ref{dissipativesystem})  ('exact solution') and (\ref{Lorentzsystemgene}), with measures (\ref{measuresDrude}) ('approximate solution'). The equations were discretized as follows:
\begin{itemize}  
\item For the solution of (\ref{dissipativesystem}) we use the Yee scheme  (see for instance \cite{Taflove}), where dispersive terms are discretized with the help of the trapezoid rule. We take the mesh size $\Delta x=\Delta y$ and $\Delta t=\frac{\Delta x}{2}$. 
\item The discretization of (\ref{Lorentzsystemgene}) requires an additional step consisting in the approximation of the $\xi$-integrals. In order to do so, we first perform the change of variables $\xi = \alpha \, \tan \tau$ in (\ref{NevanlinnaDrude}) 
\begin{align*}
1 - \frac{\Omega^2}{ i \, \alpha \, \omega + \omega^2 } = 1 + \frac{\Omega^2}{\pi}
\int\limits_{-\frac{\pi}{2}}^{\frac{\pi}{2}}\frac{d\tau}{\alpha^2\tan^2\tau-\omega^2}=1 - 2 \frac{\alpha \, \Omega^2}{\pi} \int\limits_{-\frac{\pi}{2}}^{\frac{\pi}{2}}\frac{d\tau}{\alpha^2\tan^2\tau-\omega^2}.
\end{align*}
and discretize the above integral, using the Gauss-Legendre quadrature rule \cite[pp. 177-178]{Gauss} on the interval $\left(0,\frac{\pi}{2}\right)$, with quadrature  weights $\{w_\ell, 1 \leq \ell \leq N_q\}$ and quadrature nodes $\{\tau_\ell, 1 \leq \ell \leq N_q\}$:
\begin{align*}
1 - \frac{\Omega^2}{ i \, \alpha \, \omega + \omega^2 } \approx 1 + 2 \frac{\Omega^2}{\pi} \; \sum\limits_{\ell=1}^{N_q} \; 
\frac{w_\ell}{\alpha^2\tan^2\tau_\ell-\omega^2}.
\end{align*}
The model (\ref{Lorentzsystemgene}) is thus approximated by a local Lorentz system (\ref{Lorentzsystem}) with $N_e=N_m=N_q$ and
$$
\omega_{e, \ell} = \alpha_e \, \tan \tau_\ell, \qquad \omega_{m, \ell} = \alpha_m \, \tan \tau_\ell,\qquad  \Omega_{e,\ell}^2 = 2 w_\ell \qquad \frac{\Omega_{e}^2}{\pi}, \qquad  \Omega_{m,\ell}^2 = 2 w_\ell \;  \frac{ \Omega_{m}^2}{\pi} .
$$
\end{itemize}
We first tested the approximation of (\ref{dissipativesystem}) with the discrete augmented system. In figures
 \ref{fig:Hz1} and \ref{fig:Hz2}, we represent the evolution of $H_z$ at the origin as a
  function of time in the time window $[8,10]$, for different values of $N_q$. In both cases we observe the convergence of
   the approximate solution (red curves) towards the exact solution (blue curve) when $N_q$ increases. In the case of
    small absorption, figure \ref{fig:Hz1}, the convergence is attained quite quickly, but for large absorption, figure \ref{fig:Hz2}, fairly accurate solution is obtained only with 40 quadrature points (which is not surprising since the
	 augmented models are non-dissipative). In both cases, the approximation of the exact dissipative model with the 
	  discrete non-dissipative model does not provide an efficient numerical method. However, our exact dissipative model is already itself local, which
is not always the case. In figure \ref{fig:totalenergy}, we plot the variations of the energy $\mathcal{E}_{loc}(t)$ (\ref{defenergylocal}) as a function of time for the two cases. The curves confirm the theoretical decay of this energy, which, as expected, is much stronger decay in the case $\alpha_e = \alpha_m =1$. In figure \ref{fig:totalenergy2}, we study the variations of the energy $\mathcal{E}(t)$ for the exact solution (red curve), compared with the variations of the energy $\mathcal{E}(t)$ for the approximate solution for different values of $N_q$ (blue curve). We see that, contrary to $\mathcal{E}_{loc}(t)$, $\mathcal{E}(t)$ is not a decreasing function of time, even though it tends to $0$ when $t \rightarrow + \infty$, as it will be proven in the next section.
\begin{figure}
\includegraphics[width=\textwidth]{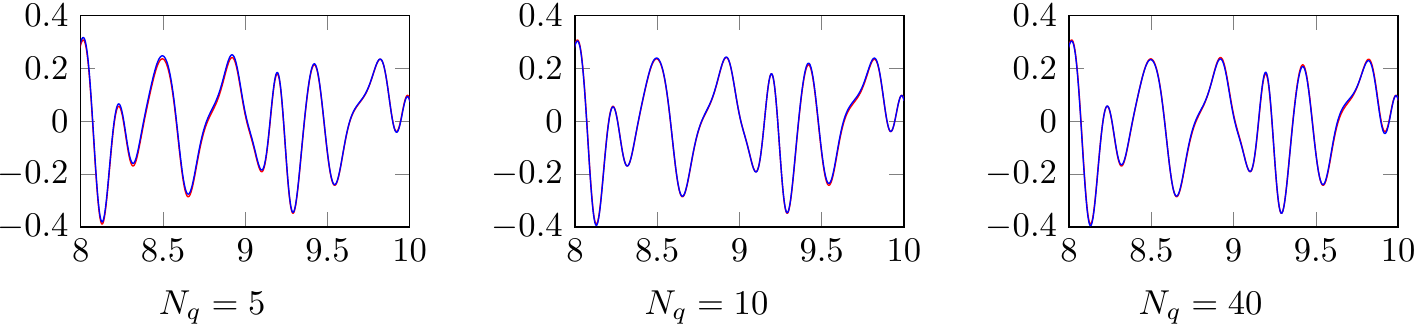}
\caption{Graphs of $t \rightarrow H_z(0,t)$ for $\alpha_e = \alpha_m =0.1$ and $N_q = 5, 10, 40$ (Left to right)}
\label{fig:Hz1}
\end{figure}
\begin{figure}
\includegraphics[width=\textwidth]{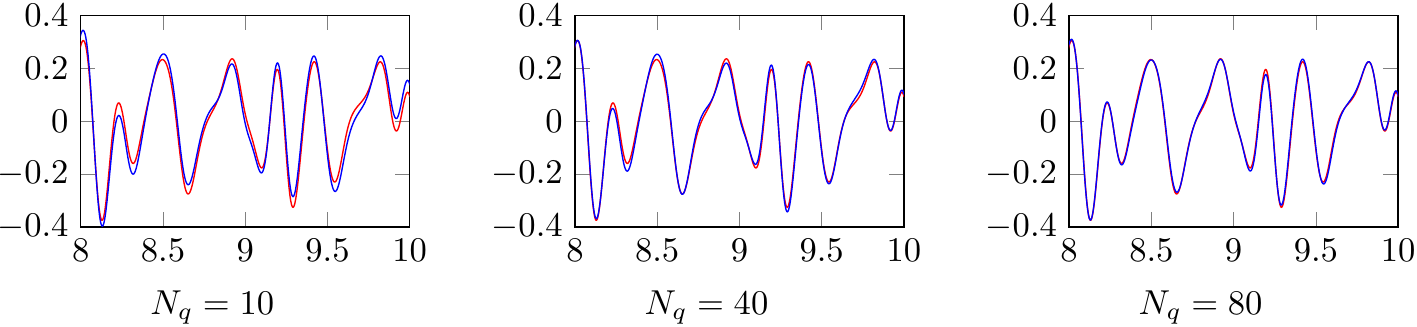}
\caption{Graphs of $t \rightarrow H_z(0,t)$ for $\alpha_e = \alpha_m =1$ and $N_q = 10, 20, 80$ (Left to right)}
\label{fig:Hz2}
\end{figure}
\begin{figure}
\centerline{
\includegraphics[height=4cm]{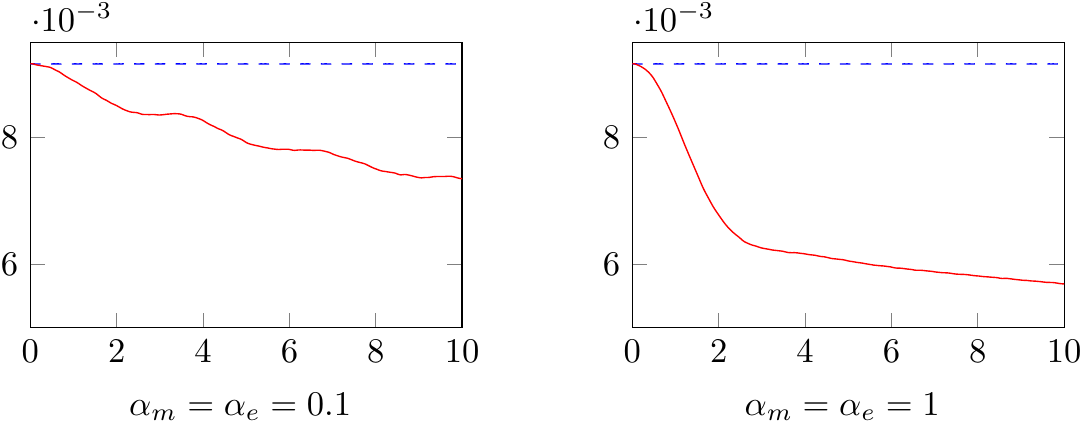}}
\caption{
Evolution of the energy $\mathcal{E}_{loc}(t)$ for $\alpha_e = \alpha_m =0.1$ (left) and $\alpha_e = \alpha_m =1$ (right).}
\label{fig:totalenergy}
\end{figure}
\vspace{-10pt}
\begin{figure}
\includegraphics[height=4cm]{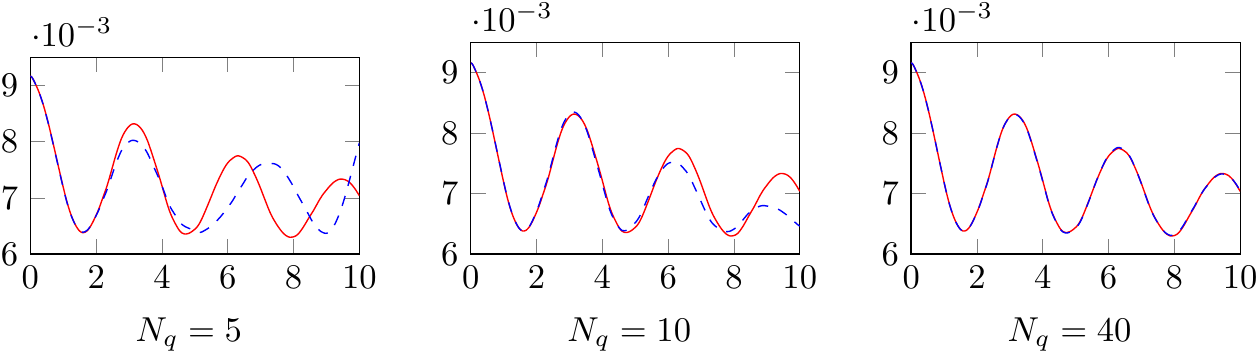}\\	\includegraphics[height=4cm]{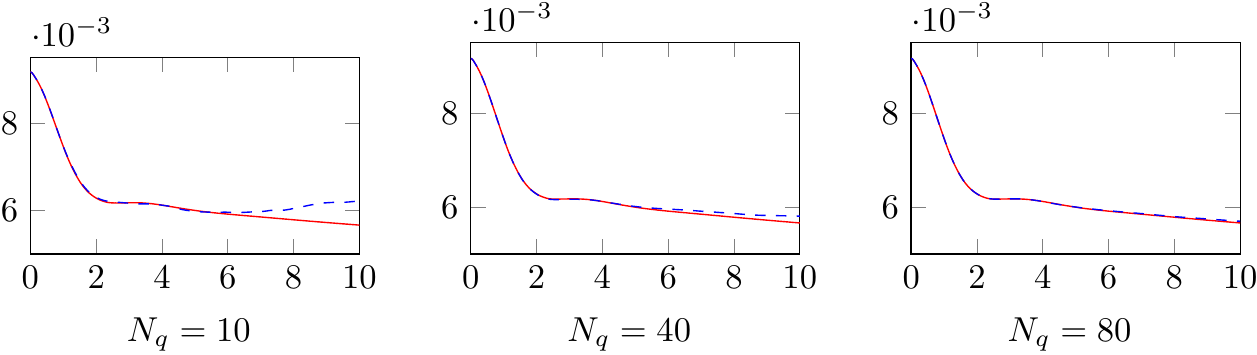}
\caption{
Evolution of the electromagnetic energy $\mathcal{E}(t)$ for $\alpha_e = \alpha_m =0.1$ (top) and $\alpha_e = \alpha_m =1$ (bottom).}
\label{fig:totalenergy2}
\end{figure}
\subsubsection{Energy analysis of the dissipativity}
\begin{lema} \label{dissipationenergy}
Let $(\mathbf{E}, \mathbf{H}, \mathbb{P}, \mathbb{M})$ solve (\ref{dissipativesystem})  with initial conditions 
\begin{align*}
(\mathbf{E}_0, \; \mathbf{H}_0)\in (H_{\operatorname{rot}}(\mathbb{R}^3))^2, \quad \mbox{satisfying } \quad  (\mathbf{rot}\mathbf{E}_0, \; \mathbf{rot}\mathbf{H}_0)\in (H_{\operatorname{rot}}(\mathbb{R}^3))^2.
\end{align*}
Then the electromagnetic energy $\mathcal{E}(t)$ satisfies
$ \dsp
\lim\limits_{t\rightarrow +\infty}\mathcal{E}(t)=0.
$
\end{lema}
\begin{proof}
Let us demonstrate that $\|\mathbf{E}(\cdot,t)\|_{L^2}\rightarrow 0$ as $t\rightarrow +\infty$. The proof for $\|\mathbf{H}(\cdot,t)\|_{L^2}\rightarrow 0$ is similar. From (\ref{energydecay}), we infer
that the energy ${\cal E}_{loc}(t)$ decays in time towards a limit ${\cal E}_{loc}^\infty \geq 0$ and that
\begin{equation} \label{estimate1}
\alpha_e\varepsilon_0 \Omega_e^2 \,  \int_{0}^{\infty}\|\partial_t\mathbb{P}(\cdot,t)\|_{L^2}^2 \; dt + \alpha_m \varepsilon_0 \Omega_m^2 \, \int_{0}^{\infty}\|\partial_t\mathbb{M}(\cdot,t)\|_{L^2}^2 \; dt \leq \frac{1}{2}\; \|\mathbf{E}_0\|_{L^2}^2 +  \frac{1}{2}\; \|\mathbf{H}_0\|_{L^2}^2< + \infty.	
\end{equation}
Applying the same reasoning to $(\partial_t \mathbf{E}, \partial_t \mathbf{H}, \partial_t \mathbb{P}, \partial_t \mathbb{M})$, 
which also solves (\ref{dissipativesystem}) (with different initial conditions), we obtain 
\begin{equation*} 
\alpha_e\varepsilon_0 \Omega_e^2 \,  \int_{0}^{\infty}\|\partial_t^2\mathbb{P}(\cdot,t)\|_{L^2} \; dt + \alpha_m \mu_0 \Omega_m^2 \, \int_{0}^{\infty}\|\partial_t^2\mathbb{M}(\cdot,t)\|_{L^2} \; dt \leq \frac{1}{2} \; \| {\bf rot} \, \mathbf{E}_0\|_{L^2}^2 +  \frac{1}{2} \; \| {\bf rot} \, \mathbf{H}_0\|_{L^2}^2,
\end{equation*}
which is finite. 
Writing
$
\|\partial_t\mathbb{P}(\cdot,T)\|_{L^2}^2 = 2 \int_0^T  \big(\partial^2_t\mathbb{P}(\cdot,t), \partial_t\mathbb{P}(\cdot,t) \big)_{L^2} dt
$ shows, thanks to (\ref{estimate1}) and the above identity that $\|\partial_t\mathbb{P}(\cdot,t)\|_{L^2}^2$ has a limit when $\rightarrow + \infty$. This limit is necessarily $0$ because of (\ref{estimate1}).
Repeating the above arguments for the second time derivatives of the fields allows us (using the additional second order space regularity of $(\mathbf{E}_0, \mathbf{H}_0)$) to show the same result for 
$\|\partial_t^2\mathbb{P}(\cdot,t)\|_{L^2}^2$.  Therefore,  we have proven that
$$
\lim\limits_{t\rightarrow+\infty}\|\partial^2_t \mathbb{P}(\cdot,t)\|_{L^2}= \lim\limits_{t\rightarrow+\infty}\|\partial_t \mathbb{P}(\cdot,t)\|_{L^2}= 0.
$$
Finally, using the first equation in the second line of (\ref{dissipativesystem}), we show that $\lim\limits_{t\rightarrow+\infty}\|\mathbf{E}(\cdot,t)\|_{L^2} = 0$. \end{proof}

{\appendix
\section{Characterization of non-dissipative local materials: the proof of theorem \ref{thm.equivpassive}}
Let $\varepsilon(\omega)$ and $\mu(\omega)$ be associated to a local non-dissipative material and ${\cal F}(\omega) :=  \omega^2\,\varepsilon(\omega)\,\mu(\omega)$, with set of poles ${\cal P}$ and set of zeros ${\cal Z}$. Saying that (\ref{relationdispersion}) has only real solutions implies that
\begin{equation}\label{propF}
\forall \omega \in \C^+\setminus {\cal P}, \quad  {\cal F}(\omega) \in	\C \setminus \R^+.
\end{equation}
Let $a < b$ de two real numbers that do not belong to ${\cal P}  \cup {\cal Z}$. Let us denote $n_P(a,b)$ (resp. $n_Z(a,b)$) the numbers of zeros of ${\cal F}(\omega)$ in the 
interval $(a,b)$, counting multiplicity. 
$$
 n_P(a,b) := \sharp \; {\cal P} \cap (a,b), \quad  n_Z(a,b) := \sharp  \; {\cal Z} \cap (a,b)
$$
\begin{lema} \label{estimnpmnz} For any $a < b$ that do not belong to $\mathcal{P} \cup \mathcal{Z}$, one has the inequality
\begin{equation} \label{inegnpnz}
|n_P(a,b) - n_Z(a,b)| \leq 2.
\end{equation}
\end{lema}
\begin{proof}	
Thanks to the well-known argument principle \cite{henrici} and the analyticity properties of ${\cal F}(\omega)$ we have 
$$
\forall \; \delta > 0, \quad n_Z(a,b) - n_P(a,b) =  \frac{1}{2i\pi} \int_{\gamma^\delta} \frac{{\cal F}'(\omega)}{{\cal F}(\omega)} \; d \omega, \quad \gamma_\delta = \partial R_\delta, \quad R_\delta = [a,b] \times [-\delta, \delta].
$$
Since $\gamma_\delta = [a,b] \times \{-\delta, \delta\}  \cup \, \gamma^0_{\delta}$,  $ \gamma_{\delta}^0 = \{a , b\} \times [-\delta, \delta]$, oriented counterclockwise (figure \ref{contour}), we have 
$$
n_Z(a,b) - n_P(a,b) = I_\delta  + 
 \frac{1}{2i\pi} \int_a^b \frac{{\cal F}'(x + i \delta)}{{\cal F}(x + i \delta)} \; dx -   \frac{1}{2i\pi} \int_a^b \frac{{\cal F}'(x - i \delta)}{{\cal F}(x - i \delta)} \; dx
$$
where
$ \displaystyle
I_\delta := \frac{1}{2i\pi} \int_{\gamma_0^\delta} \frac{{\cal F}'(\omega)}{{\cal F}(\omega)} \; d \omega \mbox{ satisfies } |I_\delta| \leq C \; \delta
$ by continuity of ${\cal F}$ along $\gamma_\delta^0$. \\[12pt]
Thanks to the property ${\cal F}(\overline{\omega}) = \overline{{\cal F}(\omega)}$, which yields ${\cal F}'(\overline{\omega}) = \overline{{\cal F}'(\omega)}$, we obtain
$$
n_Z(a,b) - n_P(a,b) = I_\delta  + \frac{1}{2\pi}  \; 
{\cal I}m \, \int_a^b \frac{{\cal F}'(x + i \delta)}{{\cal F}(x + i \delta)} \; dx
$$
Thanks to (\ref{propF}), we can define in $\C^+ \setminus {\cal P} \cup {\cal Z}$ an analytic function $\log {\cal F}(\omega)=|{\cal F}(\omega)|+i \, \mbox{Arg} \, {\cal F}(\omega)$, where  
$\mbox{Arg}\, {\cal F}(\omega) \in (0, 2\pi)$ (the usual determination of the logarithm with branch cut along $\R_+$). Thus
$$
{\cal I}m \, \int_a^b \frac{{\cal F}'(x + i \delta)}{{\cal F}(x + i \delta)} \; dx = {\cal I}m \, \int_a^b \frac{d}{dx} \log {\cal F}(x + i \delta) \; dx = \mbox{Arg} \, {\cal F}(b+i\delta) - \mbox{Arg} \, {\cal F}(a+i\delta).
$$
Therefore, since $\mbox{Arg} \, {\cal F}(\omega) \in (0, 2\pi)$ we have
$
|n_P(a,b) - n_z(a,b)| \leq 2 + C \; \delta,
$
and one gets the announced result by making $\delta \rightarrow 0$. \end{proof}
	\begin{figure}[h]
	\centerline{
	\includegraphics[width=6cm]{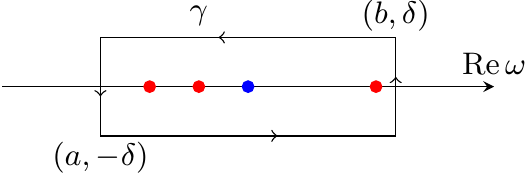}}
	\caption{The integration contour $\gamma_\delta$}
	\label{contour}
	\end{figure}
\noindent Now, let $\{p_k\}_{k=1}^{n}$ and $\{z_k\}_{k=1}^{n}$ be  the non-negative poles and zeros of the product $\varepsilon(\omega)\,\mu(\omega)$, ranked by increasing values. The non-zero poles or zeros are repeated with their multiplicity (1 or 2, see corollary 3.19). However, if $0$ is a pole of multiplicity $2m$, it is only counted $m$ times:
$$
	0\leq p_1\leq p_2\leq \ldots\leq p_n, \quad
	0\leq z_1\leq z_2\leq \ldots \leq z_n.
$$	
\begin{lema} \label{interlacing} The numbers $\{p_k\}_{k=1}^{n}$ and $\{z_k\}_{k=1}^{n}$ satisfy
	\begin{align}
	\label{eq:vicinity_zeros}
	0\leq p_1<z_1<p_3<z_3<\ldots, \qquad 0\leq p_2<z_2<p_4<z_4<\ldots.
	\end{align}
\end{lema}
\begin{proof}
 Proving (\ref{eq:vicinity_zeros}) amounts to showing that $p_m < z_m < p_{m+2}$.
	We shall prove the first inequality only, by contradiction arguments. The proof of the second one is similar.
	Let us consider the following cases:
	\begin{enumerate}
		\item $\omega=0$ is a zero of ${\cal F}(\omega)$, with multiplicity $2$ according to lemma \ref{ND-Lossless}. 		
Let us now assume that $z_m<p_m$ for some $m\geq 1$. Then, for $\eta > 0$ small enough, the interval $(-\eta - z_m ,\, z_m + \eta)$ contains exactly  $2m+2$ zeros (the $\pm p_\ell, 1 \leq \ell \leq m$ plus $0$ counted two times) and $m-2$  poles of ${\cal F}$, so that $|n_P(-\eta-z_m, z_m +\eta) - n_Z (-\eta- z_m, z_m +\eta)|=4$ which  contradicts (\ref{inegnpnz}).  
	\begin{figure}[h]
	\centerline{
\includegraphics[width=10cm]{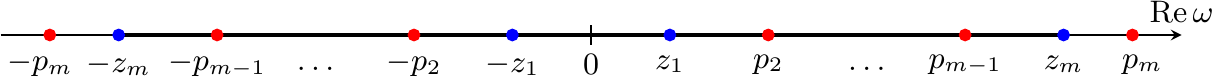}}
	\caption{The poles and zeros of ${\cal F}$}
	\label{Fig5}
	\end{figure}
	\item $\omega=0$ is neither a pole nor a zero of ${\cal F}(\omega)$. In this case $\varepsilon(\omega)\mu(\omega)$ has a pole of multiplicity 2 in $\omega=0$ so that $0=p_1<z_1$, and $p_2\neq 0$. Assume again that, for some $m\geq 1$, $z_m<p_m$. For for $\eta > 0$ small enough,the interval $(-z_m -\eta,\, z_m + \eta)$ contains exactly  $2m$ zeros and $2m-4$  poles of ${\cal F}$, since $\pm p_1$ must not be counted in the poles of  ${\cal F}(\omega)$. This means that $|n_P(-\eta-z_m, z_m +\eta) - n_Z (-\eta- z_m, z_m +\eta)|=4$ which  contradicts (\ref{inegnpnz}). 
		\item $\omega=0$ is a pole of ${\cal F}(\omega)$, with multiplicity $2$ according to . Thus, $\varepsilon\, \mu$ has a pole of multiplicity 4 in $\omega=0$ that is to say $p_1=p_2=0$. Assume again that, for some $m\geq 1$, $z_m<p_m$. For for $\eta > 0$ small enough,the interval $(-z_m -\eta,\, z_m + \eta)$ contains exactly  $2m$ zeros and $2m-4$  poles of ${\cal F}$, $0$ and $\{ \pm p_3, \cdots, \pm \, p_m \}$. One concludes as in the previous cases.
\end{enumerate}
\end{proof}
\noindent {\bf Proof of theorem \ref{thm.equivpassive}.}
Without loss of generality let us assume that the degree of the numerator of $\varepsilon(\omega)\,\mu(\omega)$ is $2(2m+1)$. Let us define
\begin{align*}
\varepsilon^*(\omega)&:=\frac{(\omega^2-z_2^2)(\omega^2-z_4^2)\ldots (\omega^2-z_{2m}^2)}{(\omega^2-p_2^2)(\omega^2-p_4^2)\ldots (\omega^2-p_{2m}^2)},\\
\mu^*(\omega)&:=\frac{(\omega^2-z_1^2)(\omega^2-z_3^2)\ldots (\omega^2-z_{2m+1}^2)}{(\omega^2-p_1^2)(\omega^2-p_3^2)\ldots (\omega^2-p_{2m+1}^2)}.
\end{align*}
which satisfies $\varepsilon^*(\omega) \mu^*(\omega) = \varepsilon(\omega) \mu(\omega)$. Since $p_{2\ell}$ is a simple pole (lemma \ref{interlacing})
$$
({\cal R}es \, \varepsilon^*, p_{2\ell}) = \frac{A_\ell \, B_\ell}{2 \, p_{2\ell}}, \quad A_\ell :=  (-1)^{m} \prod_{k=1}^m{(z_{2k}^2 - p_{2\ell}^2)}, \quad B_\ell := (-1)^{m-1} \prod_{k\neq \ell}{(p_{2k}^2 -p_{2\ell}^2)}^{-1}.
$$
The sign of $B_\ell$ is $(-1)^{\ell-1}(-1)^{m}$. According to (\ref{eq:vicinity_zeros}), $z_{2k}^2 - p_{2\ell}^2 > 0$ for $k < \ell$ and $z_{2k}^2 - p_{2\ell}^2 > 0$ for $k \geq \ell$: thus, the sign of $B_\ell$ is $(-1)^{\ell}(-1)^{m-1}$. Thus $({\cal R}es f, p_{2\ell + 1}) < 0$ and we can write
$$\varepsilon^*(\omega)=1+\sum\limits_{\ell=0}^{m}\frac{a_{\ell}}{\omega_{\ell}^2 - \omega^2}, \quad \mbox{with } a_{\ell}<0. $$
As a similar reasoning applies to $\mu$, one concludes easily.
\section{The Nevanlinna representation theorem for Herglotz functions} \label{Nevanlinna}
Here we prove lemma \ref{thm.Nevanlinna}, based on a similar representation result
for functions analytic inside the open unit disk $D$ and the use of the M\"obius transform, a conformal 
mapping which maps the unit disk $D$ onto the complex half-space $\C^+$. The proof of 
lemma \ref{thm.Nevanlinna} relies on the Poisson's formula:
\begin{lema}\label{lem.reppoisson}{(Poisson representation)}
Let the function $f$ be analytic in the unit disk $D$. Then, for any $R < 1, f$ admits the following representation:
$$
\forall \; z \in D(0,R) := \{ z \; / \; |z| < R \}, \quad f(z)=i\, {\cal I}m \, f(0)+\frac{1}{2\pi}\int_{0}^{2\pi} \frac{R\, e^{i \theta} +z}{R\, e^{i \theta} -z}\;{\cal R}e  f( R\, e^{i \theta}) \,\md \theta.
$$
\end{lema}
\begin{proof}
As $f$ is analytic in $D$, we can write $ \dsp f(z)=\sum_{n=0}^{\infty} a_n z^n$ for $|z| < 1$.  In particular,
$$
{\cal R}e f(R \, e^{i \theta})=\frac{1}{2} \sum_{=0}^{\infty}  \{a_n \, e^{in\theta} + \overline{a_n} \, e^{-in\theta} \} \; R^n =  {\cal R}e \, a_0 + \frac{1}{2}  \sum_{n=1}^{+\infty}  a_n \; R^n \; e^{in\theta} + \frac{1}{2} \sum_{n=1}^{+\infty} \overline{a_n} \; R^n \; e^{-in\theta}.
$$
Since the family $\{ e^{in\theta}, n \in \Z \}$ is orthogonal in $L^2(0,2\pi)$, we deduce, setting $f_R(\theta) := {\cal R}e f(R \, e^{i \theta})$ 
$$
{\cal R}e \,a_0=\frac{1}{2\pi}\int_{0}^{2 \pi} f_R(\theta) \; \md \theta \quad \mbox{ and } \quad a_n=\frac{1}{\pi} \int_{0}^{2 \pi}\, f_R(\theta)\; \frac{e^{-i n \theta}}{R^n} \; \md \theta, \, \mbox{ for } n\geq 1.
$$
In particular, since $a_0 = f(0)$,  we get:  ~(i)  $\dsp \quad a_0 = {\cal R}e \,a_0 + i \, {\cal I}m \, a_0 = \frac{1}{2\pi}\int_{0}^{2 \pi} f_R(\theta) \; \md \theta + i \, {\cal I}m \, f(0) $.\\[12pt]
On the other hand, $\dsp \sum_{n=1}^{\infty} a_n z^n = \sum_{n=1}^{\infty} \frac{1}{\pi}\int_{0}^{2 \pi} f_R(\theta)  \,  \left(\frac{z}{R} \, e^{-i  \theta}  \right)^n\md \theta 
$. Moreover, for $|z| < R$, we have
$$
\sum_{n=1}^{\infty} \left(\frac{z}{R} \, e^{-i  \theta}  \right)^n = \Big(1 - \frac{z}{R} \, e^{-i  \theta}\Big)^{-1} -1 = \frac{z}{R \, e^{i \theta} -z}.
$$
Thus, after permutation series-integral, one gets:  ~(ii)  $\dsp \quad \sum_{n=1}^{\infty} a_n z^n = \frac{1}{\pi}\int_{0}^{2 \pi} f_R(\theta)  \,\frac{z}{R \, e^{i \theta} -z} \md \theta. $\\[12pt]
The final result is obtained by summing (i) and (ii), since $\dsp 1 + 2 \; \frac{z}{R \, e^{i \theta} -z} = \frac{R \, e^{i \theta} + z}{R \, e^{i \theta} -z} $.
\end{proof}
\begin{lema}\label{lem.repinitdisk}
A  function $f$ is analytic in the open unit disk and has a positive real part if and only if it admits the following representation:
\begin{equation}\label{eq.represcercle}
\forall \; z \in D, \quad f(z)=i \beta+\int_{0}^{2\pi} \frac{ e^{i \theta} +z}{ e^{i \theta} -z}\, \md \sigma (\theta) ,
\end{equation}
\noindent with $\beta ={\cal I}m \, f(0)\in \R$ and $\sigma$ a positive finite regular Borel measure of the interval $[0,2\pi]$ with finite mass.
\end{lema}
\begin{proof}
It is a clear that functions of the form (\ref{eq.represcercle}) are analytic in $D$ and have a positive real part. Moreover, one checks easily that $\beta={\cal I}m \, f(0)$.
\\[12pt]
Reciprocally, suppose now that $f$ is analytic in the open unit disk $D$. Then by the Poisson representation of $f$ (see lemma \ref{lem.reppoisson}), for all $0<R<1$, one has
\begin{equation}\label{eq.represcercleR}
\mbox{ for }|z|<R, f(z)=i \; {\cal I}m \, f(0)+\int_{0}^{2\pi} \psi_{z,R}\; \md\sigma_R(\theta) ,  \quad \psi_{z,R}(\theta) = \frac{R\, e^{i \theta} +z}{R\, e^{i \theta} -z}.
\end{equation}
where $\sigma_R$ is the absolutely continuous measure on $[0, 2 \pi]$ (with respect to Lebesgue's measure) defined by
$$
\md\sigma_R(\theta) = \frac{1}{2\pi} \, f_R(\theta) \, \md\theta, \quad  f_R(\theta) := {\cal R}e \,  f( R\, e^{i \theta}).
$$
Note that, due to the assumption that $f$ has a positive real part inside $D$, $\sigma_R$ is a positive measure. The reader will observe that the formula (\ref{eq.represcercleR}) of very close to (\ref{eq.represcercle}): we only would like to push $R$ to 1. \\[12pt]
To do so, the first step consists in identifying the limit, if any, of the measure $\sigma_R$. To do so, we identify $\sigma_R$ to a continuous linear 
form $\sigma_R^* \in E'$ on the Banach space $C^0[0, 2\pi]$
equipped with the sup-norm $\| \cdot \|_\infty:$
\begin{equation} \label{defsigmaR} \langle \sigma_R^*, \varphi \rangle := \int_0^{2\pi} \varphi(\theta) \; \md \sigma_R (\theta) \equiv \frac{1}{2\pi} \int_0^{2\pi} \varphi(\theta) \, f_R(\theta) \; \md \theta. \end{equation}
We first observe that $\sigma_R^*$ is bounded in $E'$ since 
$$
\|\sigma_R^*\|_{E'} \leq \frac{1}{2\pi} \, \|f_R\|_{L^1(0,2\pi)} = \frac{1}{2\pi} \,\int_0^{2\pi} f_R(\theta) \; d \theta \quad \mbox{(positivity of $f_R$)}
$$ 
Modulo identification  $\C \equiv\R^2$, by analyticity of $f$, ${\cal R}f (x,y) :={\cal R}e \,  f( x+iy)$
is harmonic in $D$ and we observe that $f_R = {\cal R}f |_{\partial D(0,R)}$. Then, by the mean value theorem for harmonic functions, we have
\begin{equation} \label{boundsigma}
\frac{1}{2\pi} \int_0^{2\pi} f_R(\theta) \; d \theta = {\cal R}ef(0)  \quad \mbox{which proves that } \|\sigma_R^*\|_{E'} \leq   {\cal R}ef(0).
\end{equation}
Therefore, by Banach-Alaoglu's theorem, we can find an increasing sequence $R_n \rightarrow 1$ such  that $\sigma_{R_n}$ converges to
some $\sigma^* \in E'$ for the weak-* topology of $E$, i.e. 
$$
\forall \; \varphi \in C^0[0,2\pi], \quad \lim_{n \rightarrow + \infty} \langle \sigma_{R_n}^*, \varphi \rangle =  \langle \sigma^*, \varphi \rangle
$$
Since $\langle \sigma_R^*, \varphi^+ \rangle \geq 0, \forall \; \varphi \in E^+ := \{ \varphi^+ \in C^0[0, 2\pi] \; / \; \varphi^+ \geq 0\}$, we also have $\langle \sigma^*, \varphi^+ \rangle \geq 0, \forall \; \varphi \in E^+ $. Then, by Riesz-Markov theorem, there exist a positive regular Borel measure $\sigma$ on $[0, 2\pi]$ such that
$$
\langle \sigma^*, \varphi \rangle = \int_0^{2\pi} \varphi(\theta) \; \md \sigma(\theta).
$$
Finally, the inequality $\|\sigma^*_R\|_{E'} \leq   {\cal R}ef(0)$ proves that $\|\sigma^*\|_{E'}\leq  {\cal R}ef(0)$. For $|z| < 1$ and $n$ large enough, $z \in D(0, R_n)$. Thus by (\ref{eq.represcercleR}) and (\ref{defsigmaR}), we can write
\begin{equation} \label{formulef}
f(z)=i \; {\cal I}m \, f(0)+ \langle \sigma_{R_n}^*, \psi_{z,R_n} \rangle =
i \; {\cal I}m \, f(0)+  \langle \sigma_{R_n}^*, \psi_{z,1} \rangle + \langle \sigma_{R_n}^*, \psi_{z,R_n} - \psi_{z,1} \rangle
\end{equation}
Noticing that, thanks to (\ref{boundsigma}),
$
\big| \, \langle \sigma_{R_n}^*, \psi_{z,R_n} - \psi_{z,1} \rangle \big| \leq  {\cal R}ef(0) \, \|\psi_{z,R_n} - \psi_{z,1}\|_\infty,
$
by uniform convergence of $\psi_{z,R_n}$ to $\psi_{z,1}$, we can pass to the limit in (\ref{formulef}) to obtain the announced result.\\[12pt]
{\bf  The proof of lemma \ref{thm.Nevanlinna}.}
It is clear that all functions of the form \eqref{eq.defhergl} are Herglotz functions.
For the reciprocal statement, we introduce the M\"obius transform $M(z)$ as the function defined on $\C \setminus \{1\}$ by
\begin{equation} \label{defMobius}
	M(z) = i \, \frac{1+z}{1-z}
\end{equation}
which is a bijection from $\C \setminus \{1\}$ on to $\C \setminus \{i\}$
\begin{equation} \label{defMobiusinverse}
	\omega = M(z), \quad  z \in \C \setminus \{1\}\quad \Longleftrightarrow \quad  z=M^{-1}(\omega) :=   \frac{1- i \omega}{1+ i \omega}, \quad \omega \in \C \setminus \{i\}
\end{equation}
Moreover, the important property of $M$ for our purpose is that it realizes a bijection from the unit disk $D$ onto the complex half-space $\C^+$ (and from $\partial D \setminus \{-1\}$ on to $\R = \partial \C^+$)
$$
z \in D \quad \Longrightarrow \quad \omega = M(z) \in \C^+, \quad \mbox{and} \quad \omega \in \C^+ \quad \Longrightarrow \quad  z=M^{-1}(\omega) \in D.
$$ 
As a consequence given $f : \C^+ \rightarrow \C$, one can construct $\widetilde{f} : D \rightarrow \C$ as
\begin{equation} \label{transMobius}
	\widetilde{f}(z) := - i \, f\big( M(z)\big) \quad \Longleftrightarrow \quad f(\omega) = i \, \widetilde{f}\big(M^{-1}(\omega)\big).
\end{equation}
Indeed, $f$ is analytic in $C^+$ if and only if $\widetilde{f}$ is analytic in $D$.  Furthermore, if $f$ is Herglotz, $\widetilde{f}$ satisfies the assumptions of lemma and can thus be written in the form (\ref{eq.represcercle}). As a consequence, substituting $z = M^{-1}(\omega)$ in (\ref{eq.represcercle}), using the fact that $f(\omega) = i \, \widetilde{f}\big(M^{-1}(\omega)\big)$ and the relation $- {\cal I}m \widetilde{f}(0) = {\cal R}e f\big(M(0)\big) = {\cal R}e f(i)$, we deduce that
$$
f(\omega)= {\cal R}e f(i)  +i \int_{[0,2\pi]} \frac{ e^{i \theta} +M^{-1}(\omega)}{ e^{i \theta} -M^{-1}(\omega)}\, \md \sigma (\theta), \quad \omega \in \C^{+}.
$$
or equivalently, setting $a=\sigma(\{2\pi \})+\sigma(\{0 \})\geq 0$ and using $\dsp i \frac{ e^{i \theta} +M^{-1}(\omega)}{ e^{i \theta} -M^{-1}(\omega)} = \omega$ for $\theta = 0$ or $2\pi$, 
$$
f(\omega)= {\cal R}e f(i) + a \, \omega + i \int_{]0,2\pi[} \frac{ e^{i \theta} +M^{-1}(\omega)}{ e^{i \theta} -M^{-1}(\omega)}\, \md \sigma (\theta), \quad \omega \in \C^{+}.
$$
To obtain the final result, the idea is to use the change of variable $\xi = \psi(\theta) := - \operatorname{cotan}(\theta/2)$ in the last integral, using the fact that 
$\psi$ is a bijection from $]0, 2 \pi[$ onto $\R$. This leads to introduce the positive regular Borel measure $\rho$ on $\R$ defined as the pushward measure of $\sigma$ through the continuous bijection $\psi$, namely
$$
\rho(B)=\sigma\big(\psi^{-1}(B)\big), \quad \forall B \in {\cal B}(\R) := \{ \mbox{ Borel subsets of } \R \}.
$$
Therefore,
$ \dsp 
f(\omega)= {\cal R}e f(0) + a \, \omega + i \int_{\R} \; \frac{ e^{i \psi^{-1}(\xi)} +M^{-1}(\omega)}{ e^{i \psi^{-1}(\xi)} -M^{-1}(\omega)}\; \md \rho (\xi).
$
Moreover tedious computations give
$$
\frac{ e^{i \psi^{-1}(\xi)} +M^{-1}(\omega)}{ e^{i \psi^{-1}(\xi)} -M^{-1}(\omega)} = \Big( \frac{1 + \xi \, \omega}{\xi - \omega}\Big) \equiv  \big(1+\xi^2 \big) \; \left(\frac{1}{\xi-\omega}- \frac{\xi}{1+\xi^2}\right).
$$
Then, one obtains the expression (\ref{eq.defhergl}) by introducing the positive regular Borel measure $\nu$ defined by $\md \nu(\xi) = (1+\xi^2) \; \md \rho(\xi)$. Note that the condition (\ref{finiteness}) follows from the fact that $\rho$ has a finite mass. \\[12pt]
Obtaining the value of $\beta$ in (\ref{alphabeta}) is immediate while  the formula for $\alpha$ is a straightforward application of the Lebesgue's theorem. It remains to prove \eqref{measure}. By the Herglotz representation formula \eqref{eq.defhergl}, one has: 
$$
\eta \;{\cal I}m \, f(a+ i \eta)=\alpha \, \eta^2 +\int_{\R} \frac{\md \nu (\xi)}{\eta^{-2} (a-\xi)^2+1} \, .
$$
As $\eta \to 0$, the integrand tends to ${\bf 1}_a$, the indicator function of the set $\{a \}$.  Moreover this function is bounded by  $\psi(\xi)=1 / \big[(a-\xi)^2+1\big]$ for $0<\eta\leq 1$
which is $\nu$-integrable.
Thus, applying the Lebesgue's  theorem leads to  \eqref{measure}(i). Next we prove \eqref{measure}(ii). For simplicity, we restrict ourselves to the case where (\ref{finitenessbis}) holds, the general case being treated by an approximation process (see \cite{Don-00}).
By (\ref{eq.defhergl}) again, we have
\begin{equation}\label{eq.limImag}
\int_{a}^{b} {\cal I}m \,f(x+i \eta) \; \md x=\int_{a}^{b}\int_{\R}\frac{\eta }{(x-\xi)^2+\eta^2}\, \md \nu(\xi) \md x.
\end{equation}
As the integrand in (\ref{eq.limImag}) is positive, we can use Fubini's theorem to obtain
\begin{equation*}
\int_{a}^{b} {\cal I}m \,f(x+i \eta) \; \md x=\int_{\R} \chi_{a,b}^\eta(\xi) \, \md \nu(\xi), \quad \chi_{a,b}^\eta(\xi) : = \arctan \left({\frac{b-\xi}{\eta}}\right)- \arctan{\left(\frac{a-\xi}{\eta}\right)} .
\end{equation*}
When $\eta \rightarrow 0$, $\chi_{a,b}^\eta(\xi) \longrightarrow  \chi_{a,b}(\xi) : = 0 \; \mbox{if } \xi \in (a,b), \; \frac{\pi}{2} \; \mbox{if } \xi \in \{a,b\}, \; \pi \mbox{ otherwise.}$
On the other hand
$$
0 \leq \chi_{a,b}^\eta(\xi)\leq \pi \quad \mbox{and} \quad \eta  \mapsto \chi_{a,b}^\eta(\xi) \mbox{ is increasing if } \xi \in [a,b], \mbox{ decreasing if not},
$$
so that $0 \leq \chi_{a,b}^\eta(\xi)\leq 0$ with $ \chi_{a,b}^+(\xi) : = \pi \mbox{ if } \xi \in [a,b]$ and $ \chi_{a,b}^+(\xi) : = \chi_{a,b}^1(\xi) \mbox{ if not}$.\\[12pt] Since one easily check that $\chi_{a,b}^1(\xi) \sim |\xi|^{-1}$ when $\xi \rightarrow \pm \, \infty$, it results from (\ref{finitenessbis}) that  $\chi_{a,b}^+$ is $\nu$-integrable.
Thus, by Lebesgue's theorem, we can pass to the limit in \eqref{eq.limImag} when $\eta \to 0$ to obtain \eqref{measure}(ii).
\end{proof}

	\end{document}